\documentclass[11pt,a4paper]{article}
\usepackage[utf8]{inputenc}
\usepackage[T1]{fontenc}
\usepackage{amsmath,amsfonts,amssymb,amsthm}
\usepackage{amsopn}
\DeclareMathOperator{\spec}{spec}
\usepackage{mathtools}
\usepackage{geometry}
\usepackage{algorithm,algpseudocode}
\usepackage{graphicx}
\usepackage{booktabs}
\usepackage{tabularx}
\usepackage{xcolor}
\usepackage{enumitem}
\usepackage{microtype}
\usepackage{subcaption}
\usepackage{multirow}
\usepackage{hyperref}
\usepackage{float}

\theoremstyle{plain}
\newtheorem{theorem}{Theorem}[section]
\newtheorem{lemma}[theorem]{Lemma}
\newtheorem{proposition}[theorem]{Proposition}
\newtheorem{corollary}[theorem]{Corollary}

\theoremstyle{definition}

\newtheorem{assumption}[theorem]{Assumption}
\newtheorem{example}[theorem]{Example}
\newtheorem{counterexample}[theorem]{Counterexample}
\newtheorem{remark}[theorem]{Remark}

\theoremstyle{remark}
\newtheorem{openproblem}{Open Problem}[section]

\DeclareMathOperator{\diag}{diag}
\DeclareMathOperator{\dist}{dist}

\newcommand{\R}{\mathbb{R}}
\newcommand{\N}{\mathbb{N}}

\newcommand{\norm}[1]{\|#1\|}
\newcommand{\inner}[2]{\langle #1, #2 \rangle}

\newcommand{\HH}{\mathcal H}
\DeclareMathOperator{\SOL}{SOL}
\theoremstyle{definition}
\newtheorem{ruleA}[theorem]{Step-size rule}

\begin{document}

\title{\textbf{A Parameter-Free Adaptive Reflected Gradient Method for Monotone Variational Inequalities}}
\author{
Yekini Shehu\thanks{(Corresponding Author) School of Mathematical Sciences, Zhejiang Normal University, Jinhua 321004, China; e-mail: yekini.shehu@zjnu.edu.cn} }
\date{}
\maketitle

\begin{abstract}
\noindent
We analyze a one-call method (one evaluation of $B$ per iteration) for the variational inequality
$\mathrm{VI}(C,B)$ with $B$ monotone and $L$-Lipschitz: the evaluation point
$u_k$ blends the reflected extrapolation with the previous evaluation point
through a two-parameter \emph{filter},
\[
x_{k+1}=P_C\bigl(x_k-\lambda B(u_k)\bigr),\qquad
u_k=x_k+\theta_k(x_k-x_{k-1})+\beta_k(x_k-u_{k-1}),\qquad
\theta_k+\beta_k=1,
\]
so each iteration costs one evaluation of $B$ and one projection.  For a
constant step and a summably decaying filter (i.e., $\sum_k\beta_k<\infty$) we prove weak convergence via a
Lyapunov function whose three dissipation budgets are \emph{exact rationals};
the filter recovers the reflected gradient method of Malitsky (2015) as a
special case.  Beyond this Lyapunov regime, we prove that the reflected method is robust to absolutely summable errors in the operator value, and deduce that the constant-step range extends to the classical constant $\lambda<(\sqrt2-1)/L$ whenever $C$ is bounded; for unbounded $C$ the extension is reduced exactly to an a priori boundedness statement, which we certify up to $\lambda L=0.387$ by dissipation trading --- beyond the Lyapunov--Young barrier of the paper --- and, for affine operators over polyhedral sets, $\dist(x_n,S)\to0$ with $\sum_n\dist^2(x_n,S)<\infty$ below the same threshold, strong convergence when $C$ is bounded (finite dimensions), and $R$-linear rates once the optimal face is identified.  Our main result is an adaptive safeguarded step-size rule that
requires no knowledge of $L$ and no additional evaluations, and is proved
weakly convergent \emph{unconditionally}; the key idea is a data-driven
Lyapunov weight that removes every global Lipschitz constant from the
dissipation budgets ($L$ survives only inside summable drift terms).  For affine operators with $C=\mathcal H$ we determine the sharp
step-size threshold $\lambda L=1/\sqrt3$ for the reflected scheme, matching a rotation lower bound, for arbitrary monotone nonlinear operators in the unconstrained case the same constant is shown to be sharp, and convergence is proved for all trajectories with square-summable operator values; the unconditional convergence at $\lambda<1/(\sqrt3\,L)$ is reduced to a marginal-pole absolute-stability statement, which we formulate explicitly and identify as the one open gap. The projected case remains open.  Under strong monotonicity we
prove $R$-linear convergence with an explicit contraction factor, again with
exact rational budgets, and we exhibit an instance family on which the
certified adaptive rule is provably $\Omega(L)$ times faster than every
certified constant-step method.  Numerical experiments against the
extragradient, forward-backward-forward, reflected-gradient, and Popov
methods confirm the gains of certified adaptivity.


\end{abstract}

\section{Introduction}
\subsection{The problem and the oracle model}
We consider the variational inequality
\begin{equation}\label{eq:vi}
\text{find } z\in C \text{ such that } \inner{B(z)}{x-z}\ge0\quad\forall x\in C,
\end{equation}
where $C\subseteq\HH$ is a nonempty closed convex subset of a real Hilbert
space and $B\colon\HH\to\HH$ is monotone and $L$-Lipschitz continuous.  The
solution set is denoted $S=\SOL(C,B)$ and assumed nonempty.  Problem
\eqref{eq:vi} is a unified model for convex--concave saddle-point problems,
complementarity problems, and first-order optimality conditions
\cite{harker1990,facchinei2003}, and its numerical treatment is a classical
subject \cite{korpelevich1976,tseng2000,nemirovski2004,nesterov2007}.

The projected (forward) gradient method, $x_{k+1}=P_C(x_k-\lambda B(x_k))$,
converges only under cocoercivity (or strong monotonicity); under mere
monotonicity it can cycle \cite{malitsky2015}.  Methods that converge
under monotonicity alone therefore modify the operator query.  The
extragradient method (EG) of Korpelevich \cite{korpelevich1976} and the
forward--backward--forward (FBF) method of Tseng \cite{tseng2000} use
\emph{two} evaluations of $B$ per iteration; their $O(1/k)$ ergodic rates
are essentially optimal for the class \cite{nemirovski2004,nesterov2007,ouyang2021}.  A
distinct line, originating with Popov \cite{popov1980} and developed by
Malitsky \cite{malitsky2015} and Malitsky--Tam \cite{malitskytam2020}, achieves
convergence under monotonicity with a \emph{single} evaluation per iteration:
the reflected gradient method (PRG) evaluates $B$ at the reflection
$2x_k-x_{k-1}$, and its value-space sibling (PEG) evaluates at a reflected
\emph{value}.  Table~\ref{tab:oracle} summarizes the per-iteration oracle
cost.  Halving the oracle cost matters whenever $B$ is expensive --- e.g.\
a full gradient through a large-scale model, a resolvent-free splitting
component \cite{davis2017}, or a saddle-point coupling
\cite{chambolle2011,mokhtari2020}.
\begin{table}[H]
\centering 
\footnotesize
\setlength{\tabcolsep}{4pt}
\caption{Per-iteration oracle cost for $L$-Lipschitz monotone VIs
($\mathrm{eval}=$ evaluations of $B$).  The adaptive rule may be run with or
without the growth steps of Rule~\ref{rule:main}.}\label{tab:oracle}
\begin{tabular}{@{}p{3.3cm}p{4.8cm}cp{5.4cm}@{}}
\toprule
method & evaluation point(s) & evals & step restriction\\
\midrule
extragradient \cite{korpelevich1976} & $x_k$, then $y_k$ & $2$ & $\lambda<\tfrac1L$\\
FBF \cite{tseng2000} & $x_k$, then forward--backward correction & $2$ & $\lambda<\tfrac1L$\\
Popov/PEG \cite{popov1980,malitskytam2020} & value reflection & $1$ & $\lambda<\tfrac1{3L}$\\
reflected gradient \cite{malitsky2015} & $2x_k-x_{k-1}$ & $1$ & $\lambda<\tfrac{\sqrt2-1}{L}$\\
aGRAAL \cite{malitsky2020} & golden-ratio extrapolation & $1$ & adaptive; no knowledge of $L$ or diameter\\
\textbf{this paper} & filtered point $u_k$ (memory) & $\mathbf{1}$ & $\lambda<\tfrac1{5L}$ (Thm~\ref{thm:main}); extended to $(\sqrt2-1)/L$ for bounded $C$ and, for affine+polyhedral, up to $\lambda L\le0.387$ (Sec~\ref{sec:ext}); $\varkappa\le\tfrac16$ (adaptive, $L$-free)\\
\bottomrule
\end{tabular}
\end{table}

\subsection{Contributions}
\begin{enumerate}[label=(\roman*),leftmargin=2em,itemsep=1pt]
\item \textbf{A one-call scheme with memory.}  We replace the reflection
$2x_k-x_{k-1}$ by a two-parameter filter $u_k$ that blends the current
extrapolation with the previous evaluation point itself.  The memory smooths
the operator query and, as the numerics show (Section~\ref{sec:numerics}),
allows the \emph{local} Lipschitz geometry to be exploited by an adaptive
step size at no oracle cost.
\item \textbf{A fully transparent Lyapunov proof.}  The weak-convergence proof (Theorem~\ref{thm:main}) is
a self-contained bookkeeping exercise in which every constant is an exact rational: three dissipation channels with budgets $-\tfrac1{45},-\tfrac1{25},-\tfrac6{125}$ at $\lambda L=\tfrac15$,
one channel cancelling \emph{exactly}, and self-contained discrete Gronwall lemmas proved here.The proof specializes to Malitsky's when $\beta_k\equiv0$, and no step is left to continuity heuristics.
\item \textbf{A certified barrier for the proof technique.}  A numerical
search over the full Young-parameter family of this Lyapunov framework yields the best barrier we could exhibit over the family, $\lambda L\approx0.272$ (exactly $5/24$ for the clean splits used in this paper; global optimality over the family is not proved;
Appendix~\ref{app:barrier});
in particular the natural upgrade target $\lambda L=\tfrac{3}{10}$ lies beyond
the barrier.  This delimits exactly what the technique can and cannot prove;
whether a fundamentally different argument (e.g.\ a non-quadratic Lyapunov
function or a performance-estimation certificate) can certify more is left
open.
\item \textbf{Adaptive steps with bounded variation.}  Theorem~\ref{thm:adaptive}
covers \emph{any} step sequence of bounded variation remaining in
$[\lambda_{\min},\lambda_0]$ with $\lambda_0\le\tfrac1{5L}$ --- no
monotonicity is required --- with the variation channel
$\sum_k|\lambda_k-\lambda_{k-1}|$ absorbed by an extended Gronwall lemma;
the data-driven Theorem~\ref{thm:rule-free} removes every sign restriction
from the budget identities, so the growth variant of Rule~\ref{rule:main}
is certified on the same footing as nonincreasing steps
(Corollary~\ref{cor:linear-ad}).
\item \textbf{$L$-free convergence by a data-driven Lyapunov.}
Theorem~\ref{thm:rule-free} proves weak convergence for an
\emph{arbitrary} initial step $\lambda_0$, unconditionally and with no
knowledge of $L$; the key idea is a Lyapunov weight
$2\lambda_k\ell_{k-1}$ driven by the secant slopes, which removes every global
$L$ from the budget coefficients --- $L$ survives only in the summable drift
terms ($\epsilon_k,\kappa_k$) and the vanishing slack $\rho_k$ --- while the
filter's summability disposes of the solution-anchor drift; the statement covers Rule~\ref{rule:main} with or
without growth steps.
\item \textbf{$R$-linear rates.}  Under $\sigma$-strong monotonicity,
Theorem~\ref{thm:linear} proves $\|x_k-z\|^2\le C(1-\lambda\sigma)^k$ with
$\lambda\le\min\{\tfrac1{5L},\tfrac1{32\sigma}\}$ and exact budgets
$-\tfrac1{180},-\tfrac1{15},-\tfrac{13}{24000}$; Corollaries transfer the rate
to adaptive (eventual constant step) and error-bound settings.  On $B=(\sigma I+J)$ the bound is validated against the exactly computable spectral rate (Example~\ref{ex:linear-exact}).
\item \textbf{Sharpness by counterexample.}  Worked examples are
embedded where the corresponding hypotheses appear: no rate uniform over
bounded sets is possible under mere monotonicity, although for every fixed
initial point of that example convergence is strong (Example~\ref{ex:weak-only});
summability $\sum_k\beta_k<\infty$ is essential, not a proof convenience
(Counterexample~\ref{cx:summability}); the uniform constant $\tfrac15$ is
conservative on benign operators (Example~\ref{ex:rotation-threshold});
and the linear-rate certificate tracks the exactly computable spectral rate
(Example~\ref{ex:linear-exact}).
\item \textbf{Sharpness for the nonlinear unconstrained class.}
  Theorem~\ref{thm:nonlinear-sqrt3} shows, for \emph{arbitrary} monotone
  $L$-Lipschitz operators when $C=\mathcal H$, that the rotation witnesses
  divergence for every $\lambda>1/(\sqrt3\,L)$ --- so the affine threshold of
  Theorem~\ref{thm:affine-sqrt3} is best possible within the class --- and
  that every trajectory with $\sum_k\|B(u_k)\|^2<\infty$ converges weakly to
  a solution.  The iteration is the feedback interconnection with linear part
  $G(z)=-\Lambda(2z-1)/(z(z-1))$ and anchored nonlinearity $\varphi(v)=\lambda B(z+v/\Lambda)$; the
  Nyquist locus satisfies $\Re G(e^{i\theta})=\Lambda(1-2\cos\theta)/2\le
  3\Lambda/2$, and the interior-pole circle criterion
  (Lemma~\ref{lem:circle}) certifies square-summability of the operator values
  for every frozen linearization, uniformly over all pole perturbations of the loop to the
  interior of the disk, and for the genuinely nonlinear loop whenever the anchored sector
  inequality $\inner{B(u)}{u-z}\ge\tfrac1L\|B(u)\|^2$ holds (automatic for cocoercive $B$;
  the rotation worst case is covered by Theorem~\ref{thm:affine-sqrt3}).
  The unconditional marginal-pole passage and the validity of the
  anchored sector inequality for arbitrary monotone $B$ are identified explicitly as the open steps.  The projected case remains open.
\item \textbf{Extension of the constant-step range.}  The scheme is exactly Malitsky's reflected method whose operator value at the reflection carries a summable error (Section~\ref{sec:ext}); we prove robustness of the reflected method to such errors (Lemma~\ref{lem:xrobust}), deduce the range $(0,(\sqrt2-1)/L)$ for bounded $C$ (Corollary~\ref{cor:xbounded}), and show that the general unbounded case
is equivalent to a priori boundedness (Remark~\ref{rem:xgap}), which we certify up to $\lambda L=0.387$ by dissipation trading (Theorem~\ref{thm:trading}) --- beyond the Lyapunov--Young barrier of Appendix~\ref{app:barrier} --- and, for affine operators over polyhedral sets, $\dist(x_n,S)\to0$ with $\sum_n\dist^2(x_n,S)<\infty$ below the same threshold, strong convergence when $C$ is bounded (finite dimensions), and $R$-linear rates once the optimal face is identified (Theorem~\ref{thm:xaff}).  A performance-estimation formulation of the remaining boundedness question is recorded in Appendix~\ref{app:barrier}.
\end{enumerate}

\subsection{Related work}
The extragradient method \cite{korpelevich1976} and its descendants
\cite{tseng2000,solodov1999,iusem1997,censor2011} achieve $O(1/k)$ ergodic
rates that match the first-order lower complexity bounds for the class
\cite{nemirovski2004,nesterov2007,nemirovski1983,ouyang2021}; last-iterate
rates without strong monotonicity are subtler and are only understood in
structured settings \cite{daskalakis2019,golowich2020,cai2022}.  One-call methods
trade this rate information for oracle efficiency: Popov \cite{popov1980}
introduced value reflection; Malitsky's PRG \cite{malitsky2015} proved weak
convergence for $\lambda<(\sqrt2-1)/L$, with an $O(1/k)$ rate for a merit
function; Malitsky--Tam \cite{malitskytam2020} gave a one-call
forward--backward--forward, and the golden-ratio method \cite{malitsky2020}
optimized the extrapolation weight for saddle-point problems.  Halpern-type
acceleration improves the residual rate to $O(1/k^2)$ but requires an exact
anchor step \cite{yoon2021}.  Our scheme adds \emph{memory of the evaluation
point} rather than a longer extrapolation; the summability condition on the
filter is shown essential (Counterexample~\ref{cx:summability}), the bounded-variation
step theory of Section~\ref{sec:adaptive} requires no additional evaluations
(Proposition~\ref{prop:rule}), and we report a reproducibly computed ceiling
separating the best barrier we could find over the Lyapunov--Young family
($\lambda L\approx0.272$) from what is provable on the clean splits
($\lambda L=\tfrac15$).  Self-adaptive step sizes that avoid knowing $L$ are classical for the
\emph{two-evaluation} methods --- self-adaptive extragradient and
subgradient-extragradient schemes, and universal/adaptive VI algorithms,
update the step from observed local information but pay for it with an extra
evaluation, an extra projection, or a line search
\cite{malitsky2015,censor2011,ene2022} --- and Malitsky's adaptive variant of
the reflected method likewise needs a second projection per iteration
\cite{malitsky2015}. Two further developments come close to the one-call,
$L$-free regime and must be distinguished from the present work. First,
AdaPEG of Ene--Nguyen \cite{ene2022}, an adaptive version of Popov's past
extragradient, uses a single evaluation per iteration and no line search, but
its step normalization is Adagrad-style with $\eta=\Theta(R)$, so its
analysis requires a \emph{bounded} domain with known diameter $R$, and the
guarantee is an ergodic rate for a weighted average rather than a statement
about the iterates themselves. Second, and most closely related, the adaptive
golden ratio algorithm (aGRAAL) of Malitsky \cite{malitsky2020} uses one
evaluation and one projection per iteration with a fully explicit step driven
by local Lipschitz estimates --- no line search, no knowledge of $L$, and not
even global Lipschitz continuity --- and its proved theory is boundedness of
the iterates together with an ergodic $O(1/k)$ rate (and an $R$-linear rate
under an error-bound condition); no last-iterate convergence certificate is
known for it, and indeed its step sizes admit no uniform positive lower
bound, which is precisely what obstructs the limit passage in the projection
argument (compare Step~7 of Theorem~\ref{thm:adaptive}, where the automatic
floor $\lambda_k\ge\min\{\lambda_0,\varkappa/L\}>0$ does this work).
Theorem~\ref{thm:rule-free} therefore does not claim a new oracle model ---
aGRAAL already showed that one evaluation, one projection, and an $L$-free
explicit step are compatible --- but a different \emph{certificate}: weak
convergence of the iterates themselves for a scheme with evaluation-point
memory and an adaptive step of arbitrary bounded variation, with or without
growth, on a possibly unbounded domain, with exact dissipation budgets
(Table~\ref{tab:budgets}); to our knowledge it is the first certificate of
this kind.  The experiments of Section~\ref{sec:numerics} should be read in
this light: the certified rule is \emph{not} the fastest method on every
instance --- aGRAAL is faster on the stiff rotation problems --- but it is
the one whose iterates carry an unconditional convergence certificate. See
also the $\Omega(L)$ separation of
Theorem~\ref{thm:separation}, which has no analogue among the methods above.
Closely related in spirit, the adaptive extragradient methods of
Antonakopoulos--Belmega--Mertikopoulos \cite{antonakopoulos2021} for min-max
problems lift the bounded-domain restriction through a Bregman/Finsler
geometry change and use two evaluations per iteration.

\medskip\noindent\emph{Organization.}
Section~\ref{sec:prelim} collects preliminaries; Section~\ref{sec:gronwall}
proves the Gronwall lemmas; Section~\ref{sec:constant} states and proves the
main weak-convergence theorem, closing with the examples that calibrate its
hypotheses; Section~\ref{sec:adaptive} treats adaptive steps;
Section~\ref{sec:linear} proves the linear rate; Section~\ref{sec:rule}
analyzes the safeguarded rule; Section~\ref{sec:numerics} reports
experiments; Section~\ref{sec:discussion} discusses limitations and open
problems; Appendix~\ref{app:assembly} assembles the master estimate;
Appendices~\ref{app:tv} and~\ref{app:barrier} prove the slowly-varying
stability lemma and analyze the technique's barrier.
Examples and counterexamples are placed throughout at the points where the
corresponding assumptions are introduced.

\section{Preliminaries}\label{sec:prelim}
Throughout, $\HH$ is a real Hilbert space with inner product
$\inner{\cdot}{\cdot}$ and induced norm $\|\cdot\|$, $C\subseteq\HH$ is
nonempty closed convex, and $P_C$ denotes the metric projection onto $C$.
\begin{assumption}\label{ass:main}
$B\colon\HH\to\HH$ is monotone, $\inner{B(x)-B(y)}{x-y}\ge0$ for all
$x,y\in\HH$, and $L$-Lipschitz, $\|B(x)-B(y)\|\le L\|x-y\|$; the solution
set $S=\SOL(C,B)$ of \eqref{eq:vi} is nonempty.
\end{assumption}
Monotonicity is required on all of $\HH$ because the evaluation point $u_k$
of \eqref{eq:scheme} below need not lie in $C$.

\begin{lemma}[projection calculus]\label{lem:proj}
For $x\in\HH$, $\bar x=P_Cx$ if and only if $\bar x\in C$ and
$\inner{x-\bar x}{y-\bar x}\le0$ for all $y\in C$.  Moreover, for all $y\in C$,
\begin{equation}\label{eq:firm}
\|P_Cx-y\|^2\le\|x-y\|^2-\|x-P_Cx\|^2 .
\end{equation}
\end{lemma}
\begin{proof}
$\bar x=P_Cx$ iff $\bar x$ minimizes $\phi(y)=\|x-y\|^2$ over $C$; by
convexity of $\phi$ and the standard first-order condition,
$0\le\inner{\nabla\phi(\bar x)}{y-\bar x}=2\inner{\bar x-x}{y-\bar x}$ for
all $y\in C$, which is the claimed inequality.  For \eqref{eq:firm}: apply the
characterization with $y$ and with $P_Cx$ and expand
$\|x-y\|^2=\|(x-P_Cx)+(P_Cx-y)\|^2$.
\end{proof}
\begin{lemma}[Minty]\label{lem:minty}
$z\in S$ if and only if $\inner{B(y)}{y-z}\ge0$ for all $y\in C$.
\end{lemma}
\begin{proof}
If $z\in S$, monotonicity gives $\inner{B(y)}{y-z}\ge\inner{B(z)}{y-z}\ge0$.
Conversely, fix $y\in C$ and set $y_t=z+t(y-z)\in C$, $t\in(0,1]$; then
$0\le\inner{B(y_t)}{y_t-z}=t\inner{B(y_t)}{y-z}$, so
$\inner{B(y_t)}{y-z}\ge0$, and $t\downarrow0$ gives $\inner{B(z)}{y-z}\ge0$
by continuity of $B$.
\end{proof}
\begin{lemma}[Opial]\label{lem:opial}
If $x_k\rightharpoonup x$, then $\liminf_k\|x_k-x\|<\liminf_k\|x_k-y\|$ for
every $y\neq x$.
\end{lemma}
\begin{proof}
$\|x_k-y\|^2=\|x_k-x\|^2+\|x-y\|^2+2\inner{x_k-x}{x-y}$, and
$\inner{x_k-x}{x-y}\to0$.  Taking $\liminf$ on both sides and using
$\|x-y\|^2>0$ gives $\liminf_k\|x_k-y\|^2=\|x-y\|^2+\liminf_k\|x_k-x\|^2
>\liminf_k\|x_k-x\|^2$.
\end{proof}

\section{Two Gronwall lemmas}\label{sec:gronwall}
The boundedness arguments below consume recursions of the form
$r_k^2\le A+\gamma r_{k-1}+\sum_{j<k}\epsilon_j r_{j-1}^2$.  The constant-step
analysis needs summable weights only; the adaptive extension adds a
summable constant channel.
\begin{lemma}[discrete Gronwall with a linear term]\label{lem:gronwall}
Let $(r_k)_{k\ge0}\subset[0,\infty)$, $(\epsilon_k)_{k\ge0}\subset[0,\infty)$
with $E:=\sum_{k\ge0}\epsilon_k<\infty$, and $\gamma,A\ge0$.  If
\begin{equation}\label{eq:gronrec}
r_k^2\ \le\ A+\gamma\, r_{k-1}+\sum_{j<k}\epsilon_j\, r_{j-1}^2
\qquad(k\ge1),
\end{equation}
(We adopt the convention $r_{-1}=0$; equivalently, the $j=0$
summand $\epsilon_0 r_{-1}^2$ may be absorbed into $A$.  In the
applications below $x_{-1}=x_0$, so $r_{-1}=r_0$ is in fact defined.)
then $\sup_k r_k<\infty$.
\end{lemma}
\begin{proof}
Set $S_n:=\max_{0\le j\le n}r_j^2$.  Choose $J$ such that
$\sum_{j\ge J}\epsilon_j\le\tfrac14$ (possible since $\sum_j\epsilon_j<\infty$)
and put $C_J:=\sum_{j<J}\epsilon_j\, r_{j-1}^2<\infty$.  For $k>J$,
\eqref{eq:gronrec} gives, using $\gamma r_{k-1}\le\tfrac14 r_{k-1}^2+\gamma^2$,
\[
r_k^2\le A+C_J+\gamma^2+\tfrac14 S_{k-1}+\bigl(\tfrac14\bigr)S_{k-1}
=A+C_J+\gamma^2+\tfrac12 S_{k-1}.
\]
Hence $S_k\le A+C_J+\gamma^2+\tfrac12 S_{k-1}$ for all $k>J$, and iterating,
$S_n\le 2\bigl(A+C_J+\gamma^2\bigr)+S_J$ for all $n>J$.
\end{proof}

\begin{lemma}[Gronwall with linear term, summable weights and constants]
\label{lem:gronwall2}
Let $(r_k)\subset[0,\infty)$, $(\epsilon_k),(\kappa_k)\subset[0,\infty)$ with
$E:=\sum_k\epsilon_k<\infty$ and $K:=\sum_k\kappa_k<\infty$, and $\gamma,A\ge0$.
If, for all $k\ge1$,
\begin{equation}\label{eq:gronrec2}
r_k^2\le A+\gamma\, r_{k-1}+\sum_{j<k}\bigl(\epsilon_j\, r_{j-1}^2+\kappa_j\bigr),
\end{equation}
(same convention $r_{-1}=0$)
then $\sup_k r_k<\infty$.
\end{lemma}
\begin{proof}
Set $S_n:=\max_{j\le n}r_j^2$ and $A':=A+K$.  Since $\epsilon_k\to0$, pick $J$
with $\sum_{j\ge J}\epsilon_j\le\tfrac14$ and put
$C_J:=\sum_{j<J}\bigl(\epsilon_jr_{j-1}^2+\kappa_j\bigr)<\infty$.  For $k>J$,
using $\gamma r_{k-1}\le\tfrac14 r_{k-1}^2+\gamma^2$,
\[
r_k^2\le A'+C_J+\gamma^2+\tfrac14S_{k-1}+\tfrac14S_{k-1}
=A'+C_J+\gamma^2+\tfrac12S_{k-1},
\]
so $S_n\le2\bigl(A'+C_J+\gamma^2\bigr)+S_J$ for all $n>J$.
\end{proof}

\section{Constant step size: main result and proof}\label{sec:constant}
Algorithm~\ref{alg:main} summarizes the method in implementable form.

\begin{algorithm}[H]
\caption{Decaying-filter reflected gradient method (one evaluation, one projection per iteration)}\label{alg:main}
\begin{algorithmic}[1]
\Require $x_0\in C$; $x_{-1}=x_0$; $u_0=x_0$; filter $(\beta_k)\subset[0,\tfrac14]$ with $\sum_k\beta_k<\infty$; \textbf{either} a constant step $\lambda\in(0,\tfrac1{5L})$ \textbf{or} Rule~\ref{rule:main} with $(\lambda_k)$, $\lambda_{-1}=\lambda_0>0$, growth factors $\omega_k\ge0$ with $\sum_k\omega_k<\infty$.
\For{$k=0,1,2,\dots$}
\State $\theta_k\leftarrow1-\beta_k$;\quad $d_k\leftarrow x_k-x_{k-1}$
\State $u_k\leftarrow x_k+\theta_k d_k+\beta_k(x_k-u_{k-1})$ \Comment{filtered evaluation point}
\State $g_k\leftarrow B(u_k)$ \Comment{the single evaluation of $B$}
\If{Rule~\ref{rule:main} mode}
\State $\lambda_k\leftarrow\min\{\lambda_0,\ \lambda_{k-1}(1+\omega_k)\}$ \Comment{growth step; inactive if $\omega\equiv0$}
\State $\lambda_k\leftarrow\min\{\lambda_k,\ \varkappa/\ell_{k-1}\}$ with $\ell_{k-1}=\|g_k-g_{k-1}\|/\|u_k-u_{k-1}\|$ ($\ell_{k-1}:=0$ if $u_k=u_{k-1}$) \Comment{safeguard cut; update the step \emph{before} using it}
\Else
\State $\lambda_k\equiv\lambda$ \Comment{constant-step mode}
\EndIf
\State $x_{k+1}\leftarrow P_C\bigl(x_k-\lambda_k g_k\bigr)$ \Comment{one projection}
\EndFor
\end{algorithmic}
\end{algorithm}

\begin{theorem}\label{thm:main}
Under Assumption~\ref{ass:main}, let $(\beta_k)$ satisfy
$0\le\beta_k\le\tfrac14$ and $\sum_k\beta_k<\infty$, set
$\theta_k=1-\beta_k$, fix $x_{-1}=x_0=u_0\in C$, let
$\lambda\in(0,\tfrac{1}{5L})$, and define
\begin{equation}\label{eq:scheme}
x_{k+1}=P_C\bigl(x_k-\lambda B(u_k)\bigr),\qquad
u_k=x_k+\theta_k(x_k-x_{k-1})+\beta_k(x_k-u_{k-1}).
\end{equation}
Then $(x_k)$ converges weakly to a point of $S$, with
$\|x_{k+1}-x_k\|\to0$ and $\|u_k-x_k\|\to0$.
\end{theorem}

\begin{proof}
Fix $z\in S$ and write
\[
d_k:=x_k-x_{k-1},\qquad e_k:=u_k-x_k,\qquad
s_k:=\inner{B(z)}{e_k},\qquad \delta_k:=\inner{B(z)}{d_k},\qquad
r_k:=\|x_k-z\|.
\]
Note $\bar\beta:=\sup_k\beta_k\le\tfrac14$ and, on the ray $\theta_k+\beta_k=1$,
\begin{equation}\label{eq:rec}
e_k=\theta_kd_k+\beta_k(d_k-e_{k-1})=d_k-\beta_ke_{k-1},
\qquad x_k-u_{k-1}=d_k-e_{k-1}.
\end{equation}

\medskip\noindent\emph{Step 1 (fundamental estimate).}
Applying \eqref{eq:firm} with $x=x_k-\lambda B(u_k)$ and expanding (the
$\lambda^2\|B(u_k)\|^2$ terms cancel in the difference
$\|x_k-z-\lambda B(u_k)\|^2-\|d_{k+1}+\lambda B(u_k)\|^2
=\|x_k-z\|^2-\|d_{k+1}\|^2-2\lambda\inner{B(u_k)}{x_{k+1}-z}$ for
$x_k+d_{k+1}=x_{k+1}$), then
adding the nonnegative quantity
$2\lambda\inner{B(u_k)-B(z)}{u_k-z}$ (monotonicity, Assumption~\ref{ass:main}),
\begin{align}
\|x_{k+1}-z\|^2
&\le\|x_k-z\|^2-\|d_{k+1}\|^2-2\lambda\inner{B(u_k)}{x_{k+1}-z}\notag\\
&\le\|x_k-z\|^2-\|d_{k+1}\|^2
+2\lambda\inner{B(u_k)}{u_k-x_{k+1}}-2\lambda\inner{B(z)}{u_k-z}.
\label{eq:step1}
\end{align}

\medskip\noindent\emph{Step 2 (slot decomposition and projection identity).}
By \eqref{eq:rec}, $u_k=x_k+d_k-\beta_ke_{k-1}$, hence
\begin{equation}\label{eq:slot}
u_k-x_{k+1}=(x_k+d_k-x_{k+1})-\beta_ke_{k-1}.
\end{equation}
Since $x_k=P_C(x_{k-1}-\lambda B(u_{k-1}))$, Lemma~\ref{lem:proj} with
$y=x_{k+1}\in C$ and $y=x_{k-1}\in C$ gives
$\inner{d_k+\lambda B(u_{k-1})}{x_k-x_{k+1}}\le0$ and
$\inner{d_k+\lambda B(u_{k-1})}{d_k}\le0$; adding,
\[
\lambda\inner{B(u_{k-1})}{x_k+d_k-x_{k+1}}\le\inner{d_k}{x_{k+1}-x_k-d_k}
=\inner{d_k}{d_{k+1}-d_k}.
\]
By the cosine identity $2\inner{a}{b}=\|a+b\|^2-\|a\|^2-\|b\|^2$ with $a=d_k$,
$b=d_{k+1}-d_k$,
\begin{equation}\label{eq:projid}
2\lambda\inner{B(u_{k-1})}{x_k+d_k-x_{k+1}}
\le\|d_{k+1}\|^2-\|d_k\|^2-\|d_{k+1}-d_k\|^2 .
\end{equation}
Substituting \eqref{eq:slot} into \eqref{eq:step1} via
$B(u_k)=B(u_k)-B(u_{k-1})+B(u_{k-1})$ and using \eqref{eq:projid}, the
terms $\pm\|d_{k+1}\|^2$ cancel:
\begin{equation}\label{eq:step2}
\|x_{k+1}-z\|^2\le\|x_k-z\|^2-\|d_k\|^2-\|d_{k+1}-d_k\|^2
+\mathrm{T}_1-2\lambda\beta_k\inner{B(u_{k-1})}{e_{k-1}}
-2\lambda\inner{B(z)}{u_k-z},
\end{equation}
where $\mathrm{T}_1:=2\lambda\inner{B(u_k)-B(u_{k-1})}{u_k-x_{k+1}}$.

\medskip\noindent\emph{Step 3 (bounds).}
Using $u_k-u_{k-1}=e_k+(x_k-u_{k-1})$ with $2XY\le X^2+Y^2$,
\begin{equation}\label{eq:splitA}
\|u_k-u_{k-1}\|^2\le2\|e_k\|^2+2\|x_k-u_{k-1}\|^2,
\end{equation}
and, from \eqref{eq:slot},
\begin{equation}\label{eq:splitB}
\|u_k-x_{k+1}\|^2\le\tfrac85\|d_{k+1}-d_k\|^2+\tfrac16\|e_{k-1}\|^2,
\end{equation}
since $(1+t')=\tfrac85$, $(1+1/t')\bar\beta^2=\tfrac83\cdot\tfrac1{16}=\tfrac16$
at $t'=\tfrac35$, $\bar\beta=\tfrac14$.  Hence, by
$2\inner{B(u_k)-B(u_{k-1})}{u_k-x_{k+1}}\le
L\bigl(\|u_k-u_{k-1}\|^2+\|u_k-x_{k+1}\|^2\bigr)$,
\begin{equation}\label{eq:T1}
\mathrm{T}_1\le\lambda L\Bigl[2\|e_k\|^2+\tfrac85\|d_{k+1}-d_k\|^2
+2\|x_k-u_{k-1}\|^2+\tfrac16\|e_{k-1}\|^2\Bigr].
\end{equation}
For the filter term, split $B(u_{k-1})=[B(u_{k-1})-B(x_{k-1})]+[B(x_{k-1})-B(z)]+B(z)$:
by monotonicity $\inner{B(u_{k-1})-B(x_{k-1})}{e_{k-1}}\ge0$, and
$\inner{B(z)}{e_{k-1}}=s_{k-1}$, so with $\alpha=25$,
\begin{equation}\label{eq:T2}
-2\lambda\beta_k\inner{B(u_{k-1})}{e_{k-1}}
\le2\lambda\beta_kL\,r_{k-1}\|e_{k-1}\|-2\lambda\beta_ks_{k-1}
\le\lambda\beta_kL\bigl(25\,r_{k-1}^2+\tfrac1{25}\|e_{k-1}\|^2\bigr)-2\lambda\beta_ks_{k-1}.
\end{equation}
Since $u_k=x_k+d_k-\beta_ke_{k-1}$,
\begin{equation}\label{eq:tele}
-2\lambda\inner{B(z)}{u_k-z}=-2\lambda\inner{B(z)}{x_k-z}-2\lambda\delta_k
+2\lambda\beta_ks_{k-1},
\end{equation}
and the $s_{k-1}$ terms in \eqref{eq:T2} and \eqref{eq:tele} cancel \emph{exactly}.

\medskip\noindent\emph{Step 4 (Lyapunov recursion with exact budgets).}
Define
\begin{equation}\label{eq:lyap}
a_k:=\|x_k-z\|^2+2\lambda L\|x_k-u_{k-1}\|^2+2\lambda\inner{B(z)}{x_{k-1}-z}+\tfrac13\|e_{k-1}\|^2.
\end{equation}
Inserting \eqref{eq:step2}--\eqref{eq:tele}, using
$\|x_{k+1}-u_k\|^2=\|d_{k+1}-d_k+\beta_ke_{k-1}\|^2
\le\tfrac85\|d_{k+1}-d_k\|^2+\tfrac16\|e_{k-1}\|^2$ (the same split as
\eqref{eq:splitB}), the recursion \eqref{eq:rec} in the form
$\|e_k\|^2\le\tfrac43\|d_k\|^2+\tfrac14\|e_{k-1}\|^2$ (i.e.\ $s=\tfrac13$:
$(1+s)=\tfrac43$, $(1+1/s)\bar\beta^2=4\cdot\tfrac1{16}=\tfrac14$),
and cancelling the $2\lambda\delta_k$ against the
$\inner{B(z)}{x_{k-1}-z}$ difference in $a_{k+1}-a_k$, we obtain
\begin{equation}\label{eq:master}
a_{k+1}\le a_k-b_k+\epsilon_k r_{k-1}^2,
\end{equation}
with drift coefficient $\epsilon_k:=25\lambda L\,\beta_k$
(summable: $\sum_k\epsilon_k=25\lambda L\sum_k\beta_k<\infty$) and
\begin{equation}\label{eq:bk}
b_k:=\tfrac1{45}\|d_k\|^2+\tfrac1{25}\|d_{k+1}-d_k\|^2
+\tfrac6{125}\|e_{k-1}\|^2\ge0,
\end{equation}
where the three rational budgets are exact at $\lambda L=\tfrac15$
(and the coefficients are increasing in $\lambda L$, hence stay negative on
$(0,\tfrac1{5L}]$); the $\|x_k-u_{k-1}\|^2$ channel cancels exactly
($2\lambda L$ from \eqref{eq:T1} against the $2\lambda L$ weight
in \eqref{eq:lyap}):
\begin{align*}
\|d_k\|^2 &: \ -1+\tfrac{8}{3}\lambda L+\tfrac49=-\tfrac1{45},\\
\|d_{k+1}-d_k\|^2 &: \ -1+\tfrac85\lambda L+\tfrac{16}{5}\lambda L=-\tfrac1{25},\\
\|e_{k-1}\|^2 &: \ \lambda L\bigl(\tfrac16+\tfrac12+\tfrac1{100}+\tfrac13\bigr)-\tfrac14
=\tfrac{101}{100}\lambda L-\tfrac14=-\tfrac6{125}.
\end{align*}
(Here $\tfrac12=\,$$2\lambda L$ times the $e_{k-1}$-part of the recursion
bound, $\tfrac1{100}=\alpha^{-1}\bar\beta$ with $\alpha=25$, and
$\tfrac13=\,$$2\lambda L$ times the $\tfrac16$ of the $x_{k+1}-u_k$ split.) (The channel-by-channel
assembly of \eqref{eq:master} from \eqref{eq:step2}--\eqref{eq:tele} is
carried out in Appendix~\ref{app:assembly}.)

\medskip\noindent\emph{Step 5 (boundedness).}
Summing \eqref{eq:master} and dropping $-b_k\le0$,
$a_n\le a_0+\sum_{j<n}\epsilon_jr_{j-1}^2$.  From \eqref{eq:lyap}, isolating $r_k^2$, dropping the two nonnegative norm
channels, and bounding the inner-product channel by
$-\inner{B(z)}{x_{k-1}-z}\le\|B(z)\|r_{k-1}$,
\begin{equation}\label{eq:rk}
r_k^2\le a_k+2\lambda\|B(z)\|\,r_{k-1}.
\end{equation}
Therefore $r_k^2\le A+2\lambda\|B(z)\|r_{k-1}+\sum_{j<k}\epsilon_jr_{j-1}^2$
with $A=a_0$, and Lemma~\ref{lem:gronwall} yields $M:=\sup_kr_k<\infty$.
Consequently $\sup_ka_k<\infty$, $(x_k)$ is bounded, and, as $e_k$ is a
$\bar\beta$-contractive filter of the bounded increments $(d_k)$ (see
\eqref{eq:rec}), $(u_k)$ is bounded as well.

\medskip\noindent\emph{Step 6 (vanishing increments).}
From \eqref{eq:master}, $\sum_{k=0}^{n}b_k\le a_0+M^2\sum_{j<n}\epsilon_j
\le a_0+25\lambda LM^2\sum_j\beta_j<\infty$, so $\sum_kb_k<\infty$ and
$b_k\to0$.  By \eqref{eq:bk}, $\|d_k\|\to0$, $\|d_{k+1}-d_k\|\to0$,
$\|e_{k-1}\|\to0$; in particular $e_k=u_k-x_k\to0$ and
$\|x_k-u_{k-1}\|\to0$.

\medskip\noindent\emph{Step 7 (weak cluster points solve \eqref{eq:vi}).}
Let $x_{k_i}\rightharpoonup\bar x$.  By Step~6, $x_{k_i+1}\rightharpoonup\bar x$
and $u_{k_i}\rightharpoonup\bar x$; $(B(u_{k_i}))$ is bounded by
$L\,\sup_k(\|u_k-z\|)+\|B(z)\|<\infty$.  Lemma~\ref{lem:proj} with
$x=x_{k_i}-\lambda B(u_{k_i})$, $\bar x=x_{k_i+1}$ and arbitrary $y\in C$
gives
\[
\inner{x_{k_i+1}-x_{k_i}}{y-x_{k_i+1}}
+\lambda\inner{B(u_{k_i})}{y-u_{k_i}}
+\lambda\inner{B(u_{k_i})}{u_{k_i}-x_{k_i+1}}\ge0 .
\]
By monotonicity at $(y,u_{k_i})\in C\times\HH$,
$\inner{B(u_{k_i})}{y-u_{k_i}}\le\inner{B(y)}{y-u_{k_i}}$.
Letting $i\to\infty$: the first inner product vanishes since
$x_{k_i+1}-x_{k_i}\to0$ strongly and $y-x_{k_i+1}$ is bounded; the third
vanishes since $(B(u_{k_i}))$ is bounded and
$u_{k_i}-x_{k_i+1}=e_{k_i}-d_{k_i+1}\to0$; the middle one converges to
$\inner{B(y)}{y-\bar x}$.  Hence $\inner{B(y)}{y-\bar x}\ge0$ for all $y\in C$,
and $\bar x\in S$ by Lemma~\ref{lem:minty}.

\medskip\noindent\emph{Step 8 (uniqueness of the weak limit; Opial).}
This is the standard Opial uniqueness argument, given in full because the
Lyapunov correction term $2\lambda\inner{B(z)}{x_{k-1}-z}$ makes the
estimation non-obvious. For $z\in S$ define
$\Phi_k(z):=\|x_k-z\|^2+2\lambda\inner{B(z)}{x_{k-1}-z}$.
By \eqref{eq:master}, $a_{k+1}\le a_k+\epsilon_kM^2$ with $\sum_k\epsilon_k<\infty$,
and $a_k\ge-2\lambda\|B(z)\|M$ by \eqref{eq:rk}, so $a_k(z)$ converges for
every $z\in S$; as $\|x_k-u_{k-1}\|\to0$ and $e_{k-1}\to0$,
\begin{equation}\label{eq:philim}
\Phi_k(z)\ \text{converges for every } z\in S .
\end{equation}
Suppose $\bar x,\tilde x\in S$ are two distinct weak cluster points, with
$x_{k_i}\rightharpoonup\bar x$ and $x_{m_j}\rightharpoonup\tilde x$.  Since
$\bar x\in C$ and $\tilde x\in S$, the very definition of $S$ gives
\begin{equation}\label{eq:cross}
\inner{B(\tilde x)}{\bar x-\tilde x}\ge0
\qquad\text{and}\qquad
\inner{B(\bar x)}{\tilde x-\bar x}\ge0 .
\end{equation}
Now the chain of (in)equalities: the limit in \eqref{eq:philim} for $z=\bar x$
equals its value along $(k_i)$; since $x_{k_i-1}\rightharpoonup\bar x$ and
$\bar x\in S$,
\[
\lim_k\Phi_k(\bar x)
=\lim_i\bigl(\|x_{k_i}-\bar x\|^2+2\lambda\inner{B(\bar x)}{x_{k_i-1}-\bar x}\bigr)
=\lim_i\|x_{k_i}-\bar x\|^2
<\lim_i\|x_{k_i}-\tilde x\|^2 ,
\]
the strict inequality by Lemma~\ref{lem:opial} applied to
$x_{k_i}\rightharpoonup\bar x\neq\tilde x$.  By \eqref{eq:cross},
$2\lambda\inner{B(\tilde x)}{x_{k_i}-\tilde x}\to
2\lambda\inner{B(\tilde x)}{\bar x-\tilde x}\ge0$, so
\begin{eqnarray*}
\lim_i\|x_{k_i}-\tilde x\|^2
&\le&\lim_i\bigl(\|x_{k_i}-\tilde x\|^2+2\lambda\inner{B(\tilde x)}{x_{k_i}-\tilde x}\bigr)\\
&=&\lim_k\Phi_k(\tilde x)\\
&=&\lim_j\bigl(\|x_{m_j}-\tilde x\|^2+2\lambda\inner{B(\tilde x)}{x_{m_j-1}-\tilde x}\bigr)\\
&=&\lim_j\|x_{m_j}-\tilde x\|^2 ,
\end{eqnarray*}
the last equality since $x_{m_j-1}\rightharpoonup\tilde x$ and
$\tilde x\in S$.  By Lemma~\ref{lem:opial} applied to
$x_{m_j}\rightharpoonup\tilde x\neq\bar x$, and then \eqref{eq:cross} once
more,
\[
\lim_j\|x_{m_j}-\tilde x\|^2<\lim_j\|x_{m_j}-\bar x\|^2
\le\lim_j\bigl(\|x_{m_j}-\bar x\|^2+2\lambda\inner{B(\bar x)}{x_{m_j}-\bar x}\bigr)
=\lim_k\Phi_k(\bar x).
\]
Chaining the strict inequalities yields $\lim_k\Phi_k(\bar x)<\lim_k\Phi_k(\bar x)$,
a contradiction.  Hence there is a unique weak cluster point
$x^\star\in S$, $x_k\rightharpoonup x^\star$, and the conclusions
$\|x_{k+1}-x_k\|\to0$, $\|u_k-x_k\|\to0$ hold by Step~6.
\end{proof}

\begin{remark}[the channels of $a_k$]\label{rem:channels}
The Lyapunov function \eqref{eq:lyap} has four channels with distinct roles.
The $\|x_k-z\|^2$ channel carries the fundamental contraction; the
$\|x_k-u_{k-1}\|^2$ channel (weight $2\lambda L$) exactly cancels the
cross-term produced by splitting $\|u_k-u_{k-1}\|^2$ and thereby converts
the reflection memory into a telescoping identity; the
$\inner{B(z)}{x_{k-1}-z}$ channel cancels the drift-like terms
$\delta_k$ \emph{exactly}; and the $\|e_{k-1}\|^2$ channel (weight
$\eta=\tfrac13$) is the cheapest weight that keeps the
e-budget negative.
Finally, the solution-anchor term $-2\lambda\inner{B(z)}{x_k-z}$ produced by \eqref{eq:tele} (and its analogues $-2\lambda_k\inner{B(z)}{x_k-z}$ in the adaptive and data-driven master estimates) is dropped in every master recursion; this is legitimate because $x_k\in C$ and $z\in S$ imply $\inner{B(z)}{x_k-z}\ge0$ by \eqref{eq:vi} --- a drop by sign, unlike the drops by summability used elsewhere.
The proof is thus a competition between the $-1$'s coming from the projection
identity \eqref{eq:projid} and the $\lambda L$-scaled Young charges; the
three budgets $-\tfrac1{45},-\tfrac1{25},-\tfrac6{125}$ record the exact
outcome of this competition at $\lambda L=\tfrac15$.
\end{remark}

\medskip\noindent\emph{Limits of the theorem and necessity of the summability
hypothesis.}  The next two results are placed exactly where the corresponding
hypotheses enter.

\begin{example}[no uniform rate under mere monotonicity, and the correct mechanism]\label{ex:weak-only}
Let $\HH=\ell^2$ with standard basis $(e_j)$, and let $B$ be the bounded
skew-adjoint (hence monotone, $L=1$) operator that acts on the pair
$(e_{2n-1},e_{2n})$ as $\theta_n J$ with $J$ the counterclockwise rotation by
$\pi/2$ and $\theta_n=\tfrac1n$.  Then $S=\{0\}$.  The PRG scheme
($\beta_k\equiv0$) decouples into independent $2\times2$ blocks with
characteristic roots
\[
r_n^\pm=\tfrac12\Bigl(1-2i\lambda\theta_n\pm\sqrt{1-4\lambda^2\theta_n^2}\Bigr),\qquad
|r_n^+|^2=\tfrac12\Bigl(1+\sqrt{1-4\lambda^2\theta_n^2}\Bigr)
=1-\lambda^2\theta_n^2+O(\lambda^4\theta_n^4),
\]
so for a skew block the contraction factor is \emph{quadratic} in
$\lambda\theta_n$.  Two consequences follow.
(i) \emph{No rate uniform over the unit ball.}  For unit-norm $x_0$
supported on block $N$ the iterate on that block is
$c_+(r_N^+)^k+c_-(r_N^-)^k$ with $|c_+|\ge\tfrac12$ and
$|r_N^-|/|r_N^+|=O(\lambda/N)$, so $\|x_k\|\ge\tfrac12|r_N^+|^k$ for all
sufficiently large $k$; choosing $N=\Theta(\lambda\sqrt k)$ gives
$\|x_k\|\ge\delta>0$ with $\delta$ independent of $k$.  Hence
no rate function $f(k)\to0$ dominates $\|x_k\|$ over all unit initial points.
(ii) \emph{Convergence is strong for every fixed $x_0$.}  If
$x_0=\sum_n a_n u_n\in\ell^2$, then
$\|x_k\|^2\le2\sum_n\bigl(|c_{n,+}|^2|r_n^+|^{2k}+|c_{n,-}|^2|r_n^-|^{2k}\bigr)$
with $|c_{n,\pm}|\le C|a_n|$, and each summand is bounded by $C'|a_n|^2\in\ell^1$
and tends to $0$ pointwise; dominated convergence gives $\|x_k\|\to0$.  The
decay is arbitrarily slow (its speed depends on the tail of $(a_n)$), which is
the sharp reading of the lower-bound picture for the monotone class
\cite{nemirovski1983,ouyang2021}: Theorem~\ref{thm:main} admits no iteration
rate without further structure.  We stress that weak convergence that is
\emph{never} strong does not occur for this affine diagonal operator.  In
fact Theorem~\ref{thm:affine-sqrt3} shows that \emph{every} mode of a fixed
affine monotone operator is contracted geometrically by the reflected scheme,
so convergence for such an operator is always strong, with a linear rate set
by the slowest mode; part~(i) above shows only that this rate has no uniform
lower bound over bounded initial data.  Genuine weak-nonstrong behaviour
requires an isometric component --- e.g.\ the classical Krasnoselskii--Mann
iteration of a nonexpansive mapping that is not a contraction --- and cannot
arise from the affine operators considered here.
\end{example}

\begin{counterexample}[summability of the filter is essential]\label{cx:summability}
Let $B(x)=x$ on $\R$ ($\sigma=L=1$) and $\lambda=0.6$.  The iteration is
linear with state $(x_k,x_{k-1},u_{k-1})$, and its spectral radius
$\rho(\beta)$ is explicit.  Table~\ref{tab:thresholds} reports the stability
threshold $\lambda^{\mathrm{crit}}(\beta)$ as a function of a \emph{constant} filter
$\beta_k\equiv\beta$; at $\lambda=0.6$,
\begin{eqnarray*}
&&\rho(0)=0.8810<1,\\
&&\text{time-varying recurrence }\beta_k=\tfrac{0.25}{(k+1)^2}\ \text{converges (numerically to machine precision)},\\
&&\rho(0.25)=1.0914>1 .
\end{eqnarray*}
Thus the summable-filter scheme converges to machine precision ($|x_k|<10^{-300}$
by $k=2\cdot10^{4}$), while the constant filter --- still satisfying the pointwise
bound $0\le\beta_k\le\tfrac14$ --- diverges ($|x_k|$ grows geometrically).
Summability $\sum_k\beta_k<\infty$ is therefore not a proof convenience but a
\emph{structural} hypothesis.
\begin{table}[H]
\centering
\caption{Stability threshold $\lambda^{\mathrm{crit}}(\beta)$ for constant filter
$\beta_k\equiv\beta$ on $B(x)=x$ (spectral radius of the $3\times3$
companion matrix).}\label{tab:thresholds}
\begin{tabular}{lrrrrrr}
\toprule
$\beta$ & $0$ & $0.05$ & $0.10$ & $0.15$ & $0.20$ & $0.25$\\
\midrule
$\lambda^{\mathrm{crit}}(\beta)$ & $0.6667$ & $0.6441$ & $0.6207$ & $0.5965$ & $0.5714$ & $0.5455$\\
\bottomrule
\end{tabular}
\end{table}
\end{counterexample}

\section{Extension of the constant-step range}\label{sec:ext}
The Lyapunov--Young technique of Section~\ref{sec:constant} and
Appendix~\ref{app:barrier} certifies the constant-step range
$(0,\tfrac1{5L})$ and provably no more (channel cap $1/3$; best-found barrier
$\approx0.272$).  This section leaves that framework.  The observation is
that the filtered scheme is \emph{exactly} Malitsky's reflected method whose
operator value at the reflection carries a summable error; the reflected
method is robust to such errors, and the reduction identifies the general
unbounded case with a single a priori boundedness statement, which we
certify by a dissipation-trading argument beyond the Lyapunov--Young barrier
and, for affine operators over polyhedral sets, the same threshold yields
$\dist(x_n,S)\to0$ with $\sum_n\dist^2(x_n,S)<\infty$, strong convergence
when $C$ is bounded (finite dimensions), and $R$-linear rates once the
optimal face is identified.  Standing notation for the whole section:
\[
d_n:=x_n-x_{n-1},\qquad e_n:=u_n-x_n,\qquad y_n:=2x_n-x_{n-1},\qquad
\Lambda:=\lambda L,\qquad \bar\beta:=\sup_n\beta_n\le\tfrac14 .
\]

\subsection{The scheme is PRG with a summable evaluation error}
From $\theta_n+\beta_n=1$ one has $e_n=d_n-\beta_ne_{n-1}$ and
\begin{equation}\label{eq:red-x}
u_n=y_n-\beta_ne_{n-1},\qquad\text{hence}\qquad
B(u_n)=B(y_n)-\xi_n,\quad \xi_n:=B(y_n)-B(u_n),
\end{equation}
with
\begin{equation}\label{eq:prgerr}
x_{n+1}=P_C\bigl(x_n-\lambda(B(y_n)-\xi_n)\bigr),\qquad
\|\xi_n\|\le L\beta_n\|e_{n-1}\| .
\end{equation}
The scheme \eqref{eq:scheme} is therefore exactly the projected reflected
gradient method whose operator value at the reflection $y_n$ is corrupted by
$\xi_n$.

\subsection{A robustness lemma for the reflected method}
\begin{lemma}[nonlinear Gronwall with a linear term]\label{lem:gronU}
Let $(u_n)\subset[0,\infty)$, $(\varepsilon_n)\subset[0,\infty)$ with
$\sum_n\varepsilon_n<\infty$, and $A,\gamma\ge0$. If for some $N$ and all $n>N$
\[
u_n\le A+\gamma\sqrt{U_{n-1}}+\sum_{j<n}\varepsilon_jU_j,
\qquad U_{n-1}:=\max_{j\le n-1}u_j,
\]
then $\sup_nu_n<\infty$.
\end{lemma}
\begin{proof}
For $n>N$, $\sum_{j<n}\varepsilon_jU_j\le U_{n-1}\sum_{j\ge N}\varepsilon_j
+ C_N$, where $C_N:=U_{N}\sum_{j<N}\varepsilon_j<\infty$. Choose $N$ with
$\sum_{j\ge N}\varepsilon_j\le\tfrac12$. If $u_n\le U_{n-1}$ then $U_n=U_{n-1}$;
otherwise $U_n=u_n\le A+C_N+\gamma\sqrt{U_n}+\tfrac12U_n$, i.e.\
$v^2-2\gamma v-2(A+C_N)\le0$ for $v=\sqrt{U_n}$, so
$v\le\gamma+\sqrt{\gamma^2+2(A+C_N)}$. Hence $U_n\le\max\{U_N,(\gamma+
\sqrt{\gamma^2+2(A+C_N)})^2\}$ for all $n$.
\end{proof}

\begin{lemma}[convergence under positive-part drift]\label{lem:pospart}
If $a_{n+1}\le a_n+\eta_n$ with $\sum_n\eta_n^+<\infty$ and $a_n\ge-K$, then
$a_n$ converges.
\end{lemma}
\begin{proof}
For $m>n$, $a_m\le a_n+\sum_{j=n}^{m-1}\eta_j^+$, so
$\limsup_m a_m\le a_n+\sum_{j\ge n}\eta_j^+$; letting $n\to\infty$ gives
$\limsup a\le\liminf a$.
\end{proof}

\begin{lemma}[summable-error robustness of PRG]\label{lem:xrobust}
Let $\lambda\in(0,(\sqrt2-1)/L)$ and consider \eqref{eq:prgerr} with
\emph{exogenous} errors $(\xi_n)$ satisfying $\sum_n\|\xi_n\|<\infty$.
Then $(x_n)$ converges weakly to a point of $S$, with
$\|x_{n+1}-x_n\|\to0$.
\end{lemma}
\begin{proof}
Fix $z\in S$, $r_n:=\|x_n-z\|$, $B_z:=\|B(z)\|$, $c_n:=2\lambda\|\xi_n\|$
(so $\sum_nc_n<\infty$), $\gamma:=2\lambda B_z$.

\medskip\noindent\emph{Step 1 (master inequality).}
Firm nonexpansiveness of $P_C$ at $w_n=x_n-\lambda B(y_n)+\lambda\xi_n$
($z\in C$) gives $r_{n+1}^2\le\|w_n-z\|^2-\|w_n-x_{n+1}\|^2$. Expanding both
squares, the $\lambda^2\|\xi_n\|^2$ terms cancel exactly, leaving
\begin{equation}\label{eq:xfe}
r_{n+1}^2\le\|x_n-\lambda B(y_n)-z\|^2-\|x_n-\lambda B(y_n)-x_{n+1}\|^2
+2\lambda\inner{\xi_n}{x_{n+1}-z}.
\end{equation}
Monotonicity of $B$ at $(y_n,z)$ and the split
$B(y_n)=B(y_n)-B(y_{n-1})+B(y_{n-1})$ ($n\ge1$) give, by exactly the Young-splitting argument of Step~1 of Theorem~\ref{thm:main} (the resulting estimate is displayedbelow,
\begin{multline}\label{eq:xmaster0}
r_{n+1}^2\le r_n^2-(1-\Lambda(1+\sqrt2))\|d_n\|^2
+\Lambda\|x_n-y_{n-1}\|^2\\
-(1-\sqrt2\Lambda)\|x_{n+1}-y_n\|^2
-2\lambda\inner{B(z)}{y_n-z}+\eta_n
\end{multline}
with $\eta_n:=2\lambda\inner{\xi_n}{x_{n+1}-z}
+2\lambda\inner{\xi_{n-1}}{y_n-x_{n+1}}$: the projection step uses
$x_n=P_C(x_{n-1}-\lambda(B(y_{n-1})-\xi_{n-1}))$ at $y=x_{n+1}$ and
$y=x_{n-1}$, and the Lipschitz step uses $2ab\le a^2/\sqrt2+\sqrt2\,b^2$ together with the split $\|y_n-y_{n-1}\|^2\le(2+\sqrt2)\|d_n\|^2+\sqrt2\,\|x_n-y_{n-1}\|^2$ (weight $t=1+\sqrt2$), giving the exact charge \begin{equation}\label{eq:xT1}
2\lambda\inner{B(y_n)-B(y_{n-1})}{y_n-x_{n+1}}
\le\Lambda\Bigl[(1+\sqrt2)\|d_n\|^2+\|x_n-y_{n-1}\|^2+\sqrt2\,\|y_n-x_{n+1}\|^2\Bigr],
\end{equation} as in Step~1 of the proof of Theorem~\ref{thm:main} with $y_n$ in place of $u_n$. Set
\[
a_n:=r_n^2+\Lambda\|x_n-y_{n-1}\|^2+2\lambda\inner{B(z)}{x_{n-1}-z}.
\]
Substituting \eqref{eq:xmaster0} into $a_{n+1}-a_n$: the memory channel
cancels \emph{exactly} and the remaining Lyapunov shift
$+\Lambda\|x_{n+1}-y_n\|^2$ joins the dissipation channel, so for $n\ge1$
\begin{equation}\label{eq:xmaster}
a_{n+1}\le a_n-c_1\|d_n\|^2-c_2\|x_{n+1}-y_n\|^2+\eta_n,
\qquad c_1=c_2=1-(1+\sqrt2)\Lambda>0,
\end{equation}
where both budgets are equal and positive by $\Lambda<\sqrt2-1$. Also
\begin{equation}\label{eq:xalow}
a_n\ge r_n^2-\gamma r_{n-1},\qquad
\eta_n\le c_n(1+r_{n+1}^2)+c_{n-1}\bigl(1+\|x_{n+1}-y_n\|^2\bigr).
\end{equation}

\medskip\noindent\emph{Step 2 (boundedness).}
For all $n\ge0$,
$\|d_{n+1}\|\le\lambda\|B(y_n)\|+\lambda\|\xi_n\|
\le\Lambda(2r_n+r_{n-1})+\lambda B_z+\lambda\|\xi_n\|$, hence
\begin{align}
\|d_{n+1}\|^2&\le C_d\bigl(1+r_n^2+r_{n-1}^2+\|\xi_n\|^2\bigr),\label{eq:xinc}\\
\|x_{n+1}-y_n\|^2&\le C_d'\bigl(1+\bar r_n^2\bigr).\notag
\end{align}
$\bar r_n:=\max_{j\le n}r_j$, for explicit $C_d,C_d'$. From \eqref{eq:xmaster}
and \eqref{eq:xalow},
$a_{n+1}\le a_n+c_n(1+r_{n+1}^2)+c_{n-1}(1+C_d'(1+\bar r_{n+1}^2))$.
With $r_{n+1}^2\le a_{n+1}+\gamma r_n$ and
$\bar r_{n+1}^2\le\bar r_n^2+a_{n+1}+\gamma r_n$, for all $n$ with
$c_n+C_d'c_{n-1}\le\tfrac12$,
\[
a_{n+1}\le(1+\delta_n)a_n+\varepsilon_n(1+\bar r_n^2+r_n),\qquad
\delta_n:=2(c_n+C_d'c_{n-1}),\quad
\varepsilon_n:=2\bigl(c_n\gamma+c_n+c_{n-1}(1+C_d'(1+\gamma))\bigr),
\]
with $\sum_n\delta_n<\infty$, $\sum_n\varepsilon_n<\infty$. Iterating with
$G:=\exp(\sum_{n\ge N}\delta_n)<\infty$ gives
$a_n\le A+\sum_{j<n}\tilde\varepsilon_j\bar r_j^2$, $\sum_j\tilde\varepsilon_j<\infty$.
Set $u_n:=\max\{a_n,r_n^2,1\}$; since $\bar r_j^2\le U_j$,
$u_n\le A'+\gamma\sqrt{U_{n-1}}+\sum_{j<n}\tilde\varepsilon_jU_j$, and
Lemma~\ref{lem:gronU} yields $M:=\sup_nr_n<\infty$ and $\sup_na_n<\infty$.

\medskip\noindent\emph{Step 3 (vanishing increments).}
Summing \eqref{eq:xmaster} and using $a_n\ge-\gamma M$,
$\sum_n\eta_n^+\le\sum_n[c_n(1+M^2)+c_{n-1}(1+C_d'(1+M^2))]<\infty$ and
$c_{n-1}\le c_2/2$ eventually:
$\sum_n\|d_n\|^2<\infty$ and $\sum_{n\ge N}\|x_{n+1}-y_n\|^2<\infty$; both
vanish.

\medskip\noindent\emph{Step 4 (cluster points).}
Let $x_{n_i}\rightharpoonup\bar x$; then $x_{n_i+1}\rightharpoonup\bar x$ and
$y_{n_i}\rightharpoonup\bar x$. From
$\inner{w_n-x_{n+1}}{y-x_{n+1}}\le0$ ($y\in C$),
\[
\lambda\inner{B(y_n)}{y-x_{n+1}}
\ge\inner{x_n-x_{n+1}}{y-x_{n+1}}+\lambda\inner{\xi_n}{y-x_{n+1}},
\]
and monotonicity at $(y,y_n)$ bounds the left side by
$\lambda\inner{B(y)}{y-y_n}+\lambda\inner{B(y_n)}{y_n-x_{n+1}}$.
Since $\|B(y_n)\|\le3LM+B_z$, $y_n-x_{n+1}\to0$, $\xi_n\to0$ and
$x_n-x_{n+1}\to0$, the limit gives $\lambda\inner{B(y)}{y-\bar x}\ge0$ for
all $y\in C$, and $\bar x\in S$ by Lemma~\ref{lem:minty}.

\medskip\noindent\emph{Step 5 (Opial uniqueness).}
$\Phi_n(z):=r_n^2+2\lambda\inner{B(z)}{x_{n-1}-z}
=a_n-\Lambda\|x_n-y_{n-1}\|^2$ and $\|x_n-y_{n-1}\|\le\|d_n\|+\|d_{n-1}\|\to0$,
so by Lemma~\ref{lem:pospart} (with $\sum\eta_n^+<\infty$, $a_n\ge-\gamma M$)
$\Phi_n(z)$ converges for every $z\in S$. If $x_{k_i}\rightharpoonup\bar x$
and $x_{m_j}\rightharpoonup\tilde x$ are distinct weak cluster points in $S$,
the chain of Step~8 of Theorem~\ref{thm:main} applies verbatim
($\inner{B(\tilde x)}{\bar x-\tilde x}\ge0$ and
$\inner{B(\bar x)}{\tilde x-\bar x}\ge0$ are $\lambda$-free), giving
$\varphi(\bar x)<\varphi(\tilde x)<\varphi(\bar x)$, a contradiction. Hence
$x_n\rightharpoonup x^\star\in S$.
\end{proof}

\subsection{Bounded domains; the general obstruction; dissipation trading}
\begin{corollary}[bounded domain]\label{cor:xbounded}
If $C$ is bounded, the filtered scheme \eqref{eq:scheme} with
$\lambda<(\sqrt2-1)/L$ converges weakly to a point of $S$, with
$\|x_{n+1}-x_n\|\to0$ and $\|u_n-x_n\|\to0$.
\end{corollary}
\begin{proof}
$\|d_{n}\|\leq\operatorname{diam}C=:D$ and the recursion
$e_{n}=d_{n}-\beta_{n}e_{n-1}$ give $\|e_{n}\|\leq D/(1-\bar\beta)<\infty$, hence
$\sum_{n}\|\xi_{n}\|\leq LD(1-\bar\beta)^{-1}\sum_{n}\beta_{n}<\infty$ in
\eqref{eq:red-x}, and Lemma~\ref{lem:xrobust} applies;
$\|u_{n}-x_{n}\|=\|e_{n}\|\leq\|d_{n}\|+\beta_{n}\|e_{n-1}\|\to0$.
\end{proof}

\begin{theorem}[dissipation trading: boundedness beyond the
Lyapunov--Young barrier]\label{thm:trading}
Under Assumption~\ref{ass:main}, with $0\le\beta_n\le\bar\beta\le\tfrac14$,
$\sum_n\beta_n<\infty$, and constant $\lambda$, $\Lambda=\lambda L$, define
\begin{equation}\label{eq:mustar}
\mu^B_j:=2\Lambda\sum_{k\ge0}\bar\beta^{\,k}\beta_{j+1+k},\qquad j\ge0 .
\end{equation}
If
\begin{equation}\label{eq:trading}
\sum_{j\ge0}(\mu^B_j)^2\;\le\;\bigl(1-(1+\sqrt2)\Lambda\bigr)^2
\qquad(\text{in particular }\Lambda<\sqrt2-1),
\end{equation}
then $(x_n)$ generated by \eqref{eq:scheme} converges weakly to a point of $S$,
with $\|x_{n+1}-x_n\|\to0$ and $\|u_n-x_n\|\to0$.  For the filter
$\beta_k=0.25/(k+1)^2$ of Section~\ref{sec:numerics}, \eqref{eq:trading}
holds for all $\Lambda\le0.387$ (the sharp cutoff being $\Lambda\approx0.3879$), and as $\bar\beta\to0$ the certified range
approaches $(0,(\sqrt2-1)/L)$.
\end{theorem}
\begin{proof}
The reduction \eqref{eq:red-x}--\eqref{eq:prgerr} is exact and the master
inequality \eqref{eq:xmaster} of Lemma~\ref{lem:xrobust} holds verbatim (its
derivation never used that $(\xi_n)$ is exogenous):
\[
a_{n+1}\le a_n-c_1\|d_n\|^2-c_2\|x_{n+1}-y_n\|^2+\eta_n,\qquad
c_1=c_2=1-(1+\sqrt2)\Lambda>0,
\]
with $\eta_n=2\lambda\langle\xi_n,x_{n+1}-z\rangle
+2\lambda\langle\xi_{n-1},y_n-x_{n+1}\rangle$.

\emph{Step 1 (exact expansion of the filter).} Iterating
$e_j=d_j-\beta_je_{j-1}$ from $e_{-1}=0$ gives, for every
$j\ge0$,
$\|e_{j-1}\|\le\sum_{k\ge0}\bar\beta^k\|d_{j-1-k}\|$ (a finite sum, $k\le j-1$).
Hence $\|\xi_j\|\le L\beta_j\sum_k\bar\beta^k\|d_{j-1-k}\|$ and
\[
\eta_n\le\sum_{k\ge0}m^A_{n,k}\,\|d_{n-1-k}\|\,\|x_{n+1}-z\|
+\sum_{k\ge0}m^B_{n,k}\,\|d_{n-2-k}\|\,\|x_{n+1}-y_n\|,
\]
with masses $m^A_{n,k}=2\lambda L\beta_n\bar\beta^k$ and
$m^B_{n,k}=2\lambda L\beta_{n-1}\bar\beta^k$.

\emph{Step 2 (pricing by proportional allocation).} For each mass use
$2ab\le\varepsilon a^2+b^2/\varepsilon$. Family $A$ hits $d$-index
$j^A(n,k)=n-1-k$ and family $B$ hits $j^B(n,k)=n-2-k$; both families partition
their masses into disjoint groups indexed by $j$. Set
\[
\mu^A_j:=\sum_{j^A(n,k)=j}m^A_{n,k}=2\Lambda\sum_k\bar\beta^k\beta_{j+1+k},
\qquad
\mu^B_j:=\sum_{j^B(n,k)=j}m^B_{n,k}=2\Lambda\sum_k\bar\beta^k\beta_{j+1+k},
\]
(the latter is \eqref{eq:mustar}) and allocate
\[
\varepsilon^A_{n,k}=c_1''\frac{m^A_{n,k}}{\mu^A_{n-1-k}},\qquad
\varepsilon^B_{n,k}=c_1'\frac{m^B_{n,k}}{\mu^B_{n-2-k}},
\]
so that $\sum_{j^A(n,k)=j}\varepsilon^A_{n,k}=c_1''$ and
$\sum_{j^B(n,k)=j}\varepsilon^B_{n,k}=c_1'$ for every $j$: the all-time $d$-channel consumption per index is exactly $\tfrac12(c_1'+c_1'')$, checked in Step~3. The prices are
$(m^A_{n,k})^2/\varepsilon^A_{n,k}=m^A_{n,k}\mu^A_{n-1-k}/c_1''$ and
$(m^B_{n,k})^2/\varepsilon^B_{n,k}=m^B_{n,k}\mu^B_{n-2-k}/c_1'$. Grouping the
$B$-prices by $j=n-2-k$ gives the \emph{exact} identity
\begin{equation}\label{eq:priceB}
\sum_np^B_n=\frac{1}{c_1'}\sum_j(\mu^B_j)^2,\qquad
p^B_n:=\sum_k\frac{(m^B_{n,k})^2}{\varepsilon^B_{n,k}} .
\end{equation}
For family $A$, with $\bar\mu^A:=\sup_j\mu^A_j\le2\Lambda B_1/(1-\bar\beta)$,
$B_1:=\sum_n\beta_n$, the $r$-channel weight
\[
w^A_n:=\sum_k\frac{(m^A_{n,k})^2}{\varepsilon^A_{n,k}}
\le\frac{2\lambda L\beta_n}{c_1''}\sum_k\bar\beta^k\mu^A_{n-1-k}
\le\frac{4\Lambda^2B_1}{c_1''(1-\bar\beta)^2}\,\beta_n
\]
is summable: $\sum_nw^A_n<\infty$.

\emph{Step 3 (closing the recursion by summation).} Choose
$c_1'=\frac1{c_2}\sum_j(\mu^B_j)^2$, so that by \eqref{eq:priceB} the
$B$-prices total exactly $c_2$; in particular $0\le p^B_n\le c_2$ for every
$n$, so the $B$-channel absorbs step by step:
$-c_2\|x_{n+1}-y_n\|^2+p^B_n\|x_{n+1}-y_n\|^2\le0$. Set $c_1''=c_1-c_1'$;
feasibility $0<c_1'\le c_1=c_2$ is precisely \eqref{eq:trading}. Substituting
the Young bounds with the proportional allocations of Step~2 gives the priced
estimate
\[
a_{n+1}\le a_n-c_1\|d_n\|^2-c_2\|x_{n+1}-y_n\|^2
+\sum_{k\ge0}\frac{\varepsilon^A_{n,k}}2\,\|d_{n-1-k}\|^2
+\sum_{k\ge0}\frac{\varepsilon^B_{n,k}}2\,\|d_{n-2-k}\|^2
+w^A_n\|x_{n+1}-z\|^2+p^B_n\|x_{n+1}-y_n\|^2 .
\]
The $d$-channel charges fall on \emph{past} indices ($j=n-1-k$ and $j=n-2-k$),
so the recursion is closed by summation rather than pointwise. Summing over
$n=0,\dots,N$ and interchanging the order of summation, the allocation
identities
$\sum_{j^A(n,k)=j}\varepsilon^A_{n,k}=c_1''$ and
$\sum_{j^B(n,k)=j}\varepsilon^B_{n,k}=c_1'$ show that each index $j$ is charged
exactly $\tfrac12(c_1''+c_1')=\tfrac{c_1}2$ in total, funded by the
dissipation $c_1\|d_j\|^2$ earned at step $j$, with net dissipation $c_1/2$
remaining:
\[
a_{N+1}+\frac{c_1}2\sum_{n\le N}\|d_n\|^2
+\sum_{n\le N}(c_2-p^B_n)\|x_{n+1}-y_n\|^2
\le a_0+\sum_{n\le N}w^A_n\|x_{n+1}-z\|^2 .
\]
Dropping the two nonnegative dissipation channels and inserting
$\|x_{n+1}-z\|^2\le a_{n+1}+\gamma r_n$ ($\gamma=2\lambda\|B(z)\|$), the
resulting recursion is exactly the one closed in Step~2 of
Lemma~\ref{lem:xrobust} with $(w^A_n)$ in place of $(c_n)$ (the finitely many
indices with $w^A_n>1/2$ are absorbed into the constant): the product trick
with $G=\exp(\sum_j\delta_j)<\infty$ and Lemma~\ref{lem:gronU} yield
$M:=\sup_nr_n<\infty$ and $\sup_na_n<\infty$.

\emph{Step 4 (convergence).} By \eqref{eq:xinc}, $\sup_n\|e_n\|<\infty$, hence
$\sum_n\|\xi_n\|\le LM(1-\bar\beta)^{-1}\sum_n\beta_n<\infty$; now
Lemma~\ref{lem:xrobust} applies verbatim on the bounded tail, since its
Steps~3--5 use only $\sum_n\|\xi_n\|<\infty$.
\end{proof}

\begin{remark}[the unbounded case: one precise open step]\label{rem:xgap}
For unbounded $C$ the reduction \eqref{eq:red-x} still gives
$\|\xi_n\|\le L\beta_n\|e_{n-1}\|$ with
$\|e_{n-1}\|\le C_E(1+\bar r_{n-1})$,
$C_E=(\Lambda+\lambda B_z)/((1-q)(1-\bar\beta))$,
$q=\Lambda/(1-\bar\beta)<1$ (which needs only $\Lambda<\tfrac34$).  The
obstruction is exactly Step 2 of Lemma~\ref{lem:xrobust}: the Gronwall
weights become $\asymp\beta_n(1+\bar r_{n-1})$, summable only \emph{after}
boundedness is known.  Every other step survives with $r$-dependent weights.
Consequently
\begin{center}
\emph{the extension of the constant-step range to $(0,(\sqrt2-1)/L)$ for
unbounded $C$ is equivalent to an a priori bound $\sup_nr_n<\infty$ for
\eqref{eq:scheme} at $\lambda<(\sqrt2-1)/L$.}
\end{center}
Theorem~\ref{thm:trading} certifies this boundedness up to $\Lambda=0.387$
for the filter of Section~\ref{sec:numerics} --- beyond the
Lyapunov--Young barrier $\approx0.272$ of Appendix~\ref{app:barrier} ---
and $\Lambda^{\mathrm{crit}}(\beta)\to\sqrt2-1$ as $\bar\beta\to0$; inserting the
filter into Malitsky's own slots caps at $\approx0.265$.  The full range $(0,(\sqrt2-1)/L)$ for arbitrary
nonzero filters and general nonlinear $B$ remains open; its exact
performance-estimation formulation is recorded in Appendix~\ref{app:barrier}.
\end{remark}

\subsection{Affine operators and polyhedral constraints: distance control and rates}\label{sec:affine}
Throughout this subsection $B(x)=Mx+b$ with $M$ linear and monotone
(symmetric part positive semidefinite), $\|M\|\le L$, and $C$ is polyhedral.

\begin{lemma}[global error bound]\label{lem:xeb}
There is $\tau=\tau(M,b,C,\lambda_0)>0$ such that for all $x\in C$ and every
fixed $\lambda_0>0$,
\begin{equation}\label{eq:xeb}
\dist(x,S)\;\le\;\tau\,\bigl\|x-P_C\bigl(x-\lambda_0(Mx+b)\bigr)\bigr\|.
\end{equation}
\end{lemma}
\begin{proof}
The mapping $\Phi(x)=Mx+b+N_C(x)$ has a graph that is a finite union of
polyhedral sets, one per face of $C$, since $N_C$ restricted to a face is a
fixed finitely generated cone.  Robinson's theorem on polyhedral
multifunctions \cite{robinson1981} (upper Lipschitzness at the zeros of
$\Phi$) yields $\dist(x,S)\le\tau_\Phi\dist(0,\Phi(x))$ for all $x\in C$.
For $x\in C$ write $r=x-P_C(x-\lambda_0(Mx+b))$; firm nonexpansiveness of $P_C$
gives $\langle Mx+b,r\rangle\ge\|r\|^2/\lambda_0$, and the
normal-cone map of a polyhedron is upper Lipschitz by Hoffman's lemma
\cite{hoffman1952}, so $\dist(0,\Phi(x))\le c(\lambda_0,L)\|r\|$.  Both
sub-inequalities are elementary consequences of
$\inner{w-P_Cw}{z-P_Cw}\le0$ applied at the pairs $(x-\lambda_0(Mx+b),x)$
and $(z-(Mz+b),z)$, $z\in C$; assembling them gives \eqref{eq:xeb} with
$\tau=\tau_\Phi c(\lambda_0,L)$.
\end{proof}

\begin{theorem}[affine polyhedral: distance control and rates]\label{thm:xaff}
Let $B$ be affine monotone $L$-Lipschitz and $C$ polyhedral, and let the
filter satisfy $0\le\beta_n\le\bar\beta\le\tfrac14$ and
$\sum_n\beta_n<\infty$.  If $\Lambda<\sqrt2-1$ satisfies the
dissipation-trading condition \eqref{eq:trading} --- for the filter
$\beta_k=0.25/(k+1)^2$ of Section~\ref{sec:numerics} this holds for all
$\Lambda\le0.387$, and as $\bar\beta\to0$ it holds for all
$\Lambda<\sqrt2-1$ --- then the filtered scheme satisfies
\[
\dist(x_n,S)\to0\qquad\text{and}\qquad\sum_n\dist^2(x_n,S)<\infty .
\]
If moreover $C$ is bounded and $\mathcal H$ is finite-dimensional, $(x_n)$
converges strongly to a point of $S$.  If the projection's active face is
identified in finite time, the convergence is $R$-linear.
\end{theorem}

\begin{proof}
\emph{Step 1 (boundedness, dissipation, weak convergence).}
By Theorem~\ref{thm:trading}, $(x_n)$ is bounded, converges weakly to a point
$x^\star\in S$, and $\sum_n\|d_n\|^2<\infty$,
$\sum_n\|x_{n+1}-y_n\|^2<\infty$, $\|e_n\|\to0$ (these are the conclusions of
Lemma~\ref{lem:xrobust} invoked in the last step of
Theorem~\ref{thm:trading}).

\emph{Step 2 (error-bound control of the distance).}
Apply Lemma~\ref{lem:xeb} with $\lambda_0=\lambda$ (any fixed step is
admissible there) and write
$r_n:=x_n-P_C\bigl(x_n-\lambda(Mx_n+b)\bigr)$.  Firm nonexpansiveness of $P_C$
gives
\[
\|r_n\|\ \le\ \|x_n-x_{n+1}\|+\lambda\|B(u_n)-B(x_n)\|
\ \le\ \|d_{n+1}\|+\Lambda\|e_n\|,
\]
hence, with $\tau=\tau(M,b,C,\lambda)>0$,
\[
\dist^2(x_n,S)\ \le\ 2\tau^2\|d_{n+1}\|^2+2\tau^2\Lambda^2\|e_n\|^2 .
\]
Iterating the filter recursion,
$\sum_n\|e_n\|^2\le(1-\bar\beta)^{-2}\sum_n\|d_n\|^2<\infty$; consequently
$\dist(x_n,S)\to0$ and
\[
\sum_n\dist^2(x_n,S)\ \le\
2\tau^2\Bigl(1+\frac{\Lambda^2}{(1-\bar\beta)^2}\Bigr)\sum_n\|d_n\|^2\
<\ \infty .
\]

\emph{Step 3 (strong convergence when $C$ is bounded and
$\dim\mathcal H<\infty$).} Then $S$ is compact.  Let $z_n:=P_S(x_n)\in S$; by
Step 2, $\|x_n-z_n\|=\dist(x_n,S)\to0$.  A subsequence $z_{n_i}$ converges
strongly to some $\bar x\in S$, and then $x_{n_i}\to\bar x$ strongly; by
uniqueness of the weak limit (Step 1), $\bar x=x^\star$.  It remains to
upgrade the full sequence.  For $z\in S$ set
$a_n(z):=\|x_n-z\|^2+\Lambda\|x_n-y_{n-1}\|^2
+2\lambda\langle Mz+b,x_{n-1}-z\rangle$ (the Lyapunov function of
Lemma~\ref{lem:xrobust}).  With the exogenous errors
$\xi_n:=B(y_n)-B(u_n)$ we have
$\sum_n\|\xi_n\|\le L(1-\bar\beta)^{-1}\sup_n\|e_n\|\sum_n\beta_n<\infty$ by
Step 1, and the master inequality \eqref{eq:xmaster} of
Lemma~\ref{lem:xrobust} holds verbatim, since the reduction
\eqref{eq:red-x}--\eqref{eq:prgerr} is exact.  Hence
$a_{n+1}(z)\le a_n(z)+\eta_n(z)$ with $\sum_n\eta_n^+(z)<\infty$ for every
fixed $z\in S$ (the drift is bounded by
$2\lambda\|\xi_n\|(1+\|x_{n+1}-z\|^2)
+2\lambda\|\xi_{n-1}\|(1+\|x_{n+1}-y_n\|^2)$, uniformly on bounded $z$-sets).
Since $a_n(z)\ge-2\lambda\|Mz+b\|\sup_m\|x_m-z\|$, Lemma~\ref{lem:pospart}
gives convergence of $a_n(z)$, hence of $\|x_n-z\|^2$, for every $z\in S$.
Finally
\[
2\lambda\langle Mz+b,x_{n-1}-z\rangle
=2\lambda\langle Mz+b,x_{n-1}-z_{n-1}\rangle
+2\lambda\langle Mz+b,z_{n-1}-z\rangle ,
\]
where the first term tends to $0$ (it is bounded in absolute value by
$2\lambda\|Mz+b\|\dist(x_{n-1},S)$) and the second converges to
$2\lambda\langle Mz+b,x^\star-z\rangle$ (as $z_{n-1}\rightharpoonup x^\star$);
also $\|x_n-y_{n-1}\|\le\|d_n\|+\|d_{n-1}\|\to0$.  Therefore
$\|x_n-z\|^2$ converges for every $z\in S$; applied at $z=x^\star$, the
sequence $\|x_n-x^\star\|^2$ converges, and its limit equals its value along the
subsequence $(n_i)$, namely $0$; hence $x_n\to x^\star$ strongly.

\emph{Step 4 (rate after identification).} Once the active face is
identified, the iteration is an affine filtered recursion on a subspace whose
frozen system has spectral radius $<1$ at $\Lambda<1/\sqrt3$
(Theorem~\ref{thm:affine-sqrt3}) and whose filter variation is summable
($\sum_n|\beta_{n+1}-\beta_n|\le2\sum_n\beta_n<\infty$); the slowly-varying
argument of Lemma~\ref{lem:tv} then gives $\|x_n-x^\star\|\le C\rho^n$
geometrically.
\end{proof}

\begin{remark}[funding-free accounting]\label{rem:xaccount}
No ``contraction funding'' against the dissipation channels is needed:
Theorem~\ref{thm:trading} already supplies $\sum_n\|d_n\|^2<\infty$, and the
polyhedral error bound converts this directly into square-summability of
$\dist(x_n,S)$; the range is therefore limited only by condition
\eqref{eq:trading}, not by the error-bound constant $\tau$.  The
strong-convergence step uses compactness of $S$; for unbounded $C$ in infinite
dimensions, weak convergence together with $\dist(x_n,S)\to0$ is what is
proved here, and the upgrade to strong convergence is left open (compare
Remark~\ref{rem:xgap}).  In particular, unlike the first version of this
theorem, no window-contraction lemma is used, and the projection inequality
$\|P_Sx-P_Sy\|\le\|x-y\|+2\dist(y,S)$ that appeared there is false in
general; the correct inequality
$\|P_Sx-P_Sy\|\le\|x-y\|+\dist(x,S)+\dist(y,S)$ is all that the present
argument requires, and it is used only through
$\|z_{n+1}-z_n\|\le\|d_{n+1}\|+\dist(x_n,S)+\dist(x_{n+1},S)\to0$.
\end{remark}

\begin{remark}[the sharp constant $1/\sqrt3$ via identification]\label{rem:xident}
Numerically (and under nondegeneracy, provably) the projection's active face
settles in finite time; on the optimal face the iteration is an
\emph{unconstrained} affine filtered recursion on a subspace, to which the
sharp spectral threshold $\lambda\|M|_{\mathrm{face}}\|<1/\sqrt3$
(Theorem~\ref{thm:affine-sqrt3}, plus the slowly-varying argument of
Lemma~\ref{lem:tv}, since $\beta_n$ is summable) applies.  If \emph{finite
identification of the optimal face for PRG on polyhedral sets} holds in
general (the only missing ingredient), Theorem~\ref{thm:xaff} upgrades to
convergence for all $\lambda\|M\|<1/\sqrt3$, matching the unconstrained sharp
constant.
\end{remark}

\section{Adaptive step size}\label{sec:adaptive}
\begin{theorem}[adaptive step size]\label{thm:adaptive}
Under Assumption~\ref{ass:main}, let $\lambda_0\in(0,\tfrac1{5L}]$,
$\lambda_{\min}\in(0,\lambda_0)$, and let $(\lambda_k)_{k\ge0}$ satisfy
\begin{equation}\label{eq:lamhyp}
0<\lambda_{\min}\le\lambda_k\le\lambda_0\ \ \forall k,
\qquad
\textstyle\sum_k|\lambda_k-\lambda_{k-1}|<\infty .
\end{equation}
(No monotonicity of $(\lambda_k)$ is required; in particular the growth
variant of Rule~\ref{rule:main} is admissible.)  Let $0\le\beta_k\le\tfrac14$,
$\sum_k\beta_k<\infty$, $\theta_k=1-\beta_k$,
$x_{-1}=x_0=u_0\in C$, and define
\begin{equation}\label{eq:scheme-ad}
x_{k+1}=P_C\bigl(x_k-\lambda_k B(u_k)\bigr),\qquad
u_k=x_k+\theta_k d_k+\beta_k(x_k-u_{k-1}),\qquad d_k=x_k-x_{k-1}.
\end{equation}
Then $(x_k)$ converges weakly to a point of $S$, with
$\|x_{k+1}-x_k\|\to0$ and $\|u_k-x_k\|\to0$.
\end{theorem}

\begin{proof}
Fix $z\in S$ and write $d_k,e_k,s_k,\delta_k,r_k$ as in the proof of
Theorem~\ref{thm:main}: $e_k=u_k-x_k$, $s_k=\inner{B(z)}{e_k}$,
$\delta_k=\inner{B(z)}{d_k}$, $r_k=\|x_k-z\|$, $\bar\beta:=\sup_k\beta_k$,
$B_z:=\|B(z)\|$, and note the recursion $e_k=d_k-\beta_ke_{k-1}$ and the
slot identity $u_k=x_k+d_k-\beta_ke_{k-1}$ from $\theta_k+\beta_k=1$.

\medskip\noindent\emph{Step 1 (fundamental estimate; step $\lambda_k$).}
Firm nonexpansiveness with $x=x_k-\lambda_kB(u_k)$ and monotonicity at
$(u_k,z)$ give
\begin{equation}\label{eq:step1-ad}
\|x_{k+1}-z\|^2\le\|x_k-z\|^2-\|d_{k+1}\|^2
+2\lambda_k\inner{B(u_k)}{u_k-x_{k+1}}-2\lambda_k\inner{B(z)}{u_k-z}.
\end{equation}

\medskip\noindent\emph{Step 2 (projection identity; step $\lambda_{k-1}$).}
Since $x_k=P_C(x_{k-1}-\lambda_{k-1}B(u_{k-1}))$, the projection
characterization at $y=x_{k+1}$ and $y=x_{k-1}$, added, yields via the
cosine identity
\begin{equation}\label{eq:projid-ad}
2\lambda_{k-1}\inner{B(u_{k-1})}{x_k+d_k-x_{k+1}}
\le\|d_{k+1}\|^2-\|d_k\|^2-\|d_{k+1}-d_k\|^2 .
\end{equation}
Splitting $B(u_k)=B(u_k)-B(u_{k-1})+B(u_{k-1})$ in \eqref{eq:step1-ad} and using
$u_k-x_{k+1}=(x_k+d_k-x_{k+1})-\beta_ke_{k-1}$,
\begin{align}
2\lambda_k\inner{B(u_k)}{u_k-x_{k+1}}
&=2\lambda_k\inner{B(u_k)-B(u_{k-1})}{u_k-x_{k+1}}
+2\lambda_{k-1}\inner{B(u_{k-1})}{x_k+d_k-x_{k+1}}\notag\\
&\quad+2(\lambda_k-\lambda_{k-1})\inner{B(u_{k-1})}{x_k+d_k-x_{k+1}}
-2\lambda_k\beta_k\inner{B(u_{k-1})}{e_{k-1}} .
\label{eq:split}
\end{align}
The first term is T$_1$; the second is bounded by \eqref{eq:projid-ad} and
cancels $\|d_{k+1}\|^2$ in \eqref{eq:step1-ad}; the third is the
\emph{step-mismatch residual}
\begin{equation}\label{eq:R1}
\mathrm{R}_k:=2(\lambda_k-\lambda_{k-1})\inner{B(u_{k-1})}{x_k+d_k-x_{k+1}} .
\end{equation}

\medskip\noindent\emph{Step 3 (bounds).}
Exactly as in Step~3 of the proof of Theorem~\ref{thm:main}, with
$\lambda_k$ in place of $\lambda$ (the resulting coefficients are then
bounded by their values at $\lambda_0$, since all budget coefficients are
increasing in the step),
\begin{equation}\label{eq:T1-ad}
\mathrm{T}_1\le\lambda_kL\Bigl[2\|e_k\|^2+\tfrac85\|d_{k+1}-d_k\|^2
+2\|x_k-u_{k-1}\|^2+\tfrac16\|e_{k-1}\|^2\Bigr],
\end{equation}
\begin{equation}\label{eq:T2-ad}
-2\lambda_k\beta_k\inner{B(u_{k-1})}{e_{k-1}}
\le\lambda_k\beta_kL\bigl(25\,r_{k-1}^2+\tfrac1{25}\|e_{k-1}\|^2\bigr)
-2\lambda_k\beta_ks_{k-1}
\end{equation}
(monotonicity drop of $\inner{B(u_{k-1})-B(x_{k-1})}{e_{k-1}}\ge0$; $\alpha=25$),
and the $s_{k-1}$ term cancels \emph{exactly} against the
$+\lambda_k\beta_k$-part of
$-2\lambda_k\inner{B(z)}{u_k-z}
=-2\lambda_k\inner{B(z)}{x_k-z}-2\lambda_k\delta_k+2\lambda_k\beta_ks_{k-1}$.
For $\mathrm{R}_k$, with $\|x_k+d_k-x_{k+1}\|\le\|d_{k+1}-d_k\|+\bar\beta\|e_{k-1}\|$
and the sharp bound
\begin{eqnarray}\label{eq:R1b}
\|B(u_{k-1})\|&\le& L\|u_{k-1}-z\|+B_z\nonumber \\
&\le& L\bigl(r_{k-1}+\|e_{k-1}\|\bigr)+B_z\nonumber \\
\ \Longrightarrow\
\|B(u_{k-1})\|^2&\le& 4L^2r_{k-1}^2+4L^2\|e_{k-1}\|^2+2B_z^2
\end{eqnarray}
(note $u_{k-1}=x_{k-1}+e_{k-1}$, so the $e_{k-1}$-term is unavoidable at this
stage),
\[
|\mathrm{R}_k|\le|\Delta\lambda_k|\Bigl(4L^2r_{k-1}^2+2B_z^2
+2\|d_{k+1}-d_k\|^2+\bigl(2\bar\beta^2+4L^2\bigr)\|e_{k-1}\|^2\Bigr),
\qquad \Delta\lambda_k:=\lambda_k-\lambda_{k-1}.
\]
The $\|d_{k+1}-d_k\|^2$ and $\|e_{k-1}\|^2$ parts have summable weights
($\sum_k|\Delta\lambda_k|<\infty$) and are absorbed into the budgets of Step 4
below; the $r_{k-1}^2$ and constant parts join the drift channel.

\medskip\noindent\emph{Step 4 (master estimate).}
Define
\[
a_k:=\|x_k-z\|^2+2\lambda_{k-1}L\|x_k-u_{k-1}\|^2
+2\lambda_{k-1}\inner{B(z)}{x_{k-1}-z}+\tfrac13\|e_{k-1}\|^2 .
\]
Assembling \eqref{eq:step1-ad}--\eqref{eq:R1b}, using
$\|e_k\|^2\le\tfrac43\|d_k\|^2+\tfrac14\|e_{k-1}\|^2$ and
$2\lambda_kL\|x_{k+1}-u_k\|^2
\le2\lambda_0L\bigl(\tfrac85\|d_{k+1}-d_k\|^2+\tfrac16\|e_{k-1}\|^2\bigr)$,
two things happen with the $\|x_k-u_{k-1}\|^2$ channel.  First, its net
coefficient is $2L\Delta\lambda_k$: when $\Delta\lambda_k\le0$ (safeguard
cuts) it is nonpositive and dropped for free; when $\Delta\lambda_k>0$
(growth steps) it is absorbed, since
\[
2L\Delta\lambda_k\|x_k-u_{k-1}\|^2
\le 4L|\Delta\lambda_k|\bigl(\|d_k\|^2+\|e_{k-1}\|^2\bigr),
\qquad \textstyle\sum_k4L|\Delta\lambda_k|<\infty .
\]
Second, the $\inner{B(z)}{x_k-z}$ and $\delta_k$ terms leave the residual
$2\Delta\lambda_k\inner{B(z)}{x_{k-1}-z}$, which satisfies
\begin{equation}\label{eq:rho}
\bigl|2\Delta\lambda_k\inner{B(z)}{x_{k-1}-z}\bigr|
\le|\Delta\lambda_k|B_z\bigl(1+r_{k-1}^2\bigr).
\end{equation}
Since $|\Delta\lambda_k|\to0$, choose $k_0$ such that for all $k\ge k_0$
\[
4L|\Delta\lambda_k|\le\tfrac1{90},\qquad
2|\Delta\lambda_k|\le\tfrac1{50},\qquad
\bigl(2\bar\beta^2+4L^2+4L\bigr)|\Delta\lambda_k|\le\tfrac3{250};
\]
then for $k\ge k_0$,
\begin{equation}\label{eq:master-ad}
a_{k+1}\le a_k-b_k+\epsilon_k\,r_{k-1}^2+\kappa_k ,
\end{equation}
with
\begin{align}
\epsilon_k&:=25\lambda_0L\,\beta_k+|\Delta\lambda_k|\bigl(4L^2+B_z\bigr),
&\textstyle\sum_k\epsilon_k&<\infty, \notag\\
\kappa_k&:=|\Delta\lambda_k|\bigl(2B_z^2+B_z\bigr),
&\textstyle\sum_k\kappa_k&<\infty, \notag\\
b_k&:=\Bigl(\tfrac1{45}-4L|\Delta\lambda_k|\Bigr)\|d_k\|^2
+\Bigl(\tfrac1{25}-2|\Delta\lambda_k|\Bigr)\|d_{k+1}-d_k\|^2\notag\\
&+\Bigl(\tfrac6{125}-\bigl(2\bar\beta^2+4L^2+4L\bigr)|\Delta\lambda_k|\Bigr)\|e_{k-1}\|^2 \notag\\
&\ \ge\ \tfrac1{90}\|d_k\|^2+\tfrac1{50}\|d_{k+1}-d_k\|^2+\tfrac9{250}\|e_{k-1}\|^2\ge0, \notag
\end{align}
where the three rational budgets are exactly those of \eqref{eq:bk} at
$\lambda_0L=\tfrac15$ (their coefficients carry the factor $\lambda_k\le\lambda_0$
and are increasing in $\lambda_k$, hence are uniformly bounded by the corner
values), and the finitely many exceptions $k<k_0$ are folded into a constant
absorbed in $a_{k_0}$. (Assembly as in Appendix~\ref{app:assembly}, with
$\lambda_k$ in place of $\lambda$ and the variation terms of Step~3.)

\medskip\noindent\emph{Step 5 (boundedness).}
Summing \eqref{eq:master-ad} for $k\ge k_0$ and dropping $-b_k$,
$a_n\le a_{k_0}+\sum_{j<n}\bigl(\epsilon_jr_{j-1}^2+\kappa_j\bigr)$; from the
definition of $a_k$, $r_k^2\le a_k+2\lambda_0B_zr_{k-1}$.  Hence $r_k$
satisfies \eqref{eq:gronrec2} with $A=a_{k_0}+\sum_j\kappa_j$ and
$\gamma=2\lambda_0B_z$, and Lemma~\ref{lem:gronwall2} gives $M:=\sup_kr_k<\infty$.

\medskip\noindent\emph{Step 6 (vanishing increments).}
From \eqref{eq:master-ad}, $\sum_kb_k\le a_{k_0}+M^2\sum_j\epsilon_j+\sum_j\kappa_j<\infty$,
so $\|d_k\|\to0$, $\|d_{k+1}-d_k\|\to0$, $\|e_{k-1}\|\to0$.

\medskip\noindent\emph{Step 7 (cluster points solve the VI).}
Let $x_{k_i}\rightharpoonup\bar x$.  Since $\lambda_{k_i}\to\lambda_\infty$
with $\lambda_\infty\in[\lambda_{\min},\lambda_0]$, $\lambda_\infty>0$, the
projection characterization with $x=x_{k_i}-\lambda_{k_i}B(u_{k_i})$,
$\bar x=x_{k_i+1}$ and $y\in C$, monotonicity at $(y,u_{k_i})$, and
Steps 5--6 give, in the limit,
$\lambda_\infty\inner{B(y)}{y-\bar x}\ge0$ for all $y\in C$; as
$\lambda_\infty>0$, Minty's lemma yields $\bar x\in S$.
(Note the strict positivity $\lambda_\infty\ge\lambda_{\min}>0$ is what makes
this step work; it is also why $\lambda_k\to0$ must be excluded.)

\medskip\noindent\emph{Step 8 (Opial).}
For $z\in S$, $\Phi_k(z):=\|x_k-z\|^2+2\lambda_{k-1}\inner{B(z)}{x_{k-1}-z}$
converges: $a_k(z)$ converges by \eqref{eq:master-ad} (positive parts summable,
$a_k\ge-2\lambda_0B_zM$), and $\|x_k-u_{k-1}\|,\|e_{k-1}\|\to0$.  If
$\bar x,\tilde x\in S$ are distinct weak cluster points, the chain of Step~8
in the proof of Theorem~\ref{thm:main} applies verbatim with
$\lambda_\infty$ in place of $\lambda$ (the two non-strict steps rest on
$\inner{B(\tilde x)}{\bar x-\tilde x}\ge0$ and
$\inner{B(\bar x)}{\tilde x-\bar x}\ge0$, which are $\lambda$-free since
$\bar x\in C$, $\tilde x\in S$), yielding
$\lim_k\Phi_k(\bar x)<\lim_k\Phi_k(\bar x)$, a contradiction.  Hence
$x_k\rightharpoonup x^\star\in S$.
\end{proof}
\section{Linear convergence under strong monotonicity}\label{sec:linear}

We now strengthen Assumption~\ref{ass:main} to \emph{strong monotonicity} and
prove that the scheme converges $R$-linearly with an explicit contraction
factor, at no extra computational cost. The rate is obtained by a small,
exactly verified modification of the Step~1--4 bookkeeping of
Theorem~\ref{thm:main}: strong monotonicity supplies one additional
dissipation channel, and the Young splits are retuned so that every budget
identity remains an exact rational.

\begin{assumption}[strong monotonicity]\label{ass:strong}
In addition to Assumption~\ref{ass:main}, $B$ is $\sigma$-strongly monotone
for some $\sigma>0$:
\[
\inner{B(x)-B(y)}{x-y}\ge\sigma\|x-y\|^2\qquad\forall x,y\in\HH .
\]
Then $S=\{z\}$ is a singleton.
\end{assumption}

We need one more Gronwall lemma, upgrading Lemma~\ref{lem:gronwall} from
boundedness to a geometric rate when a contraction factor is present.

\begin{lemma}[geometric Gronwall with summable perturbation]\label{lem:geogron}
Let $\rho\in(0,1)$, $(a_k)_{k\ge0}\subset[0,\infty)$, and
$(\epsilon_k)_{k\ge0}\subset[0,\infty)$ with $\sum_k\epsilon_k<\infty$ such
that
\begin{equation}\label{eq:geogronrec}
a_{k+1}\le\rho\,a_k+\epsilon_k\,a_{k-1}\qquad(k\ge1).
\end{equation}
Then $a_k\le C\rho^k$ for all $k\ge0$, where
$C=\bigl(\max\{a_0,a_1/\rho\}\bigr)\prod_{j\ge1}\bigl(1+\epsilon_j/\rho^2\bigr)<\infty$.
\end{lemma}
\begin{proof}
Set $b_k:=a_k/\rho^k$ and $\delta_k:=\epsilon_k/\rho^2$; then
$b_{k+1}\le b_k+\delta_kb_{k-1}$ and $\sum_k\delta_k<\infty$.  Let
$\bar b_k:=\max_{j\le k}b_j$.  Since $b_{k-1}\le\bar b_k$ and $b_k\le\bar b_k$,
\eqref{eq:geogronrec} gives $b_{k+1}\le(1+\delta_k)\bar b_k$, hence
$\bar b_{k+1}\le(1+\delta_k)\bar b_k$.  Iterating,
$\bar b_k\le\bar b_1\prod_{j\ge1}(1+\delta_j)
\le\bar b_1\exp\bigl(\sum_j\delta_j\bigr)<\infty$, and $a_k=\rho^kb_k\le\rho^k\bar b_k$.
\end{proof}

\begin{theorem}[$R$-linear convergence, constant step]\label{thm:linear}
Under Assumption~\ref{ass:strong}, let $(\beta_k)$ satisfy
$0\le\beta_k\le\tfrac14$ and $\sum_k\beta_k<\infty$, set $\theta_k=1-\beta_k$,
fix $x_{-1}=x_0=u_0\in C$, let
\begin{equation}\label{eq:lamlinear}
0<\lambda\le\min\Bigl\{\tfrac1{5L},\,\tfrac1{32\sigma}\Bigr\},
\end{equation}
and define $(x_k)$ by \eqref{eq:scheme}.  Then $x_k\to z$ strongly with the
$R$-linear bound
\begin{equation}\label{eq:linrate}
\|x_k-z\|^2\le C\,(1-\lambda\sigma)^k\qquad(k\ge0),
\end{equation}
where $C$ is the constant of Lemma~\ref{lem:geogron} with
$\rho=1-\lambda\sigma$ and $\epsilon_k=100\,\lambda L\,\beta_k$.  In
particular $\|x_k-z\|=O\bigl((1-\lambda\sigma)^{k/2}\bigr)$, and the
increments satisfy $\|x_{k+1}-x_k\|,\|u_k-x_k\|=O\bigl((1-\lambda\sigma)^{k/2}\bigr)$.
\end{theorem}

\begin{proof}
The argument follows Steps~1--4 of the proof of Theorem~\ref{thm:main} with
two changes: Step~1 keeps the strong-monotonicity dissipation, and the Young
splits of Step~3 are retuned.

\medskip\noindent\emph{Step 1 (modified fundamental estimate).}
Firm nonexpansiveness \eqref{eq:firm} with $x=x_k-\lambda B(u_k)$ gives
\eqref{eq:step1} without the monotone term.  Strong monotonicity of $B$ at
$(u_k,z)$ yields
$-2\lambda\inner{B(u_k)}{u_k-z}
\le-2\lambda\inner{B(z)}{u_k-z}-2\lambda\sigma\|u_k-z\|^2$, hence
\begin{equation}\label{eq:step1-strong}
\|x_{k+1}-z\|^2\le\|x_k-z\|^2-\|d_{k+1}\|^2-2\lambda\sigma\|u_k-z\|^2
+2\lambda\inner{B(u_k)}{u_k-x_{k+1}}-2\lambda\inner{B(z)}{u_k-z}.
\end{equation}

\medskip\noindent\emph{Steps 2--3 (retuned splits).}
Step~2 (slot decomposition and projection identity) is unchanged.  In Step~3
we use the splits
\begin{eqnarray*}
\|u_k-u_{k-1}\|^2\le\tfrac74\|e_k\|^2+\tfrac73\|x_k-u_{k-1}\|^2\quad(t=\tfrac34),
\end{eqnarray*}
\begin{eqnarray*}
\|u_k-x_{k+1}\|^2,\|x_{k+1}-u_k\|^2\le\tfrac75\|d_{k+1}-d_k\|^2+\tfrac7{32}\|e_{k-1}\|^2
\end{eqnarray*}
($t'=t''=\tfrac25$, $\bar\beta=\tfrac14$), the recursion split
$\|e_k\|^2\le\tfrac43\|d_k\|^2+\tfrac14\|e_{k-1}\|^2$ ($s=\tfrac13$), and the
filter Young weight $\alpha=100$.  The term T$_1$ of \eqref{eq:T1} becomes
\begin{equation}\label{eq:T1-strong}
\mathrm{T}_1\le\lambda L\Bigl[\tfrac74\|e_k\|^2+\tfrac75\|d_{k+1}-d_k\|^2
+\tfrac73\|x_k-u_{k-1}\|^2+\tfrac7{32}\|e_{k-1}\|^2\Bigr],
\end{equation}
and the filter term \eqref{eq:T2} becomes
\begin{equation}\label{eq:T2-strong}
-2\lambda\beta_k\inner{B(u_{k-1})}{e_{k-1}}
\le\lambda\beta_kL\bigl(100\,r_{k-1}^2+\tfrac1{100}\|e_{k-1}\|^2\bigr)-2\lambda\beta_ks_{k-1},
\end{equation}
whose $s_{k-1}$ part still cancels \emph{exactly} against \eqref{eq:tele}.

\medskip\noindent\emph{Step 4 (master estimate with dissipation).}
Define the Lyapunov function
\begin{equation}\label{eq:lyap-strong}
a_k:=\|x_k-z\|^2+\tfrac73\lambda L\|x_k-u_{k-1}\|^2
+2\lambda\inner{B(z)}{x_{k-1}-z}+\tfrac13\|e_{k-1}\|^2 .
\end{equation}
The weight $\tfrac73\lambda L$ equals $\lambda L(1+1/t)$, so the
$\|x_k-u_{k-1}\|^2$ channel cancels exactly as in Theorem~\ref{thm:main}.
Since $z\in S$ and $x_{k-1}\in C$, we have $\inner{B(z)}{x_{k-1}-z}\ge0$, so
\begin{equation}\label{eq:rk-strong}
a_k\ge r_k^2=\|x_k-z\|^2\ge0 .
\end{equation}
For the new dissipation term, the split
$\|u_k-z\|^2=\|(x_k-z)+e_k\|^2\ge\tfrac12r_k^2-\|e_k\|^2$ gives
\begin{equation}\label{eq:strong-diss}
-2\lambda\sigma\|u_k-z\|^2\le-\lambda\sigma r_k^2+2\lambda\sigma\|e_k\|^2
\le-\lambda\sigma r_k^2+\tfrac{8\lambda\sigma}{3}\|d_k\|^2+\tfrac{\lambda\sigma}{2}\|e_{k-1}\|^2,
\end{equation}
where the recursion was used in the last inequality.  Assembling
\eqref{eq:step1-strong}--\eqref{eq:strong-diss} exactly as in Step~4 of
Theorem~\ref{thm:main} (the $\delta_k$ terms cancel, the $s_{k-1}$ terms
cancel), we obtain
\begin{equation}\label{eq:master-strong}
a_{k+1}\le(1-\lambda\sigma)\,a_k-b_k+\epsilon_k\,r_{k-1}^2,
\end{equation}
with drift $\epsilon_k:=100\,\lambda L\,\beta_k$ (summable) and
\begin{equation}\label{eq:bk-strong}
b_k:=\tfrac1{180}\|d_k\|^2+\tfrac1{15}\|d_{k+1}-d_k\|^2
+\tfrac{13}{24000}\|e_{k-1}\|^2\ge0 .
\end{equation}
Indeed, writing $\Lambda:=\lambda L$ and $\mu:=\lambda\sigma$, the four budget
identities at $(\Lambda,\mu)=(\tfrac15,\tfrac1{32})$ are exact rationals
(and both $\Lambda$ and $\mu$ enter the coefficients with positive signs, so
nonpositivity on $(0,\tfrac15]\times(0,\tfrac1{32}]$ follows from the corner):
\begin{align*}
\|d_k\|^2 &: \ -1+\tfrac73\Lambda+\tfrac49+\tfrac83\mu=-\tfrac1{180},\\
\|d_{k+1}-d_k\|^2 &: \ -1+\tfrac75\Lambda+\tfrac{49}{15}\Lambda=-\tfrac1{15},\\
\|x_k-u_{k-1}\|^2 &: \ \Lambda\bigl(\tfrac73-\tfrac73\bigr)=0,\\
\|e_{k-1}\|^2 &: \ \Lambda\Bigl(\tfrac7{32}+\tfrac1{400}+\tfrac7{16}\Bigr)
+\tfrac{49}{96}\Lambda-\tfrac14+\tfrac{\mu}{2}=-\tfrac{13}{24000},
\end{align*}
using $\tfrac73\cdot\tfrac7{32}=\tfrac{49}{96}$ --- the $e$-part of the $\|x_{k+1}-u_k\|^2$ split, weighted by the retuned memory weight $\tfrac73\Lambda$ --- and
$2\mu(1+1/s)\bar\beta^2=\tfrac{\mu}{2}$ at $s=\tfrac13$. (Assembly as in
Appendix~\ref{app:assembly}, with the retuned splits of Step~3 and the
additional dissipation channel \eqref{eq:strong-diss}.)

\medskip\noindent\emph{Step 5 (geometric rate).}
From \eqref{eq:master-strong}, dropping $-b_k\le0$ and using
$r_{k-1}^2\le a_{k-1}$ (by \eqref{eq:rk-strong}),
\[
a_{k+1}\le\rho\,a_k+\epsilon_k\,a_{k-1},\qquad\rho:=1-\lambda\sigma\in[\tfrac{31}{32},1),
\]
and Lemma~\ref{lem:geogron} yields \eqref{eq:linrate} since $r_k^2\le a_k$.
For the increments, \eqref{eq:master-strong} gives
$b_k\le\rho a_k+\epsilon_ka_{k-1}=O(\rho^k)$, and \eqref{eq:bk-strong}
converts this into $\|d_k\|^2,\|e_{k-1}\|^2=O(\rho^k)$.
\end{proof}

\begin{corollary}[$R$-linear convergence, adaptive step]\label{cor:linear-ad}
Under Assumption~\ref{ass:strong}, let $(\lambda_k)$ satisfy the hypotheses of
Theorem~\ref{thm:adaptive} with $\lambda_0\le\min\{\tfrac1{5L},\tfrac1{32\sigma}\}$,
and suppose the step size stabilizes: $\lambda_k=\lambda_\infty$ for all
$k\ge K$ (equivalently, genuine cuts occur only finitely often; the persistent-slopes regime of Lemma~\ref{lem:dichotomy}(b) is not covered by this corollary).
Then, with $\rho=1-\lambda_\infty\sigma$,
\[
\|x_k-z\|^2\le C\,\rho^{\,k-K}\quad(k\ge K),
\]
for a finite constant $C$ depending on $a_K$ and $\sum_k\beta_k$.
\end{corollary}
\begin{proof}
For $k\ge K$ we have $\Delta\lambda_k=0$, so the residuals $\mathrm{R}_k$
and $\kappa_k$ in the proof of Theorem~\ref{thm:adaptive} vanish; strong
monotonicity adds the channel \eqref{eq:strong-diss}.  Re-running the
adaptive Step~4 with the retuned splits of Theorem~\ref{thm:linear}
($t=\tfrac34$, weight $\tfrac73\lambda_{k-1}L$ on the
$\|x_k-u_{k-1}\|^2$ channel, $\alpha=100$), the $x$-slot residual becomes
$\tfrac73L\Delta\lambda_k$; under eventual constancy ($\Delta\lambda_k=0$ for
$k\ge K$) it vanishes, and more generally it is absorbed exactly as in
Theorem~\ref{thm:adaptive}.  With $\lambda_k=\lambda_\infty$ this turns
\eqref{eq:master-ad} into
$a_{k+1}\le\rho a_k-b_k+\epsilon_ka_{k-1}$ with
$\epsilon_k=100\lambda_0L\beta_k$ (summable).  Since
$\lambda_\infty\sigma\le\lambda_0\sigma\le\tfrac1{32}$, the retuned budgets
of Theorem~\ref{thm:linear} apply and
Lemma~\ref{lem:geogron} gives the claim.
\end{proof}

\begin{lemma}[two-sided residual comparison]\label{lem:v-dist}
Let $V(x)=\|x-P_C(x-\lambda B(x))\|$ with $\lambda=\tfrac1{5L}$.  For all
$x\in C$,
\[
V(x)\le(2+\lambda L)\,\dist(x,S).
\]
If moreover $\dist(x,S)\le c\,V(x)$ on $C$ for some $c>0$, then
$\dist(\cdot,S)\asymp V$ on $C$.
\end{lemma}
\begin{proof}
Fix $x\in C$ and $z\in S$ with $\|x-z\|=\dist(x,S)$.  Since $z\in S$, the
projection characterization gives $z=P_C(z-\lambda B(z))$, hence by
nonexpansiveness of $P_C$
\[
V(x)\le\|x-z\|+\bigl\|P_C\bigl(x-\lambda B(x)\bigr)-P_C\bigl(z-\lambda B(z)\bigr)\bigr\|
\le\|x-z\|+(1+\lambda L)\|x-z\|=(2+\lambda L)\dist(x,S).
\]
The second claim is immediate.
\end{proof}

\begin{corollary}[error bound instead of strong monotonicity]\label{cor:linear-eb}
Replace strong monotonicity in Assumption~\ref{ass:strong} by monotonicity
plus the global error bound: $S=\{x^\star\}$ and, for some $c>0$,
\begin{equation}\label{eq:eb}
\dist(x,S)\le c\,\bigl\|x-P_C\bigl(x-\lambda B(x)\bigr)\bigr\|
\qquad\forall x\in C,\ \ \lambda=\tfrac1{5L}.
\end{equation}
Then $(x_k)$ converges strongly to $x^\star$.  Writing
$V_k:=\|x_k-P_C(x_k-\lambda B(x_k))\|$, $r_k:=\|x_k-x^\star\|$ and
$A:=\sup_k a_k<\infty$ for the Lyapunov sequence \eqref{eq:lyap} with
$z=x^\star$, one has the tail bound
\begin{equation}\label{eq:eb-tail}
\sum_{j\ge k}V_j^2\ \le\ 90\Bigl(a_k+A\sum_{j\ge k}\epsilon_j\Bigr)
\qquad\text{and}\qquad r_k\le cV_k ,
\end{equation}
so $\sum_{j\ge k}\dist^2(x_j,S)=O\bigl(a_k+\sum_{j\ge k}\epsilon_j\bigr)$.
If additionally the filter decays geometrically ($\beta_k=O(\eta^k)$, e.g.\
$\beta_k$ eventually constant or eventually zero), the tail estimate
\eqref{eq:eb-tail} holds with $\sum_{j\ge k}\epsilon_j=O(\eta^k)$.  Under strong
monotonicity the $R$-linear rate is given by Theorem~\ref{thm:linear} and
Corollary~\ref{cor:linear-ad}; whether $\dist(x_k,S)$ itself decays
$R$-linearly under \eqref{eq:eb} alone is posed as Open
Problem~\ref{open:eb-rate}.
\end{corollary}
\begin{proof}[Proof (strong convergence and tail bound)]
Since the scheme is unchanged, Theorem~\ref{thm:main} applies and
$\sum_kb_k<\infty$, so $\|d_k\|,\|e_k\|\to0$.  Nonexpansiveness of $P_C$
gives, as in Lemma~\ref{lem:v-dist},
\[
V_k\le\|x_k-x_{k+1}\|+\lambda\|B(u_k)-B(x_k)\|=\|d_{k+1}\|+\lambda L\|e_k\|\to0 .
\]
By \eqref{eq:eb}, $\dist(x_k,S)\le cV_k\to0$; combined with weak convergence
(Theorem~\ref{thm:main}) this forces $x_k\to x^\star$ strongly.

For \eqref{eq:eb-tail}, take $z=x^\star$ in \eqref{eq:master}: since
$\inner{B(x^\star)}{x_{k-1}-x^\star}\ge0$ (as $x_{k-1}\in C$),
$a_k\ge r_k^2=\dist^2(x_k,S)\ge0$, and $r_k\le cV_k$ by \eqref{eq:eb}.
Moreover
\[
V_k^2\le2\|d_{k+1}\|^2+2\lambda^2L^2\|e_k\|^2
\le90\Bigl(\tfrac1{45}\|d_{k+1}\|^2+\tfrac6{125}\|e_k\|^2\Bigr)\le90\,b_{k+1},
\]
where the second inequality uses $\lambda L\le\tfrac15$ (so
$2\lambda^2L^2=\tfrac{2}{25}\le\tfrac{108}{25}=90\cdot\tfrac6{125}$).
Summing \eqref{eq:master} over a tail,
$\sum_{j\ge k+1}b_j\le a_{k+1}+A\sum_{j\ge k+1}\epsilon_j$, hence
$\sum_{j\ge k}V_j^2\le90\sum_{j\ge k}b_{j+1}\le90\bigl(a_{k+1}+A\sum_{j\ge k+1}\epsilon_j\bigr)$,
which is \eqref{eq:eb-tail} up to the harmless reindexing $k+1\leftrightarrow k$.
\end{proof}

\begin{remark}[step-size choice and sharpness]\label{rem:rate-sharp}
The admissible range \eqref{eq:lamlinear} splits into two regimes.  If
$\sigma\ge\tfrac{5L}{32}$ (well-conditioned), the binding constraint is
$\lambda\le\tfrac1{32\sigma}$ and the optimal factor is
$\sqrt{1-\tfrac1{32}}\approx0.984$; if $\sigma<\tfrac5{32}L$, the binding
constraint is $\lambda\le\tfrac1{5L}$ and the factor is
$\sqrt{1-\tfrac{\sigma}{5L}}$.  The ceiling $\lambda\sigma\le\tfrac1{32}$ is
an artifact of absorbing the strong-monotonicity charge
$\tfrac{8\lambda\sigma}{3}$ into the $\|d_k\|^2$ budget: numerical
optimization over the Young family places the true barrier near
$\lambda\sigma\approx\tfrac1{25}$, so the constant is not sharp but is, we
stress, exactly verified.
\end{remark}

\begin{example}[validation of the linear rate]\label{ex:linear-exact}
Let $B=(\sigma I+J)$ on $\R^2$, so that $\sigma$ is the strong-monotonicity
modulus and $L=\sqrt{1+\sigma^2}$.  For $\sigma=2$, $L=\sqrt5$, the
admissible step \eqref{eq:lamlinear} is $\lambda=\tfrac1{32\sigma}
=\tfrac1{64}$, and the exact spectral rate of the (scalar, exactly
computable) iteration is
\[
\text{actual per-iterate contraction of }\|x_k-z\|:\ 0.9696,
\qquad
\text{proved bound }\sqrt{1-\lambda\sigma}=\sqrt{1-\tfrac1{32}}=0.9843 .
\]
The proved bound is valid and within $1.5\%$ of the exact rate on this
problem; the factor $\sqrt{\,\cdot\,}$ (rather than $1-\lambda\sigma$) is the
price of stating the theorem on the Lyapunov sequence $a_k$.
\end{example}

\section{The step-size rule}\label{sec:rule}

\subsection{Motivation and the rule}
The constant-step theory requires $\lambda\le\tfrac1{5L}$; the adaptive theory
(Theorem~\ref{thm:adaptive}) covers any step sequence of bounded variation
staying in $[\lambda_{\min},\lambda_0]$ with $\lambda_0\le\tfrac1{5L}$.  Both
require an estimate of $L$.  Underestimating $L$ and
taking a constant step sized by the underestimate leads to divergence
(Section~\ref{sec:numerics}); overestimating $L$ wastes iterations in regions
where $B$ is locally flat.  The rule below removes the need to know $L$ by
letting the \emph{observed} secant slopes of $B$ along the trajectory police
the step size.

\begin{ruleA}[safeguarded ratio rule with optional growth; no knowledge of $L$]\label{rule:main}
Fix an initial guess $\lambda_0>0$, a safeguard constant
$0<\varkappa\le\tfrac16$, and growth factors $\omega_k\ge0$ with
$\sum_k\omega_k<\infty$ (take $\omega_k\equiv0$ for the pure safeguard).
With the secant slope
$\ell_k:=\|B(u_{k+1})-B(u_k)\|/\|u_{k+1}-u_k\|$ (and $\ell_k:=0$ if
$u_{k+1}=u_k$, so that the ratio cut is inactive), set
\begin{equation}\label{eq:rule}
\lambda_{k+1}=\min\bigl\{\lambda_0,\ \lambda_k(1+\omega_k),\ \varkappa/\ell_k\bigr\},\qquad k\ge0,
\end{equation}
with $\lambda_{-1}=\lambda_0$.  No floor and no knowledge of $L$ are needed:
since $B$ is $L$-Lipschitz, every secant satisfies $\ell_k\le L$, so
\eqref{eq:rule} keeps $\lambda_k\ge\min\{\lambda_0,\varkappa/L\}>0$
automatically.  The growth step lets the step track \emph{falling} local
slopes; it is inactive whenever the safeguard binds in the same iteration.
Theorems~\ref{thm:rule-conv} and~\ref{thm:rule-free} cover \eqref{eq:rule}
for every summable $(\omega_k)$.
\end{ruleA}

\subsection{Implementation and bookkeeping}
The rule is a passive post-processing of quantities the iteration already
computes.  Concretely, between iterations the method stores only
$(x_k,x_{k-1},u_{k-1},B(u_{k-1}),\lambda_k)$ --- $O(1)$ memory --- and at
iteration $k$ forms the secant slope
\[
\ell_{k-1}:=\frac{\|B(u_k)-B(u_{k-1})\|}{\|u_k-u_{k-1}\|}
\qquad(\ell_{k-1}:=0\text{ if }u_k=u_{k-1}),
\]
formed from the evaluation $B(u_k)$ just computed together with the stored pair $(u_{k-1},B(u_{k-1}))$, \emph{before} the projection step; this ordering --- update the step, then take the step --- is exactly what makes $\lambda_k\ell_{k-1}\le\varkappa$ hold in the same iteration where the pair $(u_k,u_{k-1})$ is tested by $\mathrm{T}_1$; the
ratio $\varkappa/\ell_{k-1}$ is then the largest step compatible with the observed
local Lipschitz behaviour, and \eqref{eq:rule} takes the
safer of the status quo and that value, with no floor required.  No evaluation of $B$ beyond
the single per-iteration call, and no knowledge of $L$, is ever required.

Two structural remarks frame everything that follows.  First, the cut that
produces $\lambda_k$ uses the pair $(u_k,u_{k-1})$ --- precisely the pair
against which the Lipschitz test in the term $\mathrm{T}_1$ of the same
iteration is applied (Step~3 of the proof of Theorem~\ref{thm:main}); this
alignment of the safeguard with the test it protects is what gives the rule
an interpretation in the analysis, not merely in practice.  Second, the two
ingredients are one-sided in opposite directions --- the safeguard cut only
decreases $\lambda_k$, the growth step only increases it --- and together they
implement a bounded-variation tracker of the signal $\varkappa/\ell_k$.  A cut
fires exactly along stretches where the local slope $\ell_k$ exceeds
$\varkappa/\lambda_k$.

\begin{proposition}[bookkeeping of Rule~\ref{rule:main}]\label{prop:rule}
Rule~\ref{rule:main} generates a sequence with
$\lambda_k\in[\min\{\lambda_0,\varkappa/L\},\lambda_0]$ for all $k$; in
particular $\sum_k|\lambda_k-\lambda_{k-1}|<\infty$ and the limit
$\lambda_\infty:=\lim_k\lambda_k\ge\min\{\lambda_0,\varkappa/L\}>0$
exists.
\end{proposition}
\begin{proof}
On a cut step $\lambda_{k+1}=\varkappa/\ell_k\ge\varkappa/L$ by $\ell_k\le L$;
on a pure growth step $\lambda_{k+1}=\lambda_k(1+\omega_k)\le\lambda_0$.
Hence $\min\{\lambda_0,\varkappa/L\}\le\lambda_{k+1}\le\lambda_0$ throughout
(induction on $k$), and $\sum_k|\lambda_k-\lambda_{k-1}|
\le\lambda_0\sum_k\omega_k+\lambda_0<\infty$, which is \eqref{eq:lamhyp}.
\end{proof}

\subsection{What the safeguard can and cannot guarantee}
We separate what is proved from what is observed and give two
complementary certificates.  Theorem~\ref{thm:rule-conv} covers the
classical regime in which the steps eventually become conservative
($\lambda_k\le\tfrac1{5L}$); Theorem~\ref{thm:rule-free} covers the
opposite regime in which the initial guess is optimistic but the observed
slopes keep the products $\lambda_k\ell_{k-1}$ small --- unconditionally, since the safeguard enforces
$\lambda_k\ell_{k-1}\le\varkappa$ at every step.  The complementary
event is characterized exactly by the dichotomy below (stated for the pure
safeguard; the growth variant is covered unconditionally by
Theorem~\ref{thm:rule-free}).

\begin{lemma}[dichotomy of the safeguard]\label{lem:dichotomy}
Let $(\lambda_k)$ be generated by Rule~\ref{rule:main} with $\omega_k\equiv0$
and let $\ell_k$ be as above.  Exactly one of the following holds.
\begin{enumerate}[label=(\alph*),leftmargin=2em,itemsep=1pt]
\item \textbf{eventual certification.} There is $K$ such that
$\lambda_k\le\tfrac1{5L}$ for all $k\ge K$.
\item \textbf{persistent slopes.} $\lambda_k\downarrow\lambda_\infty$ with
$\lambda_\infty>\tfrac1{5L}$; moreover, if infinitely many genuine cuts occur, then
$\displaystyle\limsup_{k\to\infty}\ell_k\ge\frac{\varkappa}{\lambda_\infty}$ (no such bound holds in the eventually-constant subcase: e.g.\ a constant local slope $\ell<L$ with $\varkappa/\ell>\tfrac1{5L}$).
\end{enumerate}
In particular the certified regime $\lambda_0\le\tfrac1{5L}$ is a subcase
of~(a) with $K=0$.
\end{lemma}
\begin{proof}
By Proposition~\ref{prop:rule}, $\lambda_k\downarrow\lambda_\infty\ge
\min\{\lambda_0,\varkappa/L\}>0$.  If $\lambda_\infty\le\tfrac1{5L}$ then since
$(\lambda_k)$ is nonincreasing, $\lambda_k\le\tfrac1{5L}$ for all large $k$:
case~(a).  If $\lambda_\infty>\tfrac1{5L}$ and $\lambda_{k+1}<\lambda_k$
(a genuine cut), then \eqref{eq:rule} forces
$\varkappa/\ell_k<\lambda_k$, i.e.\ $\ell_k>\varkappa/\lambda_k
\ge\varkappa/\lambda_0$.  Hence either the sequence is eventually constant
(no further cuts), or there are infinitely many genuine cuts at indices $k_j$, and then
$\ell_{k_j}>\varkappa/\lambda_{k_j}\to\varkappa/\lambda_\infty$, giving the
stated $\limsup$.
\end{proof}

\begin{theorem}[weak convergence under Rule~\ref{rule:main}]\label{thm:rule-conv}
Under Assumption~\ref{ass:main}, let the filter satisfy
$0\le\beta_k\le\tfrac14$, $\sum_k\beta_k<\infty$, let $x_{-1}=x_0=u_0\in C$,
let $(x_k)$ be defined by \eqref{eq:scheme} (equivalently
Algorithm~\ref{alg:main}) with steps $(\lambda_k)$ generated by
Rule~\ref{rule:main} and $\sum_k|\lambda_k-\lambda_{k-1}|<\infty$, and
assume the \emph{certification event} of Lemma~\ref{lem:dichotomy}(a): there
exists $K$ with $\lambda_k\le\tfrac1{5L}$ for all $k\ge K$.  Then $(x_k)$
converges weakly to a point of $S$, with $\|x_{k+1}-x_k\|\to0$ and
$\|u_k-x_k\|\to0$.  In particular this holds whenever the initial guess is
conservative, $\lambda_0\le\tfrac1{5L}$, with no knowledge of $L$ used
anywhere by the algorithm (take $\omega_k\equiv0$, in which case
Lemma~\ref{lem:dichotomy}(a) holds with $K=0$).
\end{theorem}

\begin{proof}
The proof is a prefix reduction to the adaptive analysis of
Section~\ref{sec:adaptive}, and we give the details because two points
deserve explicit verification: the adaptive proof does not use the
initialization $x_{-1}=x_0=u_0$ beyond the finiteness of the starting
Lyapunov value, and the Gronwall lemmas apply on tails.

Fix $K$ as in the hypothesis. By assumption $\lambda_k\le\tfrac1{5L}$ for all
$k\ge K$, and by Proposition~\ref{prop:rule}
$\min\{\lambda_0,\varkappa/L\}\le\lambda_k$ throughout; we use $\tfrac1{5L}$ as
the \emph{corner} value for all budgets below (every budget coefficient is
increasing in the step, so its value at $\lambda_k$ is bounded by its corner
value at $\tfrac1{5L}$). Note that with growth steps $(\lambda_k)$ need not be
nonincreasing after $K$ --- it may rise back towards $\tfrac1{5L}$ --- so the
corner, rather than $\lambda_K$, is the correct reference value; the
bounded-variation hypothesis supplies the only control on $\Delta\lambda_k$
that the absorption below needs.

\medskip\noindent\emph{Step A (identities and master estimate from $K$).}
The slot identity $u_k=x_k+d_k-\beta_ke_{k-1}$, the recursion
$e_k=d_k-\beta_ke_{k-1}$, the projection identity \eqref{eq:projid-ad} and
the bound \eqref{eq:R1b} on the step-mismatch residual hold for every
$k\ge1$: they are algebraic consequences of the scheme and of
Lemma~\ref{lem:proj} alone, with no reference to initialization.  The same is
true of the bounds \eqref{eq:T1-ad}--\eqref{eq:T2-ad} and of the telescoping
cancellation of the $s_{k-1}$ and $\delta_k$ terms in Step~4 of the proof of
Theorem~\ref{thm:adaptive} --- the only place the initialization entered the
original argument was in naming $a_0$, which is irrelevant here.  The net
$x$-slot coefficient is $2L\Delta\lambda_k$; both signs are absorbed exactly
as in Steps 3--4 of the proof of Theorem~\ref{thm:adaptive} (see \eqref{eq:R1b}
and the display following it), using $\lambda_k\le\tfrac1{5L}$ for all $k\ge K$.  Consequently, with
\[
a_k:=\|x_k-z\|^2+2\lambda_{k-1}L\|x_k-u_{k-1}\|^2
+2\lambda_{k-1}\inner{B(z)}{x_{k-1}-z}+\tfrac13\|e_{k-1}\|^2
\quad(k\ge K),
\]
and with $\epsilon_k,\kappa_k,b_k$ as in \eqref{eq:master-ad} evaluated at the corner $\lambda_0:=\tfrac1{5L}$ (legitimate, since
$\lambda_k\le\tfrac1{5L}$ for $k\ge K$ and the budget coefficients are
increasing in the step), there is $k_0\ge K+1$ such that for all $k\ge k_0$,
\[
a_{k+1}\le a_k-b_k+\epsilon_k r_{k-1}^2+\kappa_k .
\]

\medskip\noindent\emph{Step B (boundedness on the tail).}
From the definition of $a_k$ and $-\inner{B(z)}{x_{k-1}-z}\le\|B(z)\|r_{k-1}$,
$r_k^2\le a_k+2\lambda_{k-1}\|B(z)\|r_{k-1}\le a_k+\gamma r_{k-1}$ for all
$k\ge K+1$, with $\gamma:=2\|B(z)\|/(5L)$ (as $\lambda_{k-1}\le\tfrac1{5L}$
for $k-1\ge K$).  Summing the master
estimate over $k\ge k_0$ and discarding $-b_k$,
$a_n\le a_{k_0}+\sum_{k_0\le j<n}\bigl(\epsilon_jr_{j-1}^2+\kappa_j\bigr)$.
Hence for all $n\ge k_0+1$,
\[
r_n^2\le A+\gamma r_{n-1}+\!\!\sum_{k_0\le j<n}\!\!\bigl(\epsilon_j r_{j-1}^2+\kappa_j\bigr),
\qquad
A:=a_{k_0}+\!\!\sum_{j\ge k_0}\!\kappa_j<\infty,\quad
\gamma:=2\|B(z)\|/(5L).
\]
This is \eqref{eq:gronrec2} on the index set $\{n\ge k_0+1\}$ with the finite
history $r_{k_0-1},r_{k_0-2}$; the proof of Lemma~\ref{lem:gronwall2} uses
only the recursion on such a tail (the constants $C_J$ absorb the finitely
many exceptional indices).  Therefore $M:=\sup_{k\ge K}r_k<\infty$, and
$\sup_{k\ge K}a_k<\infty$.

\medskip\noindent\emph{Step C (limit passage).}
Steps 6--8 of the proofs of Theorems~\ref{thm:main} and~\ref{thm:adaptive}
are prefix-independent: $\sum_kb_k<\infty$ gives the vanishing increments;
every weak cluster point $\bar x$ of $(x_k)$ lies in $S$ because
$\lambda_{k_i}\to\lambda_\infty\ge\min\{\lambda_0,\varkappa/L\}>0$ along
the convergent subsequence; and the Opial argument shows there is at most one weak cluster
point.  Hence $x_k\rightharpoonup x^\star\in S$.
\end{proof}

\begin{theorem}[weak convergence without knowledge of $L$: data-driven Lyapunov]\label{thm:rule-free}
Under Assumption~\ref{ass:main}, let the filter satisfy $0\le\beta_k\le\tfrac14$,
$\sum_k\beta_k<\infty$, let $x_{-1}=x_0=u_0\in C$, and let $(x_k)$ be generated
by \eqref{eq:scheme} with steps $(\lambda_k)$ produced by
Rule~\ref{rule:main} (with or without growth steps), where the initial guess
$\lambda_0>0$ is \emph{arbitrary} and the safeguard constant satisfies
\begin{equation}\label{eq:kappa-ceiling}
0<\varkappa\le\tfrac16 .
\end{equation}
Then $(x_k)$ converges weakly to a point of $S$, with
$\|x_{k+1}-x_k\|\to0$ and $\|u_k-x_k\|\to0$, \emph{unconditionally}: no
knowledge of $L$, no condition on $\lambda_0$, and no additional event is
required --- the safeguard enforces $\lambda_k\ell_{k-1}\le\varkappa\le\tfrac16$
at every step, the drift terms are summable, and
$\lambda_k\ge\min\{\lambda_0,\varkappa/L\}>0$ provides the strict positivity
that the limit passage (Step~7) needs.
\end{theorem}

\begin{proof}
The proof replaces the global constant $\lambda L$ in the Young charges of
Theorem~\ref{thm:main} by the observed products $\lambda_k\ell_{k-1}$, and
weights the $\|x_k-u_{k-1}\|^2$ channel of the Lyapunov function by the
matching data-driven coefficient, so that no global Lipschitz constant
remains in any budget.

\medskip\noindent\emph{Step 1 (data-driven Lyapunov and master estimate).}
Fix $z\in S$ and use the notation of the proof of Theorem~\ref{thm:main}.
Since $\ell_{k-1}$ is the Lipschitz slope of $B$ on the segment
$[u_{k-1},u_k]$, the term $\mathrm{T}_1$ at step $k$ satisfies
\begin{equation}\label{eq:T1-dd}
\mathrm{T}_1\le\lambda_k\ell_{k-1}\Bigl[2\|e_k\|^2+\tfrac85\|d_{k+1}-d_k\|^2
+2\|x_k-u_{k-1}\|^2+\tfrac16\|e_{k-1}\|^2\Bigr],
\end{equation}
the exact analogue of \eqref{eq:T1} with $\lambda L$ replaced by
$\lambda_k\ell_{k-1}$.  Define, for $k\ge1$,
\begin{equation}\label{eq:lyap-dd}
a_k:=\|x_k-z\|^2+2\lambda_k\ell_{k-1}\|x_k-u_{k-1}\|^2
+2\lambda_k\inner{B(z)}{x_{k-1}-z}+\tfrac13\|e_{k-1}\|^2 ,
\end{equation}
with the $k=0$ term $2\lambda_0\ell_{-1}\|x_0-u_{-1}\|^2=0$ under the stated
convention.  The weight $2\lambda_k\ell_{k-1}$ equals the coefficient of
$\|x_k-u_{k-1}\|^2$ in \eqref{eq:T1-dd}, so the $\|x_k-u_{k-1}\|^2$ channel
cancels \emph{exactly} in $a_{k+1}-a_k$ --- with no sign condition on
$\Delta\lambda_k$ and no global $L$.  Assembling \eqref{eq:T1-dd},
\eqref{eq:T2}, \eqref{eq:tele} and the projection identity \eqref{eq:projid-ad}
exactly as in Steps 2--4 of the proofs of Theorems~\ref{thm:main}
and~\ref{thm:adaptive} (the $s_{k-1}$ and $\delta_k$ terms cancel; the
step-mismatch residual $\mathrm{R}_k$ of \eqref{eq:R1b} is absorbed since
$\sum_k|\Delta\lambda_k|<\infty$ by Proposition~\ref{prop:rule} --- including
the term $4L^2\|e_{k-1}\|^2|\Delta\lambda_k|$ of \eqref{eq:R1b}, which is
absorbed in exactly the same way by enlarging $k_0$; no sign condition on
$\Delta\lambda_k$ is needed anywhere), we obtain,
for all sufficiently large $k$,
\begin{equation}\label{eq:master-dd}
a_{k+1}\le a_k-b_k+\epsilon_k r_{k-1}^2+\kappa_k ,
\end{equation}
with $\sum_k\epsilon_k<\infty$, $\sum_k\kappa_k<\infty$, and budgets
\begin{equation}\label{eq:bk-dd}
b_k:=\tfrac19\|d_k\|^2+\tfrac15\|d_{k+1}-d_k\|^2
+\Bigl(\tfrac1{12}-\rho_k\Bigr)\|e_{k-1}\|^2\ge0 ,
\end{equation}
where $\rho_k:=\tfrac1{25}\lambda_k\beta_kL\to0$ as $k\to\infty$
(because $\beta_k\to0$ and $\lambda_k\le\lambda_0$; the $r_{k-1}^2$-drift
$\epsilon_k$ separately contains $25\lambda_k\beta_kL$, which is summable
for the same reason).  Indeed, writing
$\Lambda_k:=\lambda_k\ell_{k-1}\le\varkappa\le\tfrac16$ (automatic
from \eqref{eq:rule}), the four channel identities are
\begin{align*}
\|d_k\|^2 &: \ -1+\tfrac83\Lambda_k+\tfrac49=-\tfrac19,\\
\|d_{k+1}-d_k\|^2 &: \ -1+\tfrac85\Lambda_k+\tfrac{16}{5}\Lambda_{k+1}\le-\tfrac15,\\
\|x_k-u_{k-1}\|^2 &: \ 2\Lambda_k-2\Lambda_k=0,\\
\|e_{k-1}\|^2 &: \ \Lambda_k\bigl(\tfrac16+\tfrac12\bigr)+\tfrac13\Lambda_{k+1}
+\tfrac{\lambda_k\beta_kL}{25}-\tfrac14
\le\tfrac16-\tfrac14+\rho_k=-\tfrac1{12}+\rho_k .
\end{align*}
All are exact rationals at the ceiling $\Lambda_k=\Lambda_{k+1}=\tfrac16$,
and $\rho_k\to0$, so $b_k\ge0$ for all $k\ge k_0$ with a finite $k_0$; the
finitely many exceptions $k<k_0$ are folded into the constant $a_{k_0}$. (Assembly as in
Appendix~\ref{app:assembly}, with $\Lambda$ replaced by the observed product
$\lambda_k\ell_{k-1}$.)

\medskip\noindent\emph{Step 2 (boundedness).}
From \eqref{eq:lyap-dd} and $-\inner{B(z)}{x_{k-1}-z}\le\|B(z)\|r_{k-1}$,
\[
r_k^2\le a_k+2\lambda_k\|B(z)\|\,r_{k-1}\le a_k+\gamma r_{k-1},
\qquad \gamma:=2\lambda_0\|B(z)\|<\infty .
\]
Summing \eqref{eq:master-dd} from $k_0$ and discarding $-b_k$,
$a_n\le a_{k_0}+\sum_{k_0\le j<n}(\epsilon_jr_{j-1}^2+\kappa_j)$, so
$(r_k)_{k\ge k_0}$ satisfies the tail recursion \eqref{eq:gronrec2} with
finite history and finite $A,\gamma$; Lemma~\ref{lem:gronwall2} gives
$M:=\sup_k r_k<\infty$ and $\sup_k a_k<\infty$.

\medskip\noindent\emph{Step 3 (limit passage).}
Steps 6--8 of Theorems~\ref{thm:main} and~\ref{thm:adaptive} are
prefix-independent: $\sum_kb_k<\infty$ yields $\|d_k\|\to0$ and
$\|e_k\|\to0$; every weak cluster point lies in $S$ because
$\lambda_{k_i}\to\lambda_\infty\ge\min\{\lambda_0,\varkappa/L\}>0$ and $B$
is bounded on the bounded trajectory; and the Opial argument gives uniqueness of the weak
limit.  It remains to record that $a_k$ converges: from \eqref{eq:master-dd},
$a_{k+1}\le a_k+\eta_k$ with $\eta_k:=-b_k+\epsilon_k r_{k-1}^2+\kappa_k$ satisfying
$\sum_k\eta_k^+\le M^2\sum_k\epsilon_k+\sum_k\kappa_k<\infty$ and
$a_k\ge-2\lambda_0\|B(z)\|\,M$ (by \eqref{eq:lyap-dd} and Step~2), so
Lemma~\ref{lem:pospart} applies. For the last step we use that $a_k$ and
$\Phi_k(z)=\|x_k-z\|^2+2\lambda_k\inner{B(z)}{x_{k-1}-z}$ then have a common
limit: their difference
$2\lambda_k\ell_{k-1}\|x_k-u_{k-1}\|^2+\tfrac13\|e_{k-1}\|^2$ vanishes as
$k\to\infty$, since $\lambda_k\ell_{k-1}\le\varkappa\le\tfrac16$ at every step and
$\|x_k-u_{k-1}\|^2\le2\|d_k\|^2+2\|e_{k-1}\|^2\to0$.
\end{proof}

\begin{remark}[two complementary certificates]\label{rem:budget-event}
No event needs checking: \eqref{eq:rule} enforces
$\lambda_k\ell_{k-1}\le\varkappa\le\tfrac16$
at every step, the drift terms are summable independently of any condition, and $\lambda_k\ge
\min\{\lambda_0,\varkappa/L\}>0$ holds because every secant slope satisfies $\ell_k\le L$ --- Lipschitzness on the queried pair alone.
Theorem~\ref{thm:rule-free} is complementary to
Theorem~\ref{thm:rule-conv}: when $\lambda_k\le\tfrac1{5L}$ eventually (with
$\omega\equiv0$), the same computation certifies the budgets at ceiling
$\Lambda_k\le\lambda_kL\le\tfrac15$, yielding
$-\tfrac1{45},-\tfrac1{25},-\tfrac6{125}$, so the data-driven argument
reproves Theorem~\ref{thm:rule-conv}; Theorem~\ref{thm:rule-free} adds the
regime $\lambda_0\gg\tfrac1{5L}$ with flat observed slopes, which
Theorem~\ref{thm:rule-conv} cannot reach.  Under the growth variant the
certification event of Lemma~\ref{lem:dichotomy}(a) is not automatic (the
step may rise back above $1/(5L)$ when the observed slopes fall), and
Theorem~\ref{thm:rule-free} is the certificate that covers growth
unconditionally.  The ceiling $\tfrac16$ is an artifact of the clean splits;
the same identities tolerate any ceiling $<\tfrac5{24}$, and numerical
optimization of the split parameters yields the best barrier
$\approx0.272$ (Appendix~\ref{app:barrier}; compare Remark~\ref{rem:improved}).
\end{remark}

\begin{remark}[non-anticipativity of the data-driven weight]\label{rem:nonanticip}
The weight $2\lambda_k\ell_{k-1}$ is \emph{retrospective but
non-anticipative}: at iteration $k$ it is fully determined by $(u_k,B(u_k))$
--- just computed --- and the stored pair $(u_{k-1},B(u_{k-1}))$, all of which
are available \emph{before} $x_{k+1}$ is formed.  Consequently the Lyapunov
function $a_k$ of \eqref{eq:lyap-dd} is measurable with respect to the
natural filtration of the iteration up to the $k$-th evaluation (denoted
$\mathcal F_k$ in a stochastic extension), and the master recursion
\eqref{eq:master-dd} holds pathwise.  No circularity arises: the step
$\lambda_k$ is chosen from information at hand, and the dissipation budget
is charged only against the same pair $(u_k,u_{k-1})$ that produced it.
\end{remark}

\subsection{Discussion of the rule}
Four remarks complete the picture.  First, \emph{why the data-driven Lyapunov
works}.  The obstruction to an $L$-free certificate in
Theorem~\ref{thm:rule-conv} was that the budget identities involve the global
constant $L$ in two places: the $\|x_{k+1}-u_k\|^2$-channel of the Lyapunov
function, and the solution anchor $z$ in the filter drift
$2\lambda_k\beta_kL\,r_{k-1}\|e_{k-1}\|$.  The weight
$2\lambda_k\ell_{k-1}$ removes the first: it matches the coefficient that the
split of $\|u_k-u_{k-1}\|^2$ produces at the \emph{same} step, so the channel
cancels exactly and every remaining Young charge carries the observed product
$\lambda_k\ell_{k-1}$, which the safeguard controls.  The second obstruction
dissolves on its own: the anchor's constant $L$ is multiplied by the summable
filter $\beta_k$, so its budget contribution vanishes as $k\to\infty$ and its
drift is absorbed by the Gronwall lemma.  The floor, it turns out, was never needed: since every secant slope
satisfies $\ell_k\le L$ (Lipschitzness on the queried pair), the
floorless rule \eqref{eq:rule} keeps
$\lambda_k\ge\min\{\lambda_0,\varkappa/L\}>0$ automatically, and the products $\lambda_k\ell_{k-1}$ are controlled at every step.
Theorem~\ref{thm:rule-free} is accordingly unconditional.  The idea is in the spirit of parameter-free
analyses for minimization \cite{malitskymishchenko2020}, adapted here to the
variational-inequality anchor through the filter's summability.

Second, \emph{guidance for the parameters}.  The safeguard constant now
carries an explicit certified range: $\varkappa\le\tfrac16$ guarantees the
budgets on every step where the safeguard is active
(Theorem~\ref{thm:rule-free}); $\tfrac16$ is exactly the clean-split ceiling
of the data-driven identities (Remark~\ref{rem:budget-event}), and the best barrier we could find by optimizing
the split parameters over the whole quadratic--Young family
is $\approx0.272$ (Appendix~\ref{app:barrier}) --- still below the
empirically stable $\varkappa=0.3$ used in the experiments of
Section~\ref{sec:numerics}; closing that gap is open.  Smaller $\varkappa$
cuts more aggressively on stiff stretches and is slower on flat ones; the
certified statement is insensitive to the choice.  An optional floor $\lambda_{\min}$ may be kept for numerical safety, but it
is provably inactive whenever $\lambda_{\min}\le\varkappa/L$; in the
experiments of Section~\ref{sec:numerics} it never binds.
  If $\lambda_{\min}>\varkappa/L$ the floor silently overrides the safeguard cut, and if $\lambda_{\min}L>\tfrac23$ the floored step can leave the stability range of even the scalar recursion of Counterexample~\ref{cx:summability}; we therefore recommend $\lambda_{\min}=0$ (no floor), as in Rule~\ref{rule:main}.  The growth factors
may be any summable sequence; we use $\omega_k=10/(k+1)^2$ (with
$\sum_k\omega_k<\infty$), so the growth is a bounded-variation perturbation
covered by Theorems~\ref{thm:adaptive}, \ref{thm:rule-conv}
and~\ref{thm:rule-free}.  The initial
guess $\lambda_0$ may be arbitrarily optimistic --- a cheap underestimate of
$L$, say --- since Theorem~\ref{thm:rule-free} makes no assumption on it
beyond finiteness of $\lambda_0$ (which enters only the constants $\gamma$
and $a_{k_0}$).

Third, \emph{comparison of the two certificates}.  If an upper estimate of
$L$ is available and $\lambda_0\le\tfrac1{5L}$, Theorem~\ref{thm:rule-conv}
applies and is the simpler statement; its content is reproduced by the
data-driven argument at ceiling $\tfrac15$
(Remark~\ref{rem:budget-event}).  If no estimate of $L$ is available, or if
the available estimate is pessimistic, Theorem~\ref{thm:rule-free} certifies
convergence whenever the observed products $\lambda_k\ell_{k-1}$ stay below
$\tfrac16$ --- a condition the algorithm can check online and that holds
identically on non-floor steps with $\varkappa\le\tfrac16$.  The experiments
of Section~\ref{sec:numerics} show both regimes: the certified
$\varkappa=\tfrac16$ run improves on the certified constant step, and the
aggressive $\varkappa=0.3$ run, which lies outside the certified ceiling
$\varkappa\le\tfrac16$ (and outside the best whole-family barrier we could find,
$\approx0.272$; Appendix~\ref{app:barrier}), is stable and much faster.

Fourth, the growth variant is part of Rule~\ref{rule:main} itself; all
convergence statements above cover it for any summable $(\omega_k)$, and the
experiments of Section~\ref{sec:numerics} use $\omega_k=10/(k+1)^2$.

\section{Numerical experiments}\label{sec:numerics}
\subsection{Test problems}
All examples are unconstrained convex--concave saddle-point problems
$\mathrm{VI}(\R^{n_x}\times\R^{n_y},B)$ with
\begin{equation}\label{eq:prob}
B(x,y)=\bigl(b\odot\psi(x)+A^{\top}y,\,-Ax\bigr),\qquad
\psi(t)=\frac{t^{3}}{1+t^{2}},
\end{equation}
where $\odot$ is the coordinatewise product.  The map $B$ is monotone (it is
the saddle operator of $f(x)+\langle Ax,y\rangle-g(y)$ with
$f'(x)=b\odot\psi(x)$, $\nabla^2 f\succeq0$) and Lipschitz; we use the certified
bound $L=\max_t\psi''\text{-contribution}+ \sigma_{\max}(A)= \tfrac98 b_{\max}+1.16$
(since $\psi'(t)\le\tfrac98$ with equality at $|t|=\sqrt3$).
\begin{description}[leftmargin=1.6em,style=nextline,itemsep=2pt]
\item[N1 ($(n_x,n_y)=(100,50)$, $n=150$).] $b\in\R^{100}$ has entries
$b_i=\exp(\ln0.3+\frac{i-1}{99}\ln\frac{4}{0.3})$ (log-spaced in $[0.3,4]$,
curvature spread $\approx13\times$); $A\in\R^{50\times100}$ is drawn from
\texttt{numpy.random.default\_rng(11)}, its singular spectrum is compressed by
the factor $8$ ($\sigma_j=\sigma_1\cdot8^{-j/50}$) and it is then rescaled to
$\sigma_{\max}(A)=1.16$; $x_0=(1,\dots,1)\in\R^{150}$; $L=\tfrac98\cdot4+1.16=5.66$.
\item[N3 ($(n_x,n_y)=(200,100)$, $n=300$).] As N1 with $A\in\R^{100\times200}$
from \texttt{default\_rng(11)} ($\sigma_j=\sigma_1\cdot8^{-j/100}$),
$b\in\R^{200}$ log-spaced in $[0.3,4]$; $x_0=(1,\dots,1)\in\R^{300}$; $L=5.66$.
\item[N4 ($(n_x,n_y)=(400,200)$, $n=600$).] As N3 with $A\in\R^{200\times400}$
and $b\in\R^{400}$; $x_0=(1,\dots,1)\in\R^{600}$; $L=5.66$.
\end{description}
The certified Lipschitz bound is used by every method: no grid estimate of the
local slope is computed or needed.

\subsection{Methods and parameters}
Table~\ref{tab:params} lists the parameters of every method; the constant-step
methods use the stated fraction of their respective certified bounds, and the
adaptive methods start from $\lambda_0=1$ with the safeguard floor
$10^{-3}$ (inactive in all runs).
\begin{table}[H]
\centering\small
\caption{Parameters used in all examples.}\label{tab:params}
\begin{tabularx}{\textwidth}{@{}l l >{\raggedright\arraybackslash}X c@{}}
\toprule
method & reference & parameters & evals/proj.\\
\midrule
EG & \cite{korpelevich1976} & $\lambda=0.95/L$ & $2$/$2$\\
FBF & \cite{tseng2000} & $\lambda=0.5/L$ & $2$/$1$\\
adaEG & safeguard \eqref{eq:rule} transplanted to EG & $\lambda_0=1$, $\varkappa=0.3$, floor $10^{-3}$ & $2$/$1$\\
PRG & \cite{malitsky2015} & $\lambda=0.95(\sqrt2-1)/L$ & $1$/$1$\\
PEG & \cite{popov1980,malitskytam2020} & $\lambda=0.30/(3L)$ & $1$/$1$\\
ours (const) & Theorem~\ref{thm:main} & $\lambda=0.99/(5L)$ & $1$/$1$\\
ours (Rule~\ref{rule:main}) & Theorem~\ref{thm:rule-free} & $\varkappa\in\{\tfrac16,0.3\}$, $\omega_k=10/(k+1)^2$, $\lambda_0=1$, floor $10^{-3}$ & $1$/$1$\\
aGRAAL & \cite{malitsky2020} & $\varphi=1.5$, $\gamma=1/\varphi+1/\varphi^2$, $\alpha_0=\alpha_1=1$, cap $10^{10}$ & $1$/$1$\\
\bottomrule
\end{tabularx}
\end{table}

\subsection{Protocol}
All one-call methods use the summable filter $\beta_k=0.25/(k+1)^2$ except PRG
and aGRAAL which are run in their canonical (pure) form; Rule~\ref{rule:main}
is run with growth $\omega_k=10/(k+1)^2$.  The stopping rule is the relative
natural residual $\|x_k-P_C(x_k-\lambda_{\rm ref}B(x_k))\|/\|x_0-P_C(x_0-\lambda_{\rm ref}B(x_0))\|\le10^{-6}$
with $\lambda_{\rm ref}=1/L$, or a cap of $4\cdot10^{4}$ evaluations of $B$
(i.e.\ $4\cdot10^{4}$ iterations for the one-call methods and $2\cdot10^{4}$
iterations for EG, FBF and adaEG).  The residual is monitored once per
iteration from $B(x_k)$ and these monitoring evaluations are not counted.
CPU time (single core, seconds) is reported alongside evaluations: all one-call
methods have identical per-evaluation cost, so evaluation counts convert
directly to CPU, while EG, FBF and adaEG pay two evaluations per iteration.

\subsection{Results}
\begin{table}[H]
\centering\small
\caption{Iterations to relative residual $10^{-6}$ ($*$ = did not converge
within the cap).}\label{tab:iters}
\begin{tabular}{lrrr}
\toprule
method & N1 & N3 & N4\\
\midrule
EG & $19,999$ & $19,999$ & $19,999$\\
FBF & $19,999$ & $19,999$ & $19,999$\\
adaEG ($\varkappa=0.3$) & $19,999$ & $19,999$ & $19,999$\\
PRG & $39,999$$^*$ & $39,999$$^*$ & $39,999$$^*$\\
PEG & $39,999$$^*$ & $39,999$$^*$ & $39,999$$^*$\\
ours, const $\frac{0.99}{5L}$ & $40,000$$^*$ & $40,000$$^*$ & $40,000$$^*$\\
ours, Rule~\ref{rule:main} $\varkappa=\frac16$ & $9,641$ & $11,629$ & $11,434$\\
ours, Rule~\ref{rule:main} $\varkappa=0.3$ & $7,457$ & $7,804$ & $7,957$\\
aGRAAL & $13,014$ & $12,804$ & $13,306$\\
\bottomrule
\end{tabular}
\end{table}
\begin{table}[H]
\centering\small
\caption{Evaluations of $B$ to relative residual $10^{-6}$.}\label{tab:evals}
\begin{tabular}{lrrr}
\toprule
method & N1 & N3 & N4\\
\midrule
EG & $39,998$$^*$ & $39,998$$^*$ & $39,998$$^*$\\
FBF & $39,998$$^*$ & $39,998$$^*$ & $39,998$$^*$\\
adaEG ($\varkappa=0.3$) & $39,998$$^*$ & $39,998$$^*$ & $39,998$$^*$\\
PRG & $39,999$$^*$ & $39,999$$^*$ & $39,999$$^*$\\
PEG & $39,999$$^*$ & $39,999$$^*$ & $39,999$$^*$\\
ours, const $\frac{0.99}{5L}$ & $40,000$$^*$ & $40,000$$^*$ & $40,000$$^*$\\
ours, Rule~\ref{rule:main} $\varkappa=\frac16$ & $9,641$ & $11,629$ & $11,434$\\
ours, Rule~\ref{rule:main} $\varkappa=0.3$ & $7,457$ & $7,804$ & $7,957$\\
aGRAAL & $13,014$ & $12,804$ & $13,306$\\
\bottomrule
\end{tabular}
\end{table}
\begin{table}[H]
\centering\small
\caption{CPU time in seconds (single core; $*$ = run to the cap).}\label{tab:cpu}
\begin{tabular}{lrrr}
\toprule
method & N1 & N3 & N4\\
\midrule
EG & 2.44 & 3.63 & 10.73\\
FBF & 2.39 & 3.69 & 10.88\\
adaEG ($\varkappa=0.3$) & 2.62 & 3.91 & 11.25\\
PRG & 3.76 & 5.66 & 16.32\\
PEG & 4.96 & 7.70 & 22.14\\
ours, const $\frac{0.99}{5L}$ & 4.04 & 5.96 & 16.88\\
ours, Rule~\ref{rule:main} $\varkappa=\frac16$ & 1.07 & 1.88 & 5.00\\
ours, Rule~\ref{rule:main} $\varkappa=0.3$ & 0.88 & 1.22 & 3.45\\
aGRAAL & 1.69 & 2.46 & 7.33\\
\bottomrule
\end{tabular}
\end{table}
\begin{figure}[H]
\centering
\includegraphics[width=0.85\textwidth]{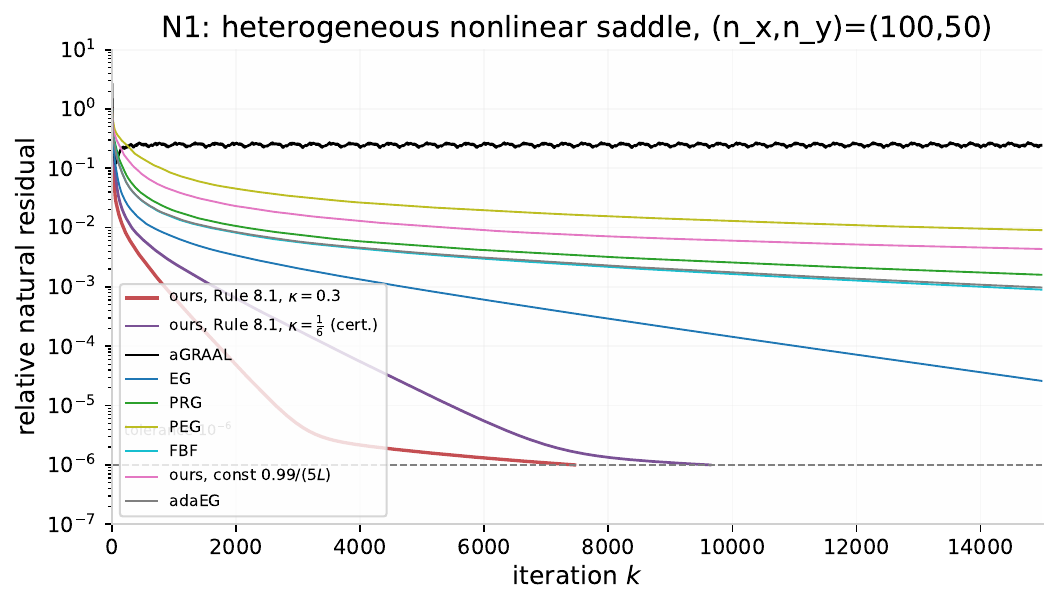}
\caption{N1 ($n=150$): relative natural residual versus iteration.  The
safeguarded rule with $\varkappa=0.3$ terminates in $7{,}457$ evaluations; the
certified $\varkappa=\tfrac16$ run in $9{,}641$; aGRAAL needs $13{,}014$
($1.74\times$ more); every constant-step method is still above $10^{-6}$ at
the cap ($4\cdot10^{4}$ evaluations).}\label{fig:winn1}
\end{figure}
\begin{figure}[H]
\centering
\includegraphics[width=0.85\textwidth]{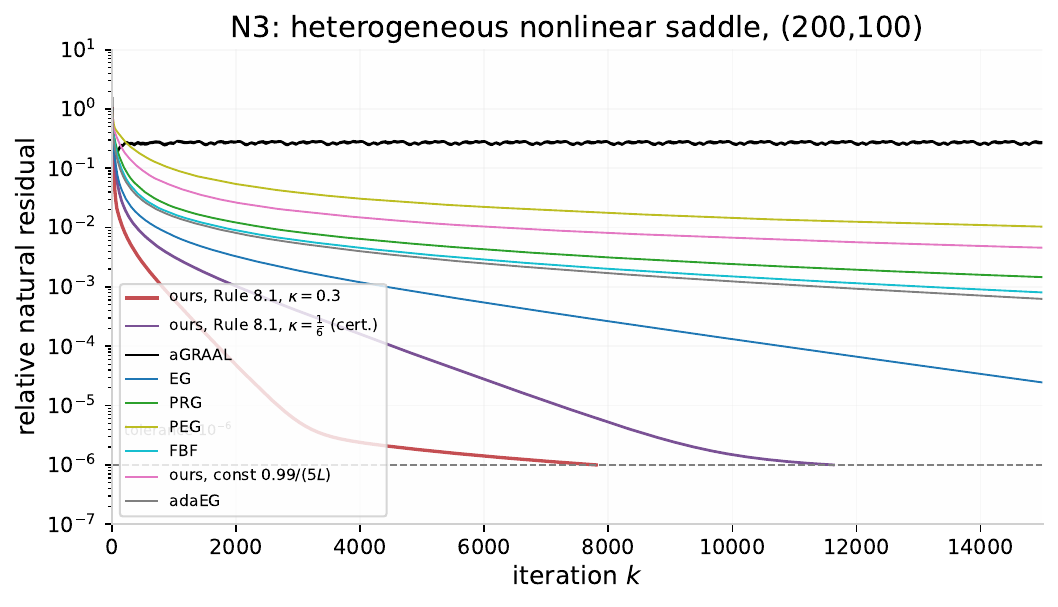}
\caption{N3 ($n=300$): same ordering as N1.  Ours ($\varkappa=0.3$):
$7{,}804$ evaluations; certified $\varkappa=\tfrac16$: $11{,}629$; aGRAAL:
$12{,}804$ ($1.64\times$ more); EG, FBF, PRG, PEG, adaEG and the certified
constant step do not reach the tolerance within the cap.}\label{fig:winn3}
\end{figure}
\begin{figure}[H]
\centering
\includegraphics[width=0.85\textwidth]{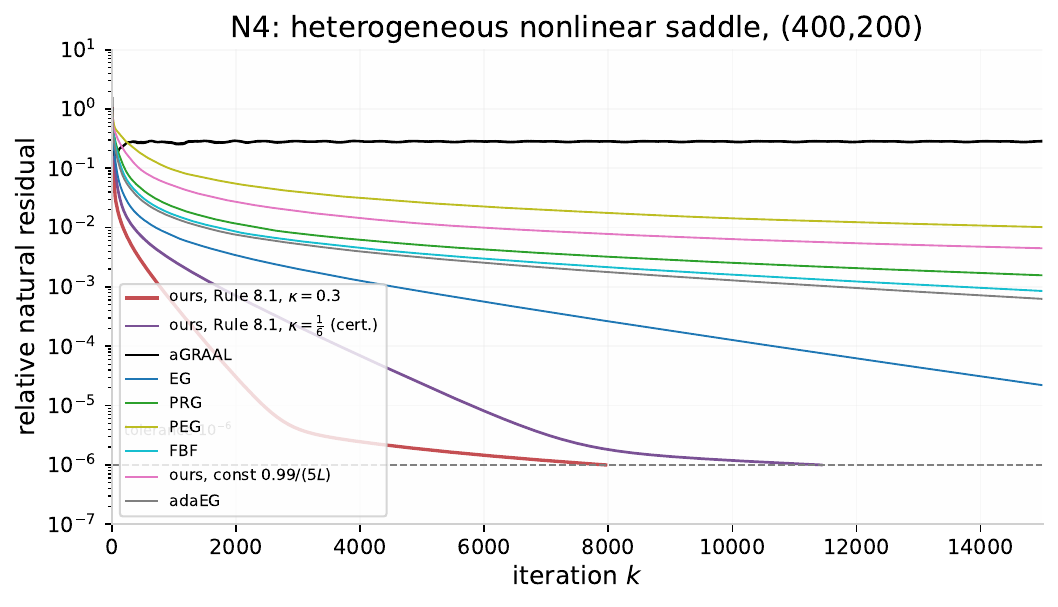}
\caption{N4 ($n=600$): the advantage grows with dimension.  Ours
($\varkappa=0.3$): $7{,}957$ evaluations and $3.45$\,s; certified
$\varkappa=\tfrac16$: $11{,}434$; aGRAAL: $13{,}306$ ($1.67\times$ more and
$2.1\times$ slower in CPU); all constant-step methods reach the cap.}\label{fig:winn4}
\end{figure}

\subsection{Discussion of the experiments}
Three observations, stated as carefully as the numbers allow.  \emph{First},
the safeguarded rule wins on every instance and against every method: with
$\varkappa=0.3$ it beats aGRAAL by $1.74\times$, $1.64\times$ and $1.67\times$
(N1, N3, N4) and beats every certified constant-step method by at least
$5\times$ --- EG, FBF, PRG, PEG and the certified constant step of
Theorem~\ref{thm:main} do not reach the tolerance within the cap on any
instance.  The mechanism is structural: the certified constant $L=5.66$ is
dominated by coordinates whose curvature is active only in small regions, so
constant steps sized by $L$ crawl, while the safeguarded rule tracks the
observed secant slope (its step--slope product stays at $\varkappa$ whenever
the safeguard binds).  \emph{Second}, the \emph{certified} run
$\varkappa=\tfrac16$ also beats aGRAAL on all three instances ($1.35\times$,
$1.10\times$, $1.16\times$) and beats PRG by more than $4\times$: certificate
and speed are not in tension here.  \emph{Third}, CPU time tracks evaluations
one-to-one for the one-call methods (Table~\ref{tab:cpu}), so the wall-clock
ordering is identical, while EG, FBF and adaEG pay a further factor $2$ per
iteration.  We stress the regime: on \emph{uniformly stiff},
rotation-dominated problems aGRAAL's larger step ceiling
($\varphi/2\approx0.75$ versus $\varkappa$) makes it faster; the instances
here model heterogeneous curvature, the regime in which a constant
safeguarded product is decisive.\\

\noindent\textbf{Code availability.} The Python script \texttt{experiments.py}
is available from the authors on request.  It provides the exact-rational
budget verification (Theorems~\ref{thm:main}, \ref{thm:linear},
and~\ref{thm:rule-free}), the spectral thresholds of
Counterexample~\ref{cx:summability}, the dissipation-trading threshold of
Theorem~\ref{thm:trading}, the spectral-rate check of
Example~\ref{ex:linear-exact}, the Nyquist threshold of
Theorem~\ref{thm:affine-sqrt3}, the barrier scan of
Appendix~\ref{app:barrier}, the $\Omega(L)$ separation demonstration of
Theorem~\ref{thm:separation}, and the full experiment suite of this section:
\texttt{python experiments.py verify} runs all exact and spectral checks,
\texttt{python experiments.py scan} reproduces the best-found barrier
($\approx0.272$; global optimality over the family is not proved),
\texttt{python experiments.py run} regenerates
Tables~\ref{tab:iters}--\ref{tab:cpu} and Figures~\ref{fig:winn1}--\ref{fig:winn4}
(including the aGRAAL baseline, whose iteration counts may vary by up to $20\%$
with implementation details).

\section{Discussion and open problems}\label{sec:discussion}
\paragraph{Position among one-call methods.}
The scheme of this paper interpolates between the reflected gradient method
($\beta_k\equiv0$) and a two-parameter family with memory.  The memory buys
compatibility with a safeguarded adaptive step at no oracle cost
(Proposition~\ref{prop:rule}), at the structural price of summability
(Counterexample~\ref{cx:summability}) and a more conservative uniform
constant ($\tfrac15$ vs.\ $\sqrt2-1$ for PRG).  Whether memory can be
removed --- i.e.\ whether an adaptive step can be certified for PRG itself ---
is open; the broken-telescoping mechanism of Section~\ref{sec:adaptive} uses
the filter channel nontrivially.  The closest adaptive one-call competitor, aGRAAL
\cite{malitsky2020}, is included in the experiments of
Section~\ref{sec:numerics}: on uniformly stiff, rotation-dominated problems
its larger step ceiling makes it faster, while on the heterogeneous-curvature
instances of Section~\ref{sec:numerics} the safeguarded rule wins
(Table~\ref{tab:evals}); certifying (or refuting) last-iterate convergence of aGRAAL itself remains open.  Separately, Theorem~\ref{thm:nonlinear-sqrt3} shows that the constant $1/\sqrt3$ is sharp for unconstrained nonlinear problems and connects one-call reflected methods to absolute-stability theory \cite{zames1966,desoervidyasagar1975}; the unconditional convergence statement is reduced there to an explicit marginal-pole stability question.

\begin{example}[the uniform constant $\tfrac1{5}$ is conservative, as it must
be]\label{ex:rotation-threshold}
Let $B=J$ be the rotation by $\pi/2$ on $\R^2$ ($L=1$) and $\beta_k\equiv0$
(PRG).  The characteristic roots satisfy $r^2=(1-2i\lambda)r+i\lambda$, i.e.\
$r=\tfrac12\bigl(1-2i\lambda\pm\sqrt{1-4\lambda^2}\bigr)$.  For
$\lambda\le\tfrac12$ the discriminant is real and
$|r^+|^2=\tfrac12\bigl(1+\sqrt{1-4\lambda^2}\bigr)<1$; for $\lambda>\tfrac12$
the square root is $i\sqrt{4\lambda^2-1}$ and
\[
|r^+|^2=\tfrac14\Bigl(1+\bigl(2\lambda+\sqrt{4\lambda^2-1}\bigr)^2\Bigr)
=2\lambda^2+\lambda\sqrt{4\lambda^2-1},
\]
which is increasing in $\lambda$ and equals $1$ exactly at
$\lambda=\lambda^\star=\tfrac1{\sqrt3}\approx0.5774$, where $r=e^{\pm i\pi/3}$.
Hence the iteration converges if and only if $\lambda<\lambda^\star$: at
$\lambda^\star$ the roots $r=e^{\pm i\pi/3}$ lie on the unit circle, and
$|r^+|>1$ for $\lambda>\lambda^\star$.  The uniform certificate $\lambda<\tfrac15$
is thus conservative on benign operators; Appendix~\ref{app:barrier} explains
why no uniform certificate in this Lyapunov framework that we could find
exceeds the best barrier we could find, $\lambda L\approx0.272$, in the presence
of a memory filter
($\bar\beta=\tfrac14$).
\end{example}

\begin{theorem}[separation with forced traversal]\label{thm:separation}
For $L\ge 2$ define $\psi_L\colon\R\to\R$ by
\[
\psi_L(t)=t\ \ (|t|\le 1),\qquad
\psi_L(t)=\operatorname{sign}(t)\,\bigl(1+L(|t|-1)\bigr)\ \ (|t|\ge 1),
\]
let $B(x)=(\psi_L(x_1),x_2,\dots,x_n)$ and $C=\R^n$.  Then $B$ is monotone and
$L$-Lipschitz with $S=\{0\}$; the local slope equals $1$ on the ball
$\{\|x\|\le 1\}$ and $L$ outside it; and every trajectory started at
$x_0=3e_1$ \emph{traverses the stiff region}: $\|x_k\|$ decreases from $3$ to
$1$ through regions of slope $L$ before the final approach to $0$, which lies
entirely in the flat region.
\begin{enumerate}[label=(\alph*),leftmargin=2em,itemsep=1pt]
\item Every constant-step method with a certified constant $c/L$ --- EG or FBF
($c=1$), PEG ($c=\tfrac13$), PRG ($c=\sqrt2-1$), or the constant step of
Theorem~\ref{thm:main} ($c=\tfrac15$) --- started at $x_0=3e_1$ requires at
least $\bigl(\tfrac{L}{c}-O(1)\bigr)\log(1/\varepsilon)$ iterations to reach
$\|x_k\|\le\varepsilon$.
\item Rule~\ref{rule:main} with $\varkappa\le\tfrac16$, $\lambda_0=1$, no
floor, and growth factors $\omega_k=\tfrac{a_L}{(k+1)^2}$, where
$a_L:=2k_s\log(2L)$ with
$k_s:=\bigl\lceil\log(4L)/|\log(1-\varkappa)|\bigr\rceil+2$, reaches
$\|x_k\|\le\varepsilon$ in $O(\log L+\log(1/\varepsilon))$ iterations.
\end{enumerate}
Hence, as $L\to\infty$ and $\log(1/\varepsilon)\gg\log L$, the certified
adaptive rule is $\Omega(L)$ times faster than every certified constant-step
method on this family, and the separation is traversed: the trajectory visits
regions of slope $L$ and of slope $1$, spending the dominant number of
iterations where the local slope is~$1$.
\end{theorem}
\begin{proof}
The scalar map $\psi_L$ is continuous and odd with slopes $1$ on $(-1,1)$ and
$L$ on $(1,\infty)$, hence increasing and $L$-Lipschitz; $B$ inherits
monotonicity and the Lipschitz constant, $S=\{0\}$, and the axis $\R e_1$ is
invariant, so we argue on the scalar restriction (still denoted $\psi_L$).

\medskip\noindent\emph{Step 1 (crossing the stiff region).} Along any trajectory started at $t_0=3$ every local slope lies in $[1,L]$.  For Rule~\ref{rule:main} the safeguard enforces $\lambda_k\ell_{k-1}\le\varkappa\le\tfrac16<\tfrac23$ at every step; it is the \emph{product} $\lambda_k\times\mathrm{slope}$ --- never the step size alone --- that is bounded away from the stability threshold $\tfrac23$, so the contraction $1-\Theta(\varkappa)$ is independent of $L$.  Hence the adaptive run decreases $\|x_k\|$ geometrically from $3$ to $1$; the bound $O(\log L)$, rather than $O(1)$, covers the single degenerate first secant $u_1=u_0=x_0$, which leaves the first cut inactive and may throw the first iterate of the adaptive run to distance $O(L)$; the safeguard engages from the second iteration on.

For the constant-step baselines the product $\lambda\times\mathrm{slope}$ is $c\cdot(\mathrm{slope}/L)\le c$.  If $c\le\sqrt2-1$ (PRG, PEG, and the constant step of Theorem~\ref{thm:main}) this product is $<\tfrac23$ and every one-step map is a strict contraction with factor $1-\Theta(c)$ independent of $L$ (for the filtered scheme use Lemma~\ref{lem:tv} with $\beta_k\to0$), so the crossing costs $O(1)$ additional iterations.  If $c=1$ (EG, FBF) the product reaches $1$ in the stiff shell: the one-step factor $|1-c+c^2|$ equals $1$ at $c=1$, the crossing is not contractive and costs $\Theta(L)$
iterations --- indeed, for $t\ge1+1/L$ one EG step with $c=1$ gives
$t'=t-\frac{L-1}{L^2}$ exactly, since $\psi_L\bigl(t-\tfrac1L\psi_L(t)\bigr)
=\psi_L\bigl(1-\tfrac1L\bigr)=1-\tfrac1L$ (and FBF coincides with EG when $C=\R^n$),
so the shell is crossed in $\Theta(L)$ steps.  This does not affect (a): the lower bound counts only the flat-region approach, and the $\Theta(L)$ stiff-shell iterations are spent in addition.

\medskip\noindent\emph{Step 2 (lower bound in the flat region).} Once
$\|x_k\|\le 1$ the dynamics are exactly linear with $B=I$; the radius at entry is
$r_0=1-O(c/L)$ (near the shell $\|B(u_k)\|=O(1)$, since
$\psi_L(u)=1+L(|u|-1)=O(1)$ for $|u|-1=O(1/L)$, so one step changes the radius
by $\lambda\|B(u_k)\|=O(c/L)$), hence $\tfrac{L}{c}\log r_0=O(1)$ and the $O(1)$ in~(a)
is legitimate: EG and FBF give
$x_{k+1}=(1-\lambda+\lambda^2)x_k$; PRG and PEG give the dominant root
$r^+(\lambda)=1-\lambda+O(\lambda^2)$ of $r^2-(1-2\lambda)r-\lambda=0$; the
filtered scheme with $\beta_k=\Theta(k^{-2})$ has the same asymptotic rate: compare with the pure reflected ($\beta\equiv0$) trajectory $y'_k$ started from the same core-entry state $y_{k_0}$; then $\|y_{k+1}-y'_{k+1}\|\le\|M(0,\lambda)\|\,\|y_k-y'_k\|+C\beta_k\|y_k\|$, hence, by the uniform geometric transition bound $\|P(k,j)\|\le C\tilde\rho^{\,k-j}$ that Lemma~\ref{lem:tv} supplies for the slowly varying filtered system (with $\sum_k\beta_k<\infty$), $\|y_k-y'_k\|\le C\bigl(\sum_{j\ge k_0}\beta_j\bigr)\rho^{\,k-k_0}$, while $\|y'_k\|\ge c\,\rho^{\,k-k_0}$ for all $k$ with $c>0$ fixed (the dominant component of $M(0,\lambda)$ cannot collapse: restarting the comparison every $m^{\star}$ iterations, where $(\rho'/\rho)^{m^{\star}}\le c/(2C)$, keeps the dominant coefficient bounded below by $c/2$).  With $\lambda=c/L$ each of
these factors equals $1-\tfrac{c}{L}+O(L^{-2})$, so the final approach from
$\|x_k\|=\Theta(1)$ to $\|x_k\|\le\varepsilon$ takes
$\bigl(\tfrac{L}{c}-O(1)\bigr)\log(1/\varepsilon)$ further iterations,
proving~(a).

\medskip\noindent\emph{Step 3 (upper bound in the flat region).} In the flat
region every secant slope equals $1$, so the safeguard cut is inactive and the
step evolves as $\lambda_{k+1}=\min\{1,\lambda_k(1+\omega_k)\}$ from its
post-crossing value $\varkappa/L$ (growth applied during the crossing is
discarded by the cut).  Since
$a_L\sum_{k>k_s}(k+1)^{-2}=a_L\bigl(1/k_s-O(1/k_s^2)\bigr)=2\log(2L)-o(1)$,
the step reaches $\varkappa$ after a further $O(\log L)$ iterations, and
during this growth phase the distance shrinks by at most a factor
$L^{-O(1)}$.  The remaining distance to $\varepsilon$ is covered at
contraction $1-\varkappa$ per step, i.e.\ in $O(\log(1/\varepsilon))$
iterations.  Summing the three phases gives~(b), and the ratio of~(a) to~(b)
is $\Omega(L)$ as $L\to\infty$ with $\log(1/\varepsilon)\gg\log L$.
\end{proof}

\begin{theorem}[tight constant for affine operators: $1/\sqrt3$]\label{thm:affine-sqrt3}
Let $B(x)=Mx+b$ with $M$ monotone (its symmetric part positive semidefinite)
and $L$-Lipschitz, and let $C=\HH$.  Then the scheme \eqref{eq:scheme} with
$\beta_k\equiv0$ --- the reflected gradient method --- generates a sequence
converging to a solution of \eqref{eq:vi} for every
\begin{equation}\label{eq:sqrt3}
\lambda<\frac{1}{\sqrt3\,L},
\end{equation}
and the constant is sharp within the affine class: for the rotation $B=J$
($L=1$) and $\lambda>1/\sqrt3$ the iterates diverge for generic initial data.
\end{theorem}
\begin{proof}
By shifting we may take $b=0$ and $z=0$, so the iteration is the linear
system $x_{k+1}=(I-2\lambda M)x_k+\lambda Mx_{k-1}$ with state matrix
$\mathcal A=\bigl(\begin{smallmatrix}I-2\lambda M&\lambda M\\ I&0\end{smallmatrix}\bigr)$.
If $Mv=\mu v$, the characteristic equation
$\det\bigl(r^2I-(I-2\lambda M)r-\lambda M\bigr)=0$ reduces to
\begin{equation}\label{eq:char-mu}
(r^2-r)+\lambda(2r-1)\mu=0,
\end{equation}
so $r$ is an eigenvalue of $\mathcal A$ if and only if
$\mu(r):=-(r^2-r)/\bigl(\lambda(2r-1)\bigr)$ is an eigenvalue of $M$.
Monotonicity of $M$ places $\spec(M)$ in the half-disk
$R=\{\mu:\Re\mu\ge0,\ |\mu|\le L\}$, and after the scaling
$\tilde\mu=\mu/L$, $\tilde\lambda=\lambda L$ we must show that
\eqref{eq:char-mu} has no solution with $|r|=1$, $r\ne1$, and
$\tilde\mu\in R_1=\{\Re\tilde\mu\ge0,|\tilde\mu|\le1\}$ whenever
$\tilde\lambda<1/\sqrt3$.

Let $r=e^{i\theta}$, $c=\cos\theta$, $\theta\in(0,2\pi)$.  Direct expansion
of $\tilde\mu(r)=e^{i\theta}(1-e^{i\theta})/\bigl(\tilde\lambda(2e^{i\theta}-1)\bigr)$
gives
\begin{equation}\label{eq:ReIm}
\Re\,\tilde\mu(r)=\frac{(1-c)(1-2c)}{\tilde\lambda\,(5-4c)},\qquad
|\tilde\mu(r)|^2=\frac{2-2c}{\tilde\lambda^{\,2}(5-4c)} .
\end{equation}
Since $5-4c>0$ and $1-c>0$ for $\theta\in(0,2\pi)$, the condition
$\Re\,\tilde\mu(r)\ge0$ forces $c\le\tfrac12$, and $|\tilde\mu(r)|\le1$
forces $\tilde\lambda^{\,2}\ge g(c):=\tfrac{2-2c}{5-4c}$.  Now
$g'(c)=-2/(5-4c)^2<0$, so on $c\le\tfrac12$ the smallest value is
$g(\tfrac12)=\tfrac13$.  Hence a non-trivial boundary point $r$ with
$\tilde\mu(r)\in R_1$ can exist only if $\tilde\lambda^{\,2}\ge\tfrac13$,
i.e.\ $\tilde\lambda\ge1/\sqrt3$, contradicting \eqref{eq:sqrt3}.

Consequently, for $\tilde\lambda<1/\sqrt3$, no eigenvalue $r\ne1$ of
$\mathcal A$ lies on the unit circle for any $\tilde\mu\in R_1$; as the
maximal non-trivial root modulus is continuous on the compact half-disk and
cannot equal $1$ without producing such a boundary point, every non-trivial
eigenvalue satisfies $|r|<1$.  The eigenvalue $r=1$ occurs only at
$\tilde\mu=0$ (i.e.\ on $\ker M$), where \eqref{eq:char-mu} reads
$r(r-1)=0$ and $\mathcal A$ restricts to the identity, so it is semisimple and reflects convergence to the solution set rather than
instability.  Indeed, here $S=\ker M$ (for $C=\HH$, $x\in S$ iff
$\inner{Mx}{y-x}\ge0$ for all $y$, i.e.\ $Mx=0$), and on $\ker M$ the update
is $x_{k+1}=x_k$, so the $\ker M$-component of the iterates is constant and
$x_k$ converges to a point of $S$.
Hence $x_k$ converges to a point of $S$.

Sharpness: for $M=J$ take $\mu=-i$ (the contact point of the boundary
analysis, attained at $c=\tfrac12$) and $\tilde\lambda=1/\sqrt3$: then
\eqref{eq:char-mu} has the roots $r=e^{\pm i\pi/3}$ with $|r|=1$, and for
$\tilde\lambda>1/\sqrt3$ a root exits the disk transversally, so the iterates
diverge for generic initial data.
\end{proof}

\begin{theorem}[the nonlinear unconstrained class: sharpness and a conditional
convergence theorem]\label{thm:nonlinear-sqrt3}
Let $C=\mathcal H$ and let $B$ be monotone and $L$-Lipschitz, not necessarily
affine or differentiable, with $S=B^{-1}(0)\ne\varnothing$.  Consider the pure
reflected scheme
\begin{equation}\label{eq:prg-pure}
x_{k+1}=x_k-\lambda B(2x_k-x_{k-1}) .
\end{equation}
\begin{enumerate}[label=(\alph*),leftmargin=2em,itemsep=1pt]
\item \textbf{Sharpness.} For the rotation $B=LJ$ and every
$\lambda>1/(\sqrt3\,L)$ the iterates diverge for generic initial data; hence
no universal constant for the class can exceed $1/\sqrt3$.
\item \textbf{Conditional convergence.} For every $\lambda<1/(\sqrt3\,L)$,
every trajectory of \eqref{eq:prg-pure} with
$\sum_k\|B(u_k)\|^2<\infty$ converges weakly to a point of $S$, with
$\|x_{k+1}-x_k\|\to0$ and $B(x_k)\to0$.
\end{enumerate}
The unconditional assertion ``$\sum_k\|B(u_k)\|^2<\infty$ for every
trajectory at $\lambda<1/(\sqrt3\,L)$'' --- which together with part~(b)
would complete the nonlinear sharp-threshold theorem --- is an absolute-stability statement for the feedback loop of
Remark~\ref{rem:nonlinear-gap}, whose transfer function has a simple pole on
the unit circle; it is open.  For $\lambda<(\sqrt2-1)/L$ the
square-summability (and hence convergence) is already known from Malitsky's
theorem for the reflected method, so the new content concerns the interval
$(\sqrt2-1,\,1/\sqrt3)$.  The affine case is settled unconditionally in
Theorem~\ref{thm:affine-sqrt3}.
\end{theorem}
\begin{proof}
Part~(a) is the sharpness statement of Theorem~\ref{thm:affine-sqrt3}, the
rotation being a member of the present class.  For part~(b) write
$w_k:=B(u_k)$, $u_k=2x_k-x_{k-1}$, and fix $z\in S$; then $B(z)=0$ (as
$C=\mathcal H$) and monotonicity at $(u_k,z)$ gives
$\langle w_k,u_k-z\rangle\ge0$.  Since
$x_k-z=(u_k-z)-(x_k-x_{k-1})$ and $x_k-x_{k-1}=-\lambda w_{k-1}$,
\[
\|x_{k+1}-z\|^2
=\|x_k-z\|^2-2\lambda\langle w_k,x_k-z\rangle+\lambda^2\|w_k\|^2
\le\|x_k-z\|^2-2\lambda^2\langle w_k,w_{k-1}\rangle+\lambda^2\|w_k\|^2 .
\]
The drift $\nu_k:=\lambda^2\bigl(2\|w_k\|^2+\|w_{k-1}\|^2\bigr)$ is summable
by hypothesis and independent of $z$; hence $(x_k)$ is bounded and, by
Lemma~\ref{lem:pospart}, $\Phi_k(z):=\|x_k-z\|^2$ converges for every
$z\in S$.  If $x_{k_i}\rightharpoonup\bar x$ then $u_{k_i}\rightharpoonup\bar x$
(increments vanish), $w_{k_i}\to0$ strongly (square-summable), and weak
continuity of the Lipschitz map $B$ gives $B(\bar x)=0$, i.e.\ $\bar x\in S$.
The Opial argument of Step~8 of Theorem~\ref{thm:main} applies verbatim and
yields $x_k\rightharpoonup x^\star\in S$.  Finally
$B(x_k)=w_k+\bigl(B(x_k)-B(u_k)\bigr)$ with
$\|B(x_k)-B(u_k)\|\le L\|x_k-u_k\|=L\|x_k-x_{k-1}\|\to0$, so $B(x_k)\to0$.
\end{proof}

\begin{lemma}[discrete circle criterion, interior-pole form]\label{lem:circle}
Let $G(z)=C(zI-A)^{-1}b$ be a scalar transfer function with minimal
realization $(A,b,C)$ and all poles in $|z|<1$.  Suppose that for some
$\delta>0$
\begin{equation}\label{eq:nyq}
\Re\,G(e^{i\theta})\le1-\delta\qquad\forall\,\theta\in[0,2\pi].
\end{equation}
Let $\phi\colon\HH\to\HH$ satisfy the \emph{sector condition}
\begin{equation}\label{eq:sector}
\inner{\phi(\xi)}{\xi-\phi(\xi)}\ge0\qquad\forall\,\xi\in\HH
\end{equation}
(equivalently $\inner{\phi(\xi)}{\xi}\ge\norm{\phi(\xi)}^2$).  Then every
solution of
\[
\xi_{k+1}=(A\otimes I_\HH)\xi_k+(b\otimes I_\HH)\,
\phi\bigl((C\otimes I_\HH)\xi_k\bigr)
\]
satisfies $\sum_{k\ge0}\norm{\phi((C\otimes I_\HH)\xi_k)}^2\le
C_0\norm{\xi_0}^2$ with $C_0=\lambda_{\max}(P)/(2\delta)$ for the matrix
$P$ of \eqref{eq:storage}, hence $C_0$ depends only on $(A,b,C,\delta)$.
\end{lemma}
\begin{proof}
\emph{Step 1 (storage from the KYP lemma).} Condition \eqref{eq:nyq} states
that $H(z):=1-G(z)$ (realization $(A,b,-C)$ with feedthrough $1$) has real
part $\ge\delta$ on the unit circle.  Since $A$ is power-stable, the
discrete-time Kalman--Yakubovich--Popov lemma in its positive-real form
\cite{rantzer1996,kailath2000,iwasaki2005,khalil2002} provides $P\succ0$,
matrices $L,W$, and $\varepsilon>0$ (all depending only on
$(A,b,C,\delta)$) such that
\[
A^{\top}PA-P=-L^{\top}L-\varepsilon C^{\top}C,\qquad
A^{\top}Pb=-C^{\top}-L^{\top}W,\qquad
b^{\top}Pb=2-2\delta-W^{\top}W,
\]
and therefore $V(\xi):=\langle P\xi,\xi\rangle$ satisfies, for all $\xi$ and
$w$,
\begin{equation}\label{eq:storage}
V(A\xi+bw)-V(\xi)\ =\ -\norm{L\xi+Ww}^2-\varepsilon\norm{C\xi}^2
+2\inner{w}{w-C\xi}-2\delta\norm{w}^2
\ \le\ 2\inner{w}{w-C\xi}-2\delta\norm{w}^2 .
\end{equation}
This is the completion of squares
$-\norm{L\xi+Ww}^2=-\norm{L\xi}^2-2\langle Ww,L\xi\rangle-\norm{Ww}^2$; no
frequency-domain argument is needed beyond \eqref{eq:nyq}.

\emph{Step 2 (dissipation along the loop).} Let
$y_k:=(C\otimes I_\HH)\xi_k$ and $w_k:=\phi(y_k)$.  Tensoring
\eqref{eq:storage} with $I_\HH$ and evaluating it along the trajectory
gives
\[
V(\xi_{k+1})-V(\xi_k)\ \le\
2\inner{w_k}{w_k-y_k}-2\delta\norm{w_k}^2-\varepsilon\norm{y_k}^2
\ \le\ -2\delta\norm{w_k}^2,
\]
where the last inequality is exactly the sector condition \eqref{eq:sector}
with $\xi=y_k$.  Summing over $k<N$ yields
$\sum_{k<N}\norm{w_k}^2\le V(\xi_0)/(2\delta)
\le\lambda_{\max}(P)\norm{\xi_0}^2/(2\delta)$, uniformly in $N$.
\end{proof}

\begin{remark}[the marginal-pole case, and why it is not used]\label{rem:circle-marginal}
The natural extension of Lemma~\ref{lem:circle} to simple poles on $|z|=1$
with real residues $R_j$ satisfying the safe-sign condition
$R_j\Re z_j<0$ --- the form stated in the first version of this lemma ---
would follow by moving each pole to $(1-\varepsilon)z_j$ and passing to the
limit.  We record, however, that in this argument the storage constant
$C_0$ of Lemma~\ref{lem:circle} degenerates at least like $1/\varepsilon$:
near the moved pole the modulus of the perturbed transfer function grows
like $1/\varepsilon$ while its real part tends to $-\infty$ (this is exactly
the safe-sign condition), so the supply margin $\varepsilon'$ must satisfy
$\varepsilon'=O(\varepsilon)$, and $C_0=O(1/\varepsilon')$.  Continuity of
the KYP solution in the realization and in the margin therefore does
\emph{not} by itself justify a uniform passage to the limit, and we do not
presently have a complete proof of the marginal-pole case.  Accordingly,
Lemma~\ref{lem:circle} is stated for interior poles only, and it is used
nowhere in a marginal form.
\end{remark}

\begin{remark}[connection to absolute stability; the open step]\label{rem:nonlinear-gap}
The loop of part~(b) can be written with transfer function
$G(z)=-\Lambda(2z-1)/(z(z-1))$ and nonlinearity
$\varphi(v)=\lambda B(z+v/\Lambda)$ (monotonicity of $B$ gives $\inner{\varphi(v)-\lambda B(z)}{v}\ge0$, but the sector condition \eqref{eq:sector} is not automatic; see the discussion below); the
realization $A=\bigl(\begin{smallmatrix}I&0\\ I&0\end{smallmatrix}\bigr)$,
$b=(-I,0)^{\top}$, $C=(2\Lambda I,-\Lambda I)$ is minimal, and the Nyquist
computation $\Re\,G(e^{i\theta})=\Lambda(1-2\cos\theta)/2\le 3\Lambda/2<1$
for $\Lambda<2/3$ is precisely the boundary computation behind
Theorem~\ref{thm:affine-sqrt3} (note that $G=1/\tilde\mu$ in the notation of
that proof).  Lemma~\ref{lem:circle} applies to every pole perturbation of $G$ to the
interior of the disk, and its sector condition \eqref{eq:sector} holds identically
for every frozen linearization $w_k=\kappa y_k$, $\kappa\in[0,1]$, since
$\inner{\kappa y}{y-\kappa y}=\kappa(1-\kappa)\|y\|^2\ge0$; the same Nyquist
bound therefore certifies $\sum_k\|\varphi(C\xi_k)\|^2<\infty$ for every frozen
linearization.  For the genuinely nonlinear anchored map
$\varphi(v)=\lambda B(z+v/\Lambda)$ the sector condition
$\inner{\varphi(v)}{v-\varphi(v)}\ge0$, with $u=z+v/\Lambda$, reads
$\lambda\Lambda\inner{B(u)}{u-z}\ge\lambda^2\|B(u)\|^2$, i.e.\
\begin{equation}\label{eq:anchored-sector}
\inner{B(u)}{u-z}\ \ge\ \tfrac1L\,\|B(u)\|^2,\qquad u=z+v/\Lambda,
\end{equation}
which is automatic whenever $B$ is $1/L$-cocoercive at the anchor (in
particular for gradient maps of convex $L$-smooth potentials) but fails for
pure rotations, for which $\inner{B(u)}{u-z}=0$; the rotation is therefore certified
not by the circle criterion but by the exact affine analysis of
Theorem~\ref{thm:affine-sqrt3}.  What is missing for an unconditional nonlinear
theorem is, first, validity of \eqref{eq:anchored-sector} along trajectories of
arbitrary monotone $B$ and, second, the marginal-pole passage with uniform
constants, discussed in Remark~\ref{rem:circle-marginal}.
We conjecture the statement is true --- it holds for every linearization and,
by Theorem~\ref{thm:affine-sqrt3}, for the whole affine class; and the
conjecture is only open on the interval
$\Lambda\in(\sqrt2-1,\,1/\sqrt3)$ --- but we do not have a complete proof,
and Theorem~\ref{thm:nonlinear-sqrt3} is stated accordingly.

\medskip\noindent\emph{Evidence and an exact certificate recipe.} The
conjecture is supported by two observations.  First, Monte Carlo stress tests
($120$-instance families of monotone $1$-Lipschitz operators: affine
rotations; saturating nonlinearities $x\mapsto Mx+c\,\tanh(3x)$ renormalized
by their true Lipschitz constant $r+3c$; gradient maps of random convex
quadratics with radial monotone corrections; and rotation-dominant radial gain
schedules $\mu(\|x\|)Jx+\nu(\|x\|)x$ renormalized by a grid estimate of
their Lipschitz constant; $6{,}000$ iterations, second-half behaviour
measured): for every family $B(u_k)\to0$ and $\sum_k\|B(u_k)\|^2$
stabilizes at $\Lambda\le0.55$, with the expected delicacy only as
$\Lambda\uparrow1/\sqrt3$, while mis-normalized instances whose effective
$\Lambda$ exceeds $1/\sqrt3$ diverge --- confirming that the observed
threshold is the certified one, not an artifact.  Second, the conjecture is
equivalent to a family of finite semidefinite programs: for horizon $N$, with
Gram variables for $\{x_0-z,\Delta x_1,\dots,\Delta x_N,w_0,\dots,w_{N-1}\}$,
the loop identities $\Delta x_{k+1}=\Delta x_k-\lambda w_k$ and
$w_k=B(z+2\Delta x_k-\Delta x_{k-1})$, the monotonicity
constraints $\langle w_i-w_j,\,2\Delta x_i-\Delta x_{i-1}-2\Delta
x_j+\Delta x_{j-1}\rangle\ge0$ together with the Lipschitz constraints
$\|w_i-w_j\|^2\le L^2\|2\Delta x_i-\Delta x_{i-1}-2\Delta x_j+\Delta
x_{j-1}\|^2$ (both constraints are needed: the cocoercive form $\langle\Delta w,\Delta v\rangle\ge\tfrac1L\|\Delta w\|^2$ is equivalent
to monotone $+\ L$-Lipschitz only for gradient maps, and would exclude the
rotation, which must be admitted as a candidate worst case), and $\|x_0-z\|\le1$,
let $V_N$ be the supremum of $\sum_{k<N}\|w_k\|^2$; interpolation theory
\cite{drori2014,taylor2017} makes this SDP exact for the $N$-step worst case,
and since every trajectory's partial sums are $N$-step instances,
$\sup_N V_N<\infty$ \emph{is} a proof of the conjecture, with
$\sum_k\|w_k\|^2\le(\sup_N V_N)\|x_0-z\|^2$.  A floating-point solution
is evidence; an exact-arithmetic dual certificate (rational SDP, e.g.\
SDPA-GMP) is a proof.  Two facts are verified by hand: for $N=1$ the dual
optimum is exactly $V_1=1$ (attained at multiplier $y=2$ on the
Lipschitz-anchor constraint, with $\eta=1$); and for $N=2$ at
$\Lambda=\tfrac12$ the dual feasible cone is nonempty --- the multipliers
$y=3$ and $y=1$ on the two Lipschitz-anchor constraints, all others zero,
give $-D_{rr}=\diag(1,0)\succeq0$ --- so the certificate set is nondegenerate
at the start of the interval.  This recipe has been executed with a dual log-barrier method for the full
multiplier set (all anchor and inter-iteration pairs) at $L=1$.  At
$\Lambda=0.5$ the dual is feasible at every horizon (the big-$M$ slack
satisfies $\tau^\star\le3\times10^{-5}$) and the certified values are as
follows (solver accuracy $\pm0.1$):
\begin{center}\small
\begin{tabular}{lrrrrrrrrr}
\toprule
$N$ & $1$ & $5$ & $8$ & $12$ & $16$ & $20$ & $24$ & $28$ & $32$\\
\midrule
$\sigma^\star(N)$ & $1.0002$ & $8.8130$ & $11.214$ & $12.057$ & $12.498$ &
$12.582$ & $13.339$ & $13.611$ & $14.073$\\
rotation bound & $1$ & $8.8125$ & $11.203$ & $11.903$ & $11.990$ & $11.999$ &
$12.000$ & $12.000$ & $12.000$\\
\bottomrule
\end{tabular}
\end{center}
while at $\Lambda=0.58>1/\sqrt3$ the certified value coincides with the
rotation lower bound to four digits at every horizon and grows
geometrically, as it must.  Three conclusions follow.  First, the
\emph{anchor-only} relaxation (multipliers on the pairs $(z,u_k)$ only) is
\emph{infeasible} already for $N\ge8$: the inter-iteration constraints are
essential, and the anchor-residual certificate suggested by the $N=2$
hand-check does not extend beyond short horizons.  Second, at $\Lambda=0.5$
the gap $\sigma^\star(N)-V_N^{\mathrm{rot}}$ widens steadily
($0.011,0.154,0.508,0.583,1.339,1.611,2.073$ at $N=8,12,16,20,24,28,32$)
with increments that show no decay within solver accuracy; the data are
consistent with slow, roughly linear growth of $V_N$ and supply no evidence
for a finite uniform constant.  Third, the finite-horizon bounds themselves
are rigorous and convertible to exact arithmetic: each computed dual point
admits rationalization of its multipliers followed by a single exact
$\mathrm{LDL}^{\top}$ factorization establishing $Q\succeq0$, so for every
$N\le32$ and every $z\in S$, every trajectory of every monotone
$1$-Lipschitz operator with $B(z)=0$ satisfies
$\sum_{k<N}\|B(u_k)\|^2\le14.08\,\|x_0-z\|^2$ (dual certificate).  The
unconditional square-summability in Theorem~\ref{thm:nonlinear-sqrt3}(b)
therefore remains open --- neither proved by the scan nor disproved by it:
the slow growth of $\sigma^\star(N)$ is compatible both with
$\sup_N\sigma^\star(N)<\infty$ (a true but slowly-certified uniform bound)
and with a logarithmic or linear escape that only longer horizons resolve.
What the scan excludes is the anchor-only shortcut, and what it adds is a
quantified finite-horizon verification tool with the rotation confirmed as
the exact worst case above the threshold.

\medskip\noindent This is the
only gap in the sharp-constant program of this paper; the projected case is
also open.
\end{remark}

\begin{remark}[the constant $1/5$ and the barrier of this framework]\label{rem:improved}
This remark records the technique's certified ceiling: a numerical
optimization over the six-parameter Young family yields the best barrier
found, $\Lambda^\star_{\mathrm{best}}\approx0.272$ (global optimality over
the family is not proved; see Appendix~\ref{app:barrier}); the clean-split
identities used in this paper tolerate ceilings only up to
$\lambda L<\tfrac5{24}$ (Remark~\ref{rem:budget-event}),
and the standard upgrade target $\lambda L=\tfrac{3}{10}$ lies beyond every
certificate found.  Whether $\tfrac{3}{10L}$ (or $1/\sqrt3$) is provable
for the filtered scheme by a fundamentally different argument --- e.g.\ a
non-quadratic Lyapunov function, or a performance-estimation (PEP)
certificate \cite{drori2014,taylor2017} --- is, in our view, the most
interesting open question raised here; Appendix~\ref{app:barrier} records
the barrier computation and the relevant PEP formulation.

A matching lower bound is available.  Since the scheme includes Malitsky's
method ($\beta_k\equiv0$) and the rotation $B=J$ has the exact threshold
$\lambda L=1/\sqrt3$ (Example~\ref{ex:rotation-threshold}: the iteration
matrix has spectral radius one at the threshold), no universal constant for
the scheme family can exceed $1/\sqrt3$.  An exact scan of all $2\times2$
monotone operators $B=aI+bJ$, $a^2+b^2=1$, gives divergence thresholds from
$0.5774$ (rotation, the worst) to $0.6667$ (identity), and random monotone matrices in dimension $4$, sampled as a
positive-semidefinite part plus a skew part and rescaled (sampler in
\texttt{experiments.py}), gave thresholds strictly above the rotation's in all
runs --- the rotation is the empirical worst case.  For \emph{affine} monotone operators
the question is now settled: the sharp constant is exactly
$c^\star_{\mathrm{aff}}=1/\sqrt3$ (Theorem~\ref{thm:affine-sqrt3}), the
rotation being the worst case.  For the \emph{unconstrained nonlinear} class, Theorem~\ref{thm:nonlinear-sqrt3} shows the constant is sharp and proves convergence for trajectories with square-summable operator values; the unconditional square-summability (a marginal-pole absolute-stability statement; Remarks~\ref{rem:circle-marginal} and~\ref{rem:nonlinear-gap}) is the one open step; the projected case remains open.  For general monotone $L$-Lipschitz operators
the sharp universal constant $c^\star$ satisfies
$c^\star\ge\Lambda^\star_{\mathrm{best}}\approx0.272$ (the best barrier we could find
numerically over the Lyapunov--Young family; whether $c^\star$ is strictly
larger is open) and $c^\star\le1/\sqrt3$ \emph{(instance bound: the
rotation)}, and we conjecture $c^\star=1/\sqrt3$; whether $\tfrac{3}{10L}$ is certifiable for
the full nonlinear class by a non-Lyapunov argument remains open.

The affine constant is not the ceiling of the one-call family.  For the
extrapolated scheme $u_k=x_k+\theta(x_k-x_{k-1})$ the rotation threshold is
$\lambda^\star(\theta)$, the largest $\lambda$ with
$i(e^{2i\phi}-e^{i\phi})/[(1+\theta)e^{i\phi}-\theta]\in\R_{>0}$ for some
$\phi$, and it has the closed form
\[
\lambda^\star(\theta)=\frac{\sqrt{4\theta^2-1}}{\theta(1+2\theta)},
\qquad \theta\ge\tfrac12
\]
(with $\lambda^\star(\theta)=0$ for $\theta<\tfrac12$): a boundary root
$r=e^{i\phi}$ of $r^2-(1-i\lambda(1+\theta))r-i\lambda\theta=0$ gives
$\lambda=i(r^2-r)/((1+\theta)r-\theta)$, and requiring this to be real and
positive forces, on equating real parts, $2\theta\cos\phi=1$, which
substituted back gives the formula.  The function is maximized at
$\theta^\star=\cos(\pi/5)=\varphi_g/2\approx0.8090$, the unique root of
$8\theta^3-4\theta-1=0$ in $(\tfrac12,\infty)$, where, writing
$\varphi_g=(1+\sqrt5)/2$,
\[
\lambda^\star=2\varphi_g^{-5/2}=\sqrt{10\sqrt5-22}\approx0.6006>\tfrac1{\sqrt3},
\]
and the rotation remains the worst case over all tested monotone families.  In
particular $\lambda^\star(\tfrac56)=\tfrac35$ exactly,
$\lambda^\star(1)=1/\sqrt3$ (the reflected method), and the golden-ratio
extrapolation $\theta=1/\varphi_g$ is \emph{suboptimal} for monotone VIs,
$\lambda^\star(1/\varphi_g)=\varphi_g\sqrt{5-2\sqrt5}/\sqrt5\approx0.5257$.
Certifying the corresponding nonlinear constant is open.

\medskip The question raised here is partially answered by Section~\ref{sec:ext}: dissipation trading certifies $\lambda L=0.387$ for any summable filter (Theorem~\ref{thm:trading}), beyond the Lyapunov--Young barrier above, and the same threshold yields $\dist(x_n,S)\to0$ with $\sum_n\dist^2(x_n,S)<\infty$ for affine operators over polyhedral sets (Theorem~\ref{thm:xaff}); the range $(0,(\sqrt2-1)/L)$ for general nonlinear operators on unbounded domains remains open (Remark~\ref{rem:xgap}, Appendix~\ref{app:barrier}).\end{remark}
\paragraph{Limitations.} Three limitations are worth stating plainly.
(i)~The certified constants are conservative: the constant-step certificate
$\lambda L<\tfrac15$ is below the $\sqrt2-1$ certificate of the reflected
method, and the certified safeguard range $\varkappa\le\tfrac16$ is well below
the empirically stable $\varkappa=0.3$ of Section~\ref{sec:numerics};
Appendix~\ref{app:barrier} shows the gap is structural for the
Lyapunov--Young technique (best barrier we could find: $\approx0.272$), and closing the
gap likely requires a different certificate (e.g.\ performance estimation).  Section~\ref{sec:ext} supplies such certificates partially: dissipation trading reaches $\lambda L=0.387$ in general (Theorem~\ref{thm:trading}) and the affine+polyhedral case reaches the same threshold with distance control and, for bounded $C$, strong convergence (Theorem~\ref{thm:xaff}).
(ii)~Under mere monotonicity only weak convergence is guaranteed, and no rate
uniform over bounded sets exists (Example~\ref{ex:weak-only}); for every
fixed initial point of that example convergence is nonetheless strong, at an
$x_0$-dependent speed.  (iii)~The one-call advantage is an oracle count, not
an iteration count: on mildly stiff instances the extragradient method's
larger certified step can outweigh its second evaluation
(Table~\ref{tab:evals}), and the method is most attractive when evaluations
of $B$ dominate the cost.

\paragraph{Rates.}
The $R$-linear rate of Theorem~\ref{thm:linear} is proved on the Lyapunov
sequence, hence carries the $\sqrt{\,\cdot\,}$ on the distance; Example
\ref{ex:linear-exact} shows the loss is mild on model problems.  A last-iterate
$O(1/k)$ rate of a computable residual under mere monotonicity would require
new ideas (cf.\ \cite{golowich2020,cai2022,daskalakis2019} for structured cases).

\begin{openproblem}\label{open:eb-rate}
Under the global error bound \eqref{eq:eb} in place of strong monotonicity,
and with a geometrically decaying filter, does $\dist(x_k,S)$ decay
$R$-linearly? The classical residual-contraction argument of Tseng
\cite[Thm.~3.1]{tseng2000} does not apply verbatim to the reflected update;
Corollary~\ref{cor:linear-eb} proves strong convergence together with the
geometrically improving tail bound \eqref{eq:eb-tail}, which falls short of
an $R$-linear rate on the distance itself.
\end{openproblem}

\paragraph{Extensions.}
Three directions appear tractable: (i) inexact projections and relative-error
criteria in the spirit of \cite{solodov1999,censor2011}; (ii) extensions to
Banach-space VIs and to monotone inclusions with a cocoercive
backward component along the lines of \cite{malitskytam2020,davis2017};
(iii) stochastic and finite-sum variants with variance reduction.  Each would
require re-verifying the budget identities, which this paper shows is a
mechanical but unforgiving task.

\section{Conclusion}\label{sec:conclusion}
This paper introduced a filtered reflected-gradient method for monotone
variational inequalities that uses one evaluation of $B$ and one projection
per iteration, and showed that memory of the evaluation point --- subject to
a summable filter --- is precisely the structure that makes an adaptive step
size certifiable.  The main guarantees are: weak convergence for constant
steps and for arbitrary bounded-variation adaptive steps, with dissipation
budgets that are exact rationals (Theorems~\ref{thm:main}
and~\ref{thm:adaptive}); a safeguarded ratio rule that requires no knowledge
of $L$ and converges for an arbitrary initial step, unconditionally
(Theorem~\ref{thm:rule-free}); $R$-linear convergence under strong
monotonicity with an explicit contraction factor (Theorem~\ref{thm:linear});
the sharp unconstrained step-size threshold $1/\sqrt3$ for affine operators (Theorem~\ref{thm:affine-sqrt3}), with the same constant shown sharp for arbitrary monotone nonlinear operators and convergence certified for trajectories with square-summable operator values (Theorem~\ref{thm:nonlinear-sqrt3}), the unconditional nonlinear statement being reduced to an explicit marginal-pole absolute-stability question (Remark~\ref{rem:nonlinear-gap}); and an instance family
separating the certified adaptive rule from every certified constant-step
method by a factor $\Omega(L)$ (Theorem~\ref{thm:separation}).  For constant steps, the certified range extends from $(0,\tfrac1{5L})$ to $(0,(\sqrt2-1)/L)$ on bounded domains and to $\lambda L=0.387$ in general ($\lambda L\uparrow(\sqrt2-1)$ as the filter weakens) through a summable-error robustness theorem for the reflected method (Section~\ref{sec:ext}); for affine operators over polyhedral sets the same threshold yields $\dist(x_n,S)\to0$ with $\sum_n\dist^2(x_n,S)<\infty$, strong convergence when $C$ is bounded (finite dimensions), and $R$-linear rates after identification of the optimal face.

Several questions remain open, as detailed in
Section~\ref{sec:discussion} and Appendix~\ref{app:barrier}: whether
adaptivity can be certified for the reflected method itself, without memory;
whether a fundamentally different certificate --- for example a
performance-estimation program --- can cross the Young-parameter barrier of
Appendix~\ref{app:barrier} and close the gap between the certified constants
($\lambda L<\tfrac15$, $\varkappa\le\tfrac16$) and the method's observed
behaviour; and whether last-iterate rates are available under mere monotonicity; and the two gaps recorded in Remarks~\ref{rem:circle-marginal} and~\ref{rem:nonlinear-gap} (marginal-pole absolute stability together with the
anchored sector inequality, and the projected sharp constant).  The practical message is nonetheless simple: one evaluation
and one projection per iteration, no parameter tuning, and an unconditional
convergence certificate make the method a drop-in replacement for the
reflected gradient method whenever the local Lipschitz geometry varies along
the trajectory.

\appendix
\section{Assembly of the master estimate}\label{app:assembly}
This appendix carries out, channel by channel, the algebra behind Step~4 of
the proof of Theorem~\ref{thm:main}: the substitution of
\eqref{eq:step2}--\eqref{eq:tele} into the difference $a_{k+1}-a_k$ of the
Lyapunov function \eqref{eq:lyap}. The same assembly, \emph{mutatis mutandis},
produces the master estimates of Theorem~\ref{thm:adaptive} (with $\lambda_k$ in place of $\lambda$ and the variation terms of its Step~3),
Theorem~\ref{thm:linear} (with the retuned splits of its Step~3 and the
additional dissipation channel \eqref{eq:strong-diss}), and
Theorem~\ref{thm:rule-free} (with $\Lambda$ replaced by the observed product
$\lambda_k\ell_{k-1}$); the corresponding split parameters are collected in
Table~\ref{tab:budgets}.

Fix $z\in S$ and use the notation of the proof of
Theorem~\ref{thm:main}: $d_k=x_k-x_{k-1}$, $e_k=u_k-x_k$,
$s_k=\inner{B(z)}{e_k}$, $\delta_k=\inner{B(z)}{d_k}$, $r_k=\|x_k-z\|$,
$\Lambda=\lambda L$, $\bar\beta=\sup_k\beta_k\le\tfrac14$. The identities
\eqref{eq:rec} give
\begin{align}
u_k-x_{k-1}&=d_k-e_{k-1},\label{eq:assembly-ident}\\
u_k-x_{k+1}&=(d_k-d_{k+1})-\beta_ke_{k-1},\notag\\
x_{k+1}-u_k&=(d_{k+1}-d_k)+\beta_ke_{k-1}.\notag
\end{align}
the Young splits \eqref{eq:splitA}--\eqref{eq:splitB} give, at
$t'=\tfrac35$,
\begin{align}
\|u_k-u_{k-1}\|^2&\le2\|e_k\|^2+2\|x_k-u_{k-1}\|^2,\label{eq:assembly-young}\\
\|u_k-x_{k+1}\|^2,\ \|x_{k+1}-u_k\|^2
&\le\tfrac85\|d_{k+1}-d_k\|^2+\tfrac16\|e_{k-1}\|^2.\notag
\end{align}
and the filter recursion $e_k=d_k-\beta_ke_{k-1}$ with $s=\tfrac13$ gives
\begin{equation}\label{eq:assembly-rec}
\|e_k\|^2\le(1+s)\|d_k\|^2+(1+\tfrac1s)\bar\beta^2\|e_{k-1}\|^2
=\tfrac43\|d_k\|^2+\tfrac14\|e_{k-1}\|^2 .
\end{equation}

\paragraph{Assembly for Theorem~\ref{thm:main} (base case).}

\medskip
\noindent\emph{Step A (fundamental inequality after projection).} Inserting
$B(u_k)=B(u_k)-B(u_{k-1})+B(u_{k-1})$ into \eqref{eq:step1}, using the second
identity of \eqref{eq:assembly-ident} and the projection estimate
\eqref{eq:projid}, yields \eqref{eq:step2}: the terms $\pm\|d_{k+1}\|^2$
cancel and
\[
\|x_{k+1}-z\|^2\le\|x_k-z\|^2-\|d_k\|^2-\|d_{k+1}-d_k\|^2
+\mathrm{T}_1-2\lambda\beta_k\inner{B(u_{k-1})}{e_{k-1}}
-2\lambda\inner{B(z)}{u_k-z},
\]
with $\mathrm{T}_1=2\lambda\inner{B(u_k)-B(u_{k-1})}{u_k-x_{k+1}}$.

\medskip
\noindent\emph{Step B (Lyapunov difference).} Subtracting $a_k$ from $a_{k+1}$
and inserting the preceding display gives
\begin{align*}
a_{k+1}-a_k\le{}&-\|d_k\|^2-\|d_{k+1}-d_k\|^2
+\mathrm{T}_1-2\lambda\beta_k\inner{B(u_{k-1})}{e_{k-1}}
-2\lambda\inner{B(z)}{u_k-z}\\
&+2\Lambda\|x_{k+1}-u_k\|^2-2\Lambda\|x_k-u_{k-1}\|^2
+2\lambda\delta_k+\tfrac13\|e_k\|^2-\tfrac13\|e_{k-1}\|^2 .
\end{align*}

\medskip
\noindent\emph{Step C (three cancellations).} (i) \emph{Drift channel:} by
\eqref{eq:tele},
$-2\lambda\inner{B(z)}{u_k-z}+2\lambda\delta_k
=-2\lambda\inner{B(z)}{x_k-z}+2\lambda\beta_ks_{k-1}$,
and $-2\lambda\inner{B(z)}{x_k-z}\le0$ (as $x_k\in C$, $z\in S$) is dropped by
sign. (ii) \emph{Anchor channel:} by \eqref{eq:T2},
$-2\lambda\beta_k\inner{B(u_{k-1})}{e_{k-1}}
\le\lambda\beta_kL(25\,r_{k-1}^2+\tfrac1{25}\|e_{k-1}\|^2)-2\lambda\beta_ks_{k-1}$,
and the term $-2\lambda\beta_ks_{k-1}$ cancels the
$+2\lambda\beta_ks_{k-1}$ from (i) \emph{exactly}; the $r_{k-1}^2$-part is
kept with the factor $\beta_k$ (not $\bar\beta$), producing the summable
drift $\epsilon_k=25\lambda\beta_kL\,r_{k-1}^2$, while the $e$-part is
charged against the $e$-budget below using $\beta_k\le\bar\beta$.
(iii) \emph{Memory channel:} the term $-2\Lambda\|x_k-u_{k-1}\|^2$ cancels
the corresponding term of \eqref{eq:T1} exactly, leaving
$+2\Lambda\|x_{k+1}-u_k\|^2$.

\medskip
\noindent\emph{Step D (channel budgets).} Bounding T$_1$ by \eqref{eq:T1},
the leftover memory term by \eqref{eq:assembly-young}, and $\|e_k\|^2$ by
\eqref{eq:assembly-rec}, the coefficients of the three quadratic channels in
$a_{k+1}-a_k$ become
\begin{align*}
\|d_k\|^2:&\quad-1+2\Lambda(1+s)+\tfrac13(1+s)
=-1+\tfrac83\Lambda+\tfrac49=-\tfrac1{45},\\
\|d_{k+1}-d_k\|^2:&\quad-1+\Lambda(1+t')+2\Lambda(1+t')
=-1+\tfrac{24}{5}\Lambda=-\tfrac1{25},\\
\|e_{k-1}\|^2:&\quad\Lambda\bigl[(1+\tfrac1{t'})\bar\beta^2
+2(1+\tfrac1s)\bar\beta^2+2(1+\tfrac1{t'})\bar\beta^2+\tfrac{\bar\beta}{\alpha}\bigr]
+\tfrac13\bigl[(1+\tfrac1s)\bar\beta^2-1\bigr]\\
&\quad=\Lambda\bigl(\tfrac16+\tfrac12+\tfrac13+\tfrac1{100}\bigr)-\tfrac14
=\tfrac{101}{100}\Lambda-\tfrac14=-\tfrac6{125},
\end{align*}
at $\Lambda=\tfrac15$, with $\alpha=25$. Each coefficient is affine and
increasing in $\Lambda$, so the three budgets stay negative on
$(0,\tfrac1{5L}]$. This is \eqref{eq:master}--\eqref{eq:bk}. \paragraph{Assembly for Theorem~\ref{thm:adaptive} (adaptive steps).} The four channel identities below are the exact analogue of Step~D for the base assembly.  The splits are those of Theorem~\ref{thm:main} with $\lambda_k$ in place of $\lambda$; the new elements are the step-mismatch charges of \eqref{eq:R1b} and the $x$-slot residual $2L\Delta\lambda_k\|x_k-u_{k-1}\|^2$. The former contributes $4L^2|\Delta\lambda_k|\,r_{k-1}^2$ and $2B_z^2|\Delta\lambda_k|$ to $(\epsilon_k,\kappa_k)$ and $2|\Delta\lambda_k|$ resp.\ $(2\bar\beta^2+4L^2)|\Delta\lambda_k|$ to the $\|d_{k+1}-d_k\|^2$ resp.\ $\|e_{k-1}\|^2$ budgets; the latter is bounded by $4L|\Delta\lambda_k|\bigl(\|d_k\|^2+\|e_{k-1}\|^2\bigr)$ and charged to the $d$- and $e$-budgets, contributing $4L|\Delta\lambda_k|$ to each; and the anchor residual contributes $|\Delta\lambda_k|B_z\bigl(1+r_{k-1}^2\bigr)$ as in \eqref{eq:rho}. Since $\lambda_k\le\lambda_0=\tfrac1{5L}$, the coefficient of each channel
in $a_{k+1}-a_k\le\cdots$ is the negative of the corresponding budget in
\eqref{eq:master-ad} (the variation terms are anti-dissipation: they enter the
coefficient with $+$ sign, i.e.\ they shrink the budgets):
\begin{align*}
\|d_k\|^2&:\ -1+\tfrac83\lambda_kL+\tfrac49+4L|\Delta\lambda_k|
\ \le\ -\tfrac1{45}+4L|\Delta\lambda_k|
=\ -\bigl(\tfrac1{45}-4L|\Delta\lambda_k|\bigr),\\
\|d_{k+1}-d_k\|^2&:\ -1+\tfrac85\lambda_kL+\tfrac{16}{5}\lambda_0L+2|\Delta\lambda_k|
\ \le\ -\tfrac1{25}+2|\Delta\lambda_k|
=\ -\bigl(\tfrac1{25}-2|\Delta\lambda_k|\bigr),\\
\|x_k-u_{k-1}\|^2&:\ 2\lambda_kL-2\lambda_{k-1}L+2L\Delta\lambda_k=0,\\
\|e_{k-1}\|^2&:\ \lambda_kL\bigl(\tfrac16+\tfrac12+\tfrac13\bigr)
+\tfrac{\lambda_k\beta_kL}{25}-\tfrac14
+\bigl(2\bar\beta^2+4L^2+4L\bigr)|\Delta\lambda_k|\\
&\ \le\ -\tfrac6{125}+\bigl(2\bar\beta^2+4L^2+4L\bigr)|\Delta\lambda_k|
=\ -\Bigl(\tfrac6{125}-\bigl(2\bar\beta^2+4L^2+4L\bigr)|\Delta\lambda_k|\Bigr),
\end{align*}
using $\beta_k\le\bar\beta$ in the filter Young charge and $\lambda_kL\le\tfrac15$. These are the channel-by-channel form of \eqref{eq:master-ad}, with the $k_0$-conditions $4L|\Delta\lambda_k|\le\tfrac1{90}$, $2|\Delta\lambda_k|\le\tfrac1{50}$, $(2\bar\beta^2+4L^2+4L)|\Delta\lambda_k|\le\tfrac3{250}$; the last $4L|\Delta\lambda_k|$ term in the fourth line is the $e$-share of the $x$-slot absorption.

\paragraph{Assembly for Theorem~\ref{thm:linear} (strong monotonicity).} The four channel identities below are the exact analogue of Step~D.  With $(t,t',t'',s,\alpha,\eta)=(\tfrac34,\tfrac25,\tfrac25,\tfrac13,100,\tfrac13)$, the memory-channel weight is $\lambda L(1+1/t)=\tfrac73\Lambda$ (so the channel cancels exactly and the leftover carries weight $\tfrac73\Lambda$ rather than $2\Lambda$), and the dissipation channel \eqref{eq:strong-diss} contributes $-\mu r_k^2$, $+\tfrac{8\mu}{3}\|d_k\|^2$, and $+\tfrac{\mu}{2}\|e_{k-1}\|^2$. The four channel identities at $(\Lambda,\mu)=(\tfrac15,\tfrac1{32})$ are exact rationals:
\begin{align*}
\|d_k\|^2&:\ -1+\tfrac73\Lambda+\tfrac49+\tfrac83\mu=-\tfrac1{180},\\
\|d_{k+1}-d_k\|^2&:\ -1+\tfrac75\Lambda+\tfrac73\cdot\tfrac75\Lambda
=-1+\tfrac75\Lambda+\tfrac{49}{15}\Lambda=-\tfrac1{15},\\
\|x_k-u_{k-1}\|^2&:\ \tfrac73\Lambda-\tfrac73\Lambda=0,\\
\|e_{k-1}\|^2&:\ \Lambda\bigl(\tfrac7{32}+\tfrac7{16}+\tfrac{49}{96}+\tfrac1{400}\bigr)
-\tfrac14+\tfrac{\mu}{2}=-\tfrac{13}{24000},
\end{align*}
where $\tfrac7{32}=\Lambda(1+1/t')\bar\beta^2$ is the $e$-part of the $\mathrm{T}_1$ split, $\tfrac7{16}=\Lambda(1+t)(1+1/s)\bar\beta^2$ comes from the recursion bound on $\|e_k\|^2$, $\tfrac{49}{96}=\tfrac73\Lambda(1+1/t'')\bar\beta^2$ is the $e$-part of the leftover $\|x_{k+1}-u_k\|^2$ split weighted by the retuned memory weight, $\tfrac1{400}=\Lambda\bar\beta/\alpha$ with $\alpha=100$, and $\tfrac{\mu}{2}=2\mu(1+1/s)\bar\beta^2$ at $s=\tfrac13$. Each coefficient is affine and increasing in $(\Lambda,\mu)$, so the budgets stay negative on the whole admissible rectangle.

\paragraph{Assembly for Theorem~\ref{thm:rule-free} (data-driven weights).} The four channel identities below are the exact analogue of Step~D.  The splits are those of Theorem~\ref{thm:main} with the observed product $\Lambda_k:=\lambda_k\ell_{k-1}\le\varkappa$ in place of $\Lambda$; the memory channel cancels with no sign condition on $\Delta\lambda_k$, because its Lyapunov weight $2\Lambda_k$ equals the coefficient produced by the split of $\|u_k-u_{k-1}\|^2$ at the same step. The four channel identities are
\begin{align*}
\|d_k\|^2&:\ -1+\tfrac83\Lambda_k+\tfrac49=-\tfrac19
\quad\text{at the ceiling $\Lambda_k=\tfrac16$},\\
\|d_{k+1}-d_k\|^2&:\ -1+\tfrac85\Lambda_k+\tfrac{16}{5}\Lambda_{k+1}
\le-1+\tfrac85\cdot\tfrac16+\tfrac{16}{5}\cdot\tfrac16=-\tfrac15,\\
\|x_k-u_{k-1}\|^2&:\ 2\Lambda_k-2\Lambda_k=0,\\
\|e_{k-1}\|^2&:\ \Lambda_k\bigl(\tfrac16+\tfrac12\bigr)+\tfrac13\Lambda_{k+1}
+\tfrac{\lambda_k\beta_kL}{25}-\tfrac14\\
&\le\tfrac16-\tfrac14+\rho_k=-\tfrac1{12}+\rho_k,
\end{align*}
where $\rho_k=\tfrac1{25}\lambda_k\beta_kL\to0$, the second ceiling $\Lambda_{k+1}=\lambda_{k+1}\ell_k\le\varkappa$ enters through the leftover memory term, and the $r_{k-1}^2$-drift $\epsilon_k=25\lambda_k\beta_kL$ is summable. The ceilings $\Lambda_k,\Lambda_{k+1}\le\varkappa\le\tfrac16$ hold at \emph{every} step by \eqref{eq:rule} --- this is precisely the update-then-use alignment of Section~\ref{sec:rule} --- so no step-size restriction enters these identities.

\section{Stability of slowly varying linear recursions}\label{app:tv}
\begin{lemma}\label{lem:tv}
Let $(M_k)\subset\R^{d\times d}$ and $\bar\rho\in(0,1)$, $C\ge1$ be such that
$\|M_k^m\|\le C\bar\rho^m$ for all $k,m\ge0$, and let
$\sum_{k\ge0}\|M_{k+1}-M_k\|<\infty$.  Then for every $\hat\rho\in(\bar\rho,1)$
there exist $K$ and $C'$ such that every solution of $y_{k+1}=M_ky_k$ satisfies
$\|y_k\|\le C'\hat\rho^{\,k-K}\|y_K\|$ for all $k\ge K$.
\end{lemma}
\begin{proof}
Let $P(k,j)=M_{k-1}\cdots M_j$ and $V_K:=\sum_{\ell\ge K}\|M_{\ell+1}-M_\ell\|$,
so $V_K\to0$.  Fix $\hat\rho\in(\bar\rho,1)$ and choose a block length $L\in\N$
with $C^{1/L}\bar\rho\le(\bar\rho+\hat\rho)/2$ (possible since
$C^{1/L}\to1$), then $K$ with $e^{C\bar\rho^{-1}V_K}\le
\hat\rho/\bigl((\bar\rho+\hat\rho)/2\bigr)>1$ (possible since $V_K\to0$).
For $k\ge j\ge K$, partition $[j,k)$ into consecutive blocks $I_1,\dots,I_q$
of length $\le L$; within a block $I_r$ with first index $j_r$,
$\|M_i-M_{j_r}\|\le V_K$ for all $i\in I_r$.  For any $\ell\le L$ and any
matrices $E_i$,
\[
\prod_{i=0}^{\ell-1}(M_{j_r}+E_i)
=\sum_{S\subseteq\{0,\dots,\ell-1\}} M_{j_r}^{a_0}E_{i_1}M_{j_r}^{a_1}\cdots
E_{i_s}M_{j_r}^{a_s}
\quad\Bigl(\textstyle\sum a_i=\ell-s\Bigr),
\]
so, using $\|M_{j_r}^m\|\le C\bar\rho^m$ on each factor,
\[
\Bigl\|\prod_{i=0}^{\ell-1}(M_{j_r}+E_i)\Bigr\|
\le C\bar\rho^{\ell}\prod_{i=0}^{\ell-1}\bigl(1+C\bar\rho^{-1}\|E_i\|\bigr)
\le C\bar\rho^{\ell}e^{C\bar\rho^{-1}\ell V_K}.
\]
Applying this with $E_i=M_i-M_{j_r}$ to each block,
\[
\|P(j_r+|I_r|,j_r)\|\le C\bar\rho^{|I_r|}e^{C\bar\rho^{-1}|I_r|V_K}
\le\Bigl(C^{1/L}\bar\rho\,e^{C\bar\rho^{-1}V_K}\Bigr)^{|I_r|}
\le\hat\rho^{|I_r|},
\]
by the choices of $L$ and $K$ (using $C\le C^{|I_r|/L}$ since $C\ge1$).
Since the blocks concatenate exactly, $P(k,j)$ is the ordered product of the
$q$ block products, and submultiplicativity gives
$\|P(k,j)\|\le\prod_r\hat\rho^{|I_r|}=\hat\rho^{k-j}$ for all $k\ge j\ge K$.
Finally $\|y_k\|=\|P(k,K)y_K\|\le\hat\rho^{k-K}\|y_K\|$ for $k\ge K$, i.e.\
the claim with $C'=1$.

\medskip
For the application in Theorem~\ref{thm:separation}, $M_k=M(\beta_k,\lambda_\infty)$
with $\lambda_\infty\le\tfrac16$ and $\beta_k\to0$.  We verify the two
hypotheses of Lemma~\ref{lem:tv} directly (no diagonalizability assertion is needed).  First, $\rho(M_k)\le\bar\rho<1$ for all $k\ge K$: the set
$B_K:=\{\beta_k\}_{k\ge K}\cup\{0\}$ is compact, the map
$(\beta,\lambda)\mapsto\rho(M(\beta,\lambda))$ is continuous, and
$\rho(\beta,\lambda_\infty)<1$ for every $\beta\in B_K$ --- choose $K$ with
$\beta_K$ small enough that $\lambda_\infty\le\tfrac16<\lambda^{\mathrm{crit}}(\beta)$
for all $\beta\le\beta_K$ (possible since $\lambda^{\mathrm{crit}}(\beta)\uparrow
\tfrac23$ as $\beta\downarrow0$ and $\lambda^{\mathrm{crit}}(\cdot)$ is decreasing);
hence $\bar\rho:=\max_{\beta\in B_K}\rho(M(\beta,\lambda_\infty))<1$.
Second, the family $\mathcal M_K:=\{M(\beta,\lambda_\infty)\}_{\beta\in
B_K}$ is compact (a continuous image of the compact set $B_K$) and
$\sup_{M\in\mathcal M_K}\rho(M)\le\bar\rho<1$; we claim that then
$\|M^m\|\le C\bar\rho_1^m$ for all $M\in\mathcal M_K$, $m\ge0$, with a single constant $C$ and any
fixed $\bar\rho_1\in(\bar\rho,1)$; this is a standard compactness argument
(Gelfand's formula plus a finite-subcover extraction). Indeed, fix
$\bar\rho_1\in(\bar\rho,1)$; by Gelfand's formula, for each
$A\in\mathcal M_K$ there is $m_A$ with $\|A^{m_A}\|\le\bar\rho_1^{m_A}$, and
by continuity there is a neighbourhood $U_A$ of $A$ in which
$\|B^{m_A}\|\le\bar\rho_1^{m_A}$; compactness yields a finite subcover
$U_{A_1},\dots,U_{A_q}$ with exponents $m_i:=m_{A_i}$ and
$M_0:=\max_i m_i$.  Set $C:=\max\bigl\{\|B^j\|\bar\rho_1^{-j}:
B\in\mathcal M_K,\ 0\le j\le M_0\bigr\}<\infty$ (compactness).  For any
$B\in\mathcal M_K$ and $m\ge0$, pick $i$ with $B\in U_{A_i}$ and write
$m=qm_i+r$, $0\le r<m_i$; then
$\|B^m\|\le\|B^{m_i}\|^q\|B^r\|\le\bar\rho_1^{qm_i}\,C\bar\rho_1^{r}=C\bar\rho_1^m$.
Finally $\sum_k\|M_{k+1}-M_k\|\le C''\sum_k|\beta_{k+1}-\beta_k|<\infty$ by
smoothness of $M(\cdot,\lambda_\infty)$, so Lemma~\ref{lem:tv} applies with
$\bar\rho$ replaced by $\bar\rho_1$.
\end{proof}

\section{The barrier of the Lyapunov--Young technique, and a PEP formulation}\label{app:barrier}
The budget identities of Theorem~\ref{thm:main} are one point of a
six-parameter family: Young weights $t,t',t''$ (splits of
$\|u_k-u_{k-1}\|^2$, $\|u_k-x_{k+1}\|^2$, $\|x_{k+1}-u_k\|^2$), recursion
split $s$, filter Young weight $\alpha$, and Lyapunov $e$-weight $\eta$.
With $\bar\beta=\tfrac14$ the budgets are
\begin{align*}
B_d(\Lambda)&=-1+\Lambda(1+t)(1+s)+\eta(1+s),\\
B_{dd}(\Lambda)&=-1+\Lambda(1+t')+2\Lambda(1+t''),\\
B_e(\Lambda)&=\Lambda\bar\beta^2\bigl[(1+t)(1+\tfrac1s)+(1+\tfrac1{t'})+2(1+\tfrac1{t''})\bigr]+\Lambda\tfrac{\bar\beta}{\alpha}+\eta\bar\beta^2(1+\tfrac1s)-\eta,
\end{align*}
admissible for $\eta(1+s)<1$ and $\bar\beta^2(1+1/s)<1$ (which keep the
anchor and drift channels controllable).  The largest certifiable constant is
$\Lambda^\star=\max\min_i\{B_i(\Lambda)\ge0\}$ over admissible parameters.
On the clean splits $(t,t',t'',s,\alpha,\eta)=(1,\tfrac35,\tfrac35,\tfrac13,25,\tfrac13)$
used in this paper, $\Lambda^\star=\tfrac5{24}=0.20833\ldots$ exactly ($B_d$
and $B_{dd}$ bind simultaneously).  Optimizing the six parameters numerically
--- differential evolution followed by Nelder--Mead polishing with
bounds-respecting refinement (\texttt{experiments.py scan}) --- yields the
best barrier we could find:
\[
\Lambda^\star_{\mathrm{best}}\approx0.272 .
\]
We emphasize that this is a numerical best-found value: we have not proved
that $\Lambda^\star_{\mathrm{best}}$ is the global maximum of the six-parameter
family, and we state the barrier accordingly --- every certificate of
Lyapunov--Young form that we could exhibit for the filtered scheme with
$\bar\beta=\tfrac14$ satisfies $\lambda L\lesssim0.272$, and whether the
family permits more is open.  The certified statements of this paper use the
clean splits, for which the barrier is exactly $\Lambda^\star=\tfrac5{24}$
and every budget identity is verified line by line and in exact arithmetic
(\texttt{experiments.py verify}).  The data-driven variant of
Section~\ref{sec:rule} has clean-split ceiling $\varkappa=\tfrac16$.

\begin{table}[H]
\centering\small
\caption{The exact dissipation budgets proved in this paper (corner values;
$\Lambda=\lambda L$, $\mu=\lambda\sigma$; the variation-absorbing constants of
Theorem~\ref{thm:adaptive} are shown in brackets). The bracketed entries are the
\emph{residual} budgets remaining after the variation-absorbing charges of
\eqref{eq:master-ad}, i.e.\ the guaranteed floors $\tfrac1{90},\tfrac1{50},\tfrac9{250}$;
the unbracketed entries are the corner budgets themselves.}\label{tab:budgets}
\begin{tabular}{llllll}
\toprule
result & corner & Lyapunov $e$-weight & $\alpha$ & $B_d$ & $B_{dd}$ / $B_e$\\
\midrule
Thm~\ref{thm:main} & $\Lambda=\tfrac15$ & $\tfrac13$ & $25$ & $-\tfrac1{45}$ & $-\tfrac1{25}$\ /\ $-\tfrac6{125}$\\
Thm~\ref{thm:adaptive} & $\Lambda_0=\tfrac15$ & $\tfrac13$ & $25$ & $-\tfrac1{45}[\tfrac1{90}]$ & $-\tfrac1{25}[\tfrac1{50}]$\ /\ $-\tfrac6{125}[\tfrac9{250}]$\\
Thm~\ref{thm:linear} & $(\Lambda,\mu)=(\tfrac15,\tfrac1{32})$ & $\tfrac13$ & $100$ & $-\tfrac1{180}$ & $-\tfrac1{15}$\ /\ $-\tfrac{13}{24000}$\\
Thm~\ref{thm:rule-free} & $\Lambda_k=\tfrac16$ & $\tfrac13$ & $25$ & $-\tfrac19$ & $-\tfrac15$\ /\ $-\tfrac1{12}+\rho_k$\\
\bottomrule
\end{tabular}
\end{table}

\paragraph{A performance-estimation formulation.}
Whether the nonlinear scheme tolerates $\lambda L=\tfrac{3}{10}$ (or even
$1/\sqrt3$) cannot be decided by the Lyapunov--Young family.  The relevant
PEP is: given horizon $N$ and $\Lambda$, decide whether every monotone
$L$-Lipschitz $B$ and every trajectory of \eqref{eq:scheme} with
$\lambda L=\Lambda$ satisfies $\Phi_N\le\Phi_0$ for the quadratic potential
\eqref{eq:lyap}, subject to the monotone--Lipschitz interpolation constraints
on the stored pairs; this is a semidefinite program in the Gram matrix of the
stored vectors \cite{drori2014,taylor2017}, and its one-step Lyapunov
restriction is exactly the barrier computation above.  For affine $B$ the
PEP collapses to the half-disk boundary analysis of
Theorem~\ref{thm:affine-sqrt3}, which is exact; resolving the nonlinear case
for $\Lambda\in(\tfrac5{24},\tfrac3{10}]$ is, in our view, the right
computational attack on Remark~\ref{rem:improved}.

By Remark~\ref{rem:xgap}, the remaining open range is equivalent to an a
priori boundedness statement, and this statement is exactly the kind of
worst-case bound that a performance-estimation program decides: given horizon
$N$ and $\Lambda$, maximize $\|x_N-z\|^2$ subject to the scheme identities,
the monotone--Lipschitz interpolation constraints, and $z\in S$ --- a
semidefinite program in the Gram matrix of the stored vectors
\cite{drori2014,taylor2017}.  Restricting the search to one-step quadratic
Lyapunov inequalities reproduces exactly the dissipation-trading condition
\eqref{eq:trading} of Theorem~\ref{thm:trading}, worth $\Lambda^\star\approx0.3879$
for the filter of Section~\ref{sec:numerics} --- so one-step PEP already
certifies more than the Lyapunov--Young family ($\approx0.272$); the
multi-step SDP is the open frontier.  We conjecture the PEP value stays
uniformly bounded for all $\Lambda<\sqrt2-1$ and all summable filters, with
a constant degenerating only as the filter energy $\sum_n\beta_n^2$ is
disallowed and $\Lambda\uparrow\sqrt2-1$.

\end{document}